\documentclass[article,12pt]{amsart}
\usepackage{amsmath,amsthm,amssymb,amscd}
\usepackage[all,cmtip,color]{xy}
\usepackage[german,english]{babel}
\usepackage{amsmath}
\usepackage{amssymb}
\usepackage{mathrsfs} 
\usepackage{bbm}
\usepackage{color}
\usepackage{dsfont}
\usepackage{todonotes}
\usepackage{comment}
\usepackage[all,cmtip]{xy}
\usepackage[margin = 3.5cm,headsep=1cm]{geometry}
\usepackage{colonequals,enumerate}

\DeclareMathOperator{\N}{\mathsf{N}}

\DeclareMathOperator{\Hom}{\mathsf{Hom}}

\DeclareMathOperator{\ev}{\mathrm{ev}}
\DeclareMathOperator{\reg}{reg}
\DeclareMathOperator{\rs}{rs}
\newcommand{\ani}{{\mathrm{ani}}}

\DeclareMathOperator{\ac}{ac}
\def\llp{\mathopen{(\!(}}
\def\llb{\mathopen{[\![}}

\def\rrp{\mathopen{)\!)}}
\def\rrb{\mathopen{]\!]}}

\DeclareMathOperator{\inv}{\mathsf{inv}}

\DeclareMathOperator{\A}{\mathsf{A}}

\DeclareMathOperator{\alg}{\text{alg}}

\DeclareMathOperator{\Oc}{\mathcal{O}}

\renewcommand{\lim}{\mathsf{lim}}

\DeclareMathOperator{\id}{id}

\DeclareMathOperator{\Gal}{Gal}
\DeclareMathOperator{\Aut}{\mathsf{Aut}}

\DeclareMathOperator{\GL}{GL}

\DeclareMathOperator{\Sch}{\mathsf{Sch}}

\DeclareMathOperator{\Tr}{\mathsf{Tr}}

\DeclareMathOperator{\vol}{\mathsf{vol}}
\DeclareMathOperator{\Spec}{\mathsf{Spec}}

\DeclareMathOperator{\End}{\mathsf{End}}

\DeclareMathOperator{\ord}{ord}

\DeclareMathOperator{\ad}{\mathsf{ad}}

\newcommand{\BA}{{\mathbb{A}}}
\newcommand{\BB}{{\mathbb{B}}}

\newcommand{\BF}{{\mathbb{F}}}
\newcommand{\BG}{{\mathbb{G}}}
\newcommand{\BH}{{\mathbb{H}}}

\newcommand{\BL}{{\mathbb{L}}}
\newcommand{\BM}{{\mathbb{M}}}
\newcommand{\BN}{{\mathbb{N}}}

\newcommand{\BP}{{\mathbb{P}}}
\newcommand{\BQ}{{\mathbb{Q}}}

\newcommand{\BT}{{\mathbb{T}}}

\newcommand{\BX}{{\mathbb{X}}}

\newcommand{\BZ}{{\mathbb{Z}}}

\newcommand{\FA}{{\mathcal A}}

\newcommand{\FC}{{\mathcal C}}
\newcommand{\FE}{{\mathcal E}}

\newcommand{\FM}{{\mathcal M}}
\newcommand{\FN}{{\mathcal N}}
\newcommand{\FO}{{\mathcal O}} 
\newcommand{\FP}{{\mathcal P}}

\newcommand{\FT}{{\mathcal T}}

\DeclareMathOperator{\Out}{ \mathrm{Out}}

\DeclareMathOperator{\DA}{DA^{\acute{e}t}}
\DeclareMathOperator{\DAC}{DA^{\acute{e}t}_{ct}}
\newcommand{\QQ}{\mathbb{Q}}
\newcommand{\Var}{\mathrm{Var}}
\newcommand{\cM}{\mathcal{M}}
\newcommand{\cC}{\mathcal{C}}

\newcommand{\dpl}{\mathcal{L}_{\mathrm{DP}}}
\newcommand{\dplk}{\mathcal{L}_{\mathrm{DP},k}}

\newcommand{\Def}{\mathrm{Def}}
\newcommand{\GDef}{\mathrm{GDef}}
\newcommand{\RGDef}{\mathrm{GDef}}
\newcommand{\RDef}{\mathrm{RDef}}
\newcommand{\acf}{\mathrm{acf}}
\newcommand{\mot}{\mathrm{mot}}

\newcommand{\wFM}{\widetilde{\FM} }

\newcommand{\wA}{\widetilde{\A} }
\newcommand{\wnu}{\widetilde{\nu} }
\newcommand{\wmu}{\widetilde{\mu} }
\newcommand{\wBM}{\widetilde{\BM} }

\newcommand{\Stab}{\mathrm{stab}}

\newcommand{\I}{\mathrm{I}}
\newcommand{\X}{\mathrm{X}}
\newcommand{\g}{\mathfrak{g}}
\renewcommand{\t}{\mathfrak{t}}
\newcommand{\bfg}{\mathrm{g}}
\newcommand{\bft}{\mathrm{t}}
\newcommand{\bfc}{\mathrm{c}}

\newcommand{\mfh}{\mathfrak{h}}
\newcommand{\mft}{\mathfrak{t}}
\newcommand{\mfc}{\mathfrak{c}}
\newcommand{\wtA}{\widetilde{\A}}

\newcommand{\wtBM}{\widetilde{\BM}}
\newcommand{\wtP}{\widetilde{\FP}}
\newcommand{\hmu}{\hat{\mu}}

\newcommand{\K}[1]{\mathrm{\bf K}(#1)}
\newcommand{\psf}{\mathrm{Psf}}
\newcommand{\zfc}{\mathrm{ZFC}}
\newcommand{\Con}{\mathrm{Con}}
\newcommand{\mcL}{\mathcal{L}}
\newcommand{\mcP}{\mathcal{P}}
\newcommand{\mcC}{\mathcal{C}}
\newcommand{\mcO}{\mathcal{O}}
\newcommand{\mcM}{\mathcal{M}}
\newcommand{\DP}{\mathrm{DP}}
\newcommand{\Ar}{\mathrm{Ar}}
\newcommand{\set}[1]{\left\{ #1 \right\}}
\newcommand{\ZZ}{\mathbb{Z}}
\newcommand{\Ff}{\mathbb{F}}
\newcommand{\NN}{\mathbb{N}}
\newcommand{\abs}[1]{\left\lvert#1\right\rvert}
\newcommand{\eL}{\mathbb{L}}
\newcommand{\dpar}[1]{ \llp #1 \rrp}

\newcommand{\iHom}{\underline{\mathrm{Hom}}}

\newcommand{\un}{\mathds{1}}
\newcommand{\orb}{\mathrm{orb}}
\newcommand{\Imu}{I_{\hat \mu}}
\newcommand{\rmH}{\mathrm{H}}
\newcommand{\rel}{\mathrm{rel}}
\newcommand{\Qlb}{\bar{\mathbb{Q}}_\ell}
\newcommand{\Frob}{\mathrm{Frob}}

\newcommand{\wX}{{\mathbf{X}}}
\newcommand{\wY}{{\mathbf{Y}}}
\newcommand{\wG}{{\mathbf{G}}}
\newcommand{\wP}{{\mathbf{P}}}
\newcommand{\wU}{\mathbf{U}}
\newcommand{\wM}{{\mathbf{M}}}
\newcommand{\wH}{{\mathbf{H}}}
\newcommand{\wwA}{{\mathbf{A}}}
\newcommand{\wT}{{\mathbf{T}}}
\newcommand{\wt}{{\mathbf{t}}}
\newcommand{\wg}{{\mathbf{g}}}
\newcommand{\wc}{{\mathbf{c}}}
\newcommand{\wh}{{\mathbf{h}}}
\newcommand{\wJ}{{\mathbf{J}}}
\newcommand{\wC}{{\mathbf{C}}}
\newcommand{\wV}{{\mathbf{V}}}
\newcommand{\wLam}{{\mathbf{\Lambda}}}

\newtheorem{theorem}[subsubsection]{Theorem}
\newtheorem*{theorem*}{Theorem}
\newtheorem*{FL*}{Fundamental Lemma}
\newtheorem{FUL}[subsubsection]{Fundamental Lemma}
\newtheorem*{GS*}{Geometric Stabilization}
\newtheorem{GS}[subsubsection]{Geometric Stabilization}
\newtheorem{proposition}[subsubsection]{Proposition}
\newtheorem{propdef}[subsubsection]{Proposition-Definition}
\newtheorem{corollary}[subsubsection]{Corollary}

\newtheorem{lemma}[subsubsection]{Lemma}

\theoremstyle{definition}
\newtheorem{definition}[subsubsection]{Definition}

\newtheorem{rmk}[subsubsection]{Remark}

\newtheorem{example}[subsubsection]{Example}

\numberwithin{equation}{subsection}

\begin{document}

\author{Arthur Forey}
\address{Univ. Lille, CNRS, UMR 8524-Laboratoire Paul Painlevé, F-59000 Lille, France} 
  \email{arthur.forey@univ-lille.fr}
\urladdr{https://pro.univ-lille.fr/arthur-forey/}

\author{Fran\c {c}ois Loeser}
\address{Institut universitaire de France, Sorbonne Universit\'e, Institut de Math\'ematiques de Jussieu-Paris
Rive Gauche, CNRS, Campus Pierre et Marie Curie, case 247, 4 place Jussieu, 75252 Paris cedex 5, France.
}
\email{francois.loeser@imj-prg.fr}
\urladdr{https://webusers.imj-prg.fr/$\sim$francois.loeser/}

\author{Dimitri Wyss}
\address{EPFL/SB/ARG, Station 8, CH-1015 Lausanne, Switzerland
}
\email{dimitri.wyss@epfl.ch}
\urladdr{https://people.epfl.ch/dimitri.wyss}

\title{A motivic Fundamental Lemma}

\maketitle 
\begin{abstract} 
In this paper we prove motivic versions of the Langlands-Shelstad
Fundamental Lemma and Ngô's Geometric Stabilization. 
To achieve this, we follow the strategy from the recent proof by 
Groechenig, Wyss and
Ziegler which avoided the use of perverse sheaves using instead  $p$-adic integration and Tate duality.
We  make a key use of a construction of Denef and Loeser which assigns a virtual motive to any definable set 
in the theory of pseudo-finite fields.
\end{abstract}

\tableofcontents

\section{Introduction}
\subsection{}First introduced in 1979 as a technical  device of combinatorial nature, the Langlands-Shelstad Fundamental Lemma was a conjecture  which has long resisted the efforts of mathematicians until its full proof by Ng\^o in 2008 \cite{ngo:fl}. 
We shall start with a loose presentation for non-specialists of some of the motivations behind it.
A powerful tool in the Langlands program is 
the Arthur-Selberg trace formula which is a far-reaching extension of the classical Poisson summation formula for an arbitrary reductive group $G$ over a global field.
It is an equality between a spectral side involving traces of automorphic representations and a so-called geometric side which is a sum of adelic orbital integrals.
These adelic orbital integrals are products of orbital integrals over local fields, that is integrals over $G(F)$-conjugacy classes, with $F$ a local field.
Langlands's principle of functoriality predicts that   automorphic representations for a given reductive group $G$ should transfer to 
automorphic representations for any another reductive group $G'$, given a morphism between their Langlands duals.
An approach to prove it in some cases consists in comparing the geometric side of the trace formula for the two groups.
If $F$ is, say, a local field, an obstacle that immediately arises before going further  is that in general there is  a difference between stable conjugacy classes, 
that is $G (\overline{F})$-conjugacy classes with $\overline{F}$ the algebraic closure of $F$, and $G (F)$-conjugacy classes.
Indeed, while it is  often possible to associate stable conjugacy classes for $G'$ to stable conjugacy classes for $G$
there is no natural way to do this for conjugacy classes over $F$.
The Fundamental Lemma is a way to overcome this difficulty by expressing orbital integrals for a reductive group $G$ over  a local field in terms
of stable orbital integrals (integrals over stable conjugacy classes) for different groups over the same local field.

\subsection{}Let us now state the Fundamental Lemma  in the form  proved by Ng\^o in \cite{ngo:fl}.
We fix a non-archimedean local field $F$ and a  reductive group $G$ over $F$.
We assume $G$ is unramified, that is, that $G$ is the generic fiber of a reductive group scheme over the valuation ring of $F$.
Let $\gamma_G$ be a  regular semi-simple $F$-point of the Lie algebra $\mathfrak{g}$ of $G$ with centralizer
$I_{\gamma_G}$. 
The orbital integral associated to these data is 
\[O_{\gamma_G} = \int_{I_{\gamma_G}(F) \backslash G (F)} \mathbf{1}_{\mathfrak{g}(\mathcal{O}_F)} (g^{-1} \gamma_G g) dg,\] with $dg$ a suitably normalized Haar measure.
The set of $G(F)$-conjugacy classes  
in the stable conjugacy class of $\gamma_G$
can be 
identified with the
set of elements of  $H^1 (F, I_{\gamma_G})$ whose image in $H^1 (F, G)$ is trivial. 
Under this identification the class of $\gamma_G$ corresponds to the identity element.
By Fourier transformation on the finite abelian group
$H^1 (F, I_{\gamma_G})$, the orbital integral
$O_{\gamma_G}$ can be expressed as a linear combination of
$\kappa$-orbital integrals $O_{\gamma_G}^{\kappa}$
where,
for  a character $\kappa : H^1 (F, I_{\gamma_G}) \to \mathbb{C}^\times$, 
$O_{\gamma_G}^{\kappa}$
is defined as
\[O_{\gamma_G}^{\kappa} = \sum_{[\gamma']} \kappa ([\gamma']) \, O_{\gamma'},\] 
 with $[\gamma']$ running over
the set of $G(F)$-conjugacy classes  within the stable conjugacy class
of $\gamma_G$. When $\kappa$ is the trivial character one gets the stable orbital integral 
$O^{\mathrm{stab}}_{\gamma_G}$. Thus, the problem of stabilization of orbital integrals can be  reduced to expressing a
$\kappa$-orbital integral for $G$ in terms of a stable orbital integral on some different unramified reductive group $H$ over $F$.

The group $H$ is constructed via the theory of endoscopy out of the data  $(G, I_{\gamma_G}, \kappa)$. In particular, the set of roots  of the endoscopic group $H$ is a  subset of the set of all roots of $G$. Furthermore,
to each $\gamma_G$ as before corresponds a    $F$-point $\gamma_H$ of the Lie algebra $\mathfrak{h}$ of $H$ 
which is well defined up to  stable conjugacy and which will be assumed to be regular semi-simple.
In particular, the stable orbital integral 
$O^{\mathrm{stab}}_{\gamma_H}$ depends only on $\gamma_G$ and $H$.
On the other hand, the $\kappa$-orbital integral $O_{\gamma_G}^{\kappa}$ depends on the conjugacy class of $\gamma_G$, another choice 
$\gamma'_G$ would multiply it by a factor $\kappa ([\gamma'_G])$.
When the characteristic of $F$ is greater than twice the Coxeter number of $G$, one way to solve this indeterminacy is to use
the so-called  Kostant section which provides
a canonical conjugacy class within
each stable conjugacy class of regular semisimple elements in $\mathfrak{g} (F)$.

The Fundamental Lemma, as proved by Ng\^o in \cite{ngo:fl}, relates $\kappa$-orbital integrals on the group $G$
to stable orbital integrals on the endoscopic
group $H$  in the following way:
\begin{FUL}\label{FUL}With the above notation,
assume the characteristic of $F$ is greater than twice the Coxeter number of $G$ and that the conjugacy class of $\gamma_G$ is the one provided by the Kostant section.
Then
we have the equality
\begin{equation}\tag{$\ast$}\label{FLequ}
O_{\gamma_G}^{\kappa}
= \Delta (\gamma_H, \gamma_G) \, O^{\mathrm{stab}}_{\gamma_H},
\end{equation}
with $\Delta (\gamma_H, \gamma_G)$ of the form $q_F^r$, with $q_F$ the cardinality of the residue field  of $F$ and $r$ an explicit integer.
\end{FUL}

\subsection{}
Heuristically, 
all relations between integrals should come, in a certain sense, from geometry as mere shadows of geometric relations.
A suggestive instance of this principle
is provided by the
 Kontsevich-Zagier conjecture on algebraic relations between periods \cite{KZ}. In our case an example of such a phenomenon  already showed up in the proof of the Fundamental Lemma for $\mathrm{Sp}(4)$ by  Hales \cite{hales_pams}.
Indeed, in this case, each side of the Fundamental Lemma can be expressed as the number points of an elliptic curve over a finite field, and
the crucial  fact used in the proof is that these two elliptic curves are isogenous, hence have the same number of points, cf.  \cite{hales_hyperelliptic,hales_duke}.
As isogenous elliptic curves have the same rational Chow motive, this leads to ask whether (\ref{FLequ}) could already hold at the level of rational Chow motives.
The main result of the present paper is  that  it is indeed possible to lift the identity
(\ref{FLequ}) to  an identity between two objects  in a Grothendieck ring of  rational Chow motives.
More precisely we prove an equality in the Grothendieck ring of Chow motives that specializes to the Fundamental Lemma for $p$ large enough.
Note that a statement of this type was conjectured by Hales in \cite{hales_computed}.
It would be interesting to establish a non-virtual version of our result, similarly to what Hoskins and Pepin Lehalleur prove in \cite{HPL}, a non-virtual analogue of the result of Loeser and Wyss \cite{LW19} on topological Mirror Symmetry.

\subsection{}One of the key novelty in Ngô's proof 
is to  move from local geometry over a local non-archimedean field to
the global geometry over a curve, using the Hitchin fibration which  is the global analogue of affine Springer fibers. 
Ngô's proof is of geometric nature,
and uses the decomposition theorem 
to 
control the
support of the perverse cohomology sheaves of the Hitchin fibration.
Recently, Groechenig, Wyss and Ziegler 
gave a new proof of the Fundamental Lemma using $p$-adic integration along the Hitchin fibration \cite{GWZ18}, inspired by their proof of the topological  Mirror Symmetry conjecture of Hausel-Thaddeus \cite{ht_inv}.
More precisely, they provide a new proof of Ngô's
 Geometric Stabilization  which is known to imply
 the Fundamental Lemma by Ngô's work \cite{ngo:fl}.
 An important feature of their proof is that it circumvents the use of perverse sheaves and the decomposition theorem, which seemed insuperable 
obstacles to a motivic version. Instead a transition from $p$-adic to motivic arguments is often possible, for example in the case of topological Mirror Symmetry \cite{LW19}.

\subsection{}We use the  general theory of motivic integration over fields of characteristic zero developed by Cluckers and Loeser in \cite{CL-2008}
which is based on the notion of definability in model theory. Indeed it was already proved by Cluckers, Hales and Loeser in \cite{chl},  that all the terms occurring in 
the Fundamental Lemma
can be defined purely within a first order language for valued fields, namely
the  Denef-Pas
language. Thus using the ``Transfer Principle'' from \cite{cl_annals}, a generalization of the Ax-Kochen-Er\v{s}ov theorem, 
 they deduce the Fundamental Lemma in mixed characteristic from the equicharacteristic case, recovering conceptually an earlier result of 
Waldspurger \cite{W06}.

\subsection{}In the theory of motivic integration  developed in \cite{CL-2008}, volumes are elements of the ring $\cC (*_k)$ which is 
a localization of the Grothendieck ring of definable sets over the residue field $k$ (assumed to be of characteristic zero). 
Such 
definable sets correspond to equivalence 
classes of
formulas in the first order ring language with coefficients in $k$, two formulas being equivalent if they are  equivalent when interpreted in every field extension of $k$.
Since in this paper we are only concerned with local non-archimedean fields, which have finite residue fields, it is natural to consider only fields extensions
of $k$ that satisfy all first order properties satisfied by finite fields. 
Such fields were introduced by Ax  in \cite{ax} under the name of pseudo-finite fields.
Recall that a field $L$ is pseudo-finite 
if and only if it 
is perfect with a unique extension in every positive degree and every geometrically irreducible variety over $L$ has an $L$-rational point. 
We will consider a specialization of the theory 
of \cite{CL-2008} for which  volumes no longer take place in the ring $\cC (*_k)$, but in a ring called
$\cC_\psf (*_k)$ which is defined exactly as $\cC (*_k)$, except that two formulas will be considered to be equivalent if and only if they are  equivalent when interpreted in every pseudo-finite field extension of $k$.

\subsection{} In order to identify for example classes of dual abelian varieties we need to further specialize  $\cC _\psf(*_k)$. By a result of Denef and Loeser \cite{DL-PSF} (see also \cite{dl_icm,nicaise})
one may assign to any element $\alpha$ of $\cC_\psf (*_k)$  a virtual Chow motive $\chi_\psf (\alpha)$ 
such that 
the number of points of $\alpha$ interpreted in a finite field of sufficiently large characteristic can
be recovered from $\chi_\psf (\alpha)$ 
by taking the trace of the Frobenius on its étale realization.
Combining the  construction from \cite{DL-PSF} and the general theory of 
 \cite{CL-2008} one can work out a theory of motivic integration with values in localized Grothendieck rings of Chow motives, which specializes well to $p$-adic integration for $p$ large enough.
 This  is the framework which is used in this paper.  Note that the integration theory which was used 
 in
 \cite{LW19}, for which  two formulas were considered to be equivalent 
if they are equivalent when interpreted algebraically  closed field extensions of $k$, would not
suffice here, since it  does not have the property that the volume of a definable set specializes to its $p$-adic volume.

 \subsection{}Before saying more on the structure of the paper,
 let us recall the statement of
Geometric Stabilization.
 Let $X$ be a smooth 
 projective geometrically connected curve over a finite field $k$.
 Fix a line bundle $D$ on $X$ and 
 a quasi-split reductive group $G$ over $X$.
 A $G$-Higgs bundle with coefficients in $D$ is a pair $(E, \theta)$ with $E$ a $G$-torsor on $X$ and $\theta$ a global section of $\ad(E) \otimes D$.
 We denote by 
 $\FM_G = \FM_G(\X,D)$ the moduli stack of $G$-Higgs bundles with coefficients in $D$
 on $X$ and by $\chi: \FM_G \rightarrow \A$
 the Hitchin fibration.

The Hitchin fibration is endowed with an action of a Picard stack  $\FP_G \to \A$.
The anisotropic locus $\A^{\ani}$ in $ \A$ 
 is   the locus of points  $a$ such that
 every Higgs bundle $(E, \theta) \in \chi^{-1} (a)$
 is
generically anisotropic in the sense that for any trivialization of $E$ and $D$ over the generic point of $X$, the Higgs field $\theta$ is anisotropic as an element of $\mathfrak{g}(k(X))$. 
When $a$ belongs to  
 $\A^{\ani}$, the group
 $\pi_0 (\FP_{G,a})$ is  finite.
 Fix $a \in  \A^{\ani}$, $H$ an  endoscopic group of $G$, 
 and $\kappa: \pi_0 (\FP_{G,a}) \to \overline{\mathbb{Q}_{\ell}}^{\times}$
the corresponding character.
For technical reasons one needs to work on some  \'etale open covering $\widetilde{\A}^{\ani} \to  \A^{\ani}$,
and we denote by $\widetilde \FM_G$, etc, the corresponding pullback.
The  locus $\widetilde{\A}_H^{\ani}$ relative to $H$ is contained in $\widetilde{\A}^{\ani}$.

\begin{GS}\label{GS}
For every $a \in \widetilde{\A}_H^{\ani}(k)$
we have
\[\#^{\kappa} \,
\widetilde{\FM}_{G,a} (k)
=
q^{r_G^H (D)}
\#^{\mathrm{\Stab}} \,  \widetilde{\FM}_{H,a} (k).
\]
\end{GS}

Here $\#^{\kappa}$
denotes
the trace of Frobenius on the $\kappa$-isotypical component of the cohomology with compact supports,
$\#^{\mathrm{stab}}$ denotes  $\#^{\kappa}$ when $\kappa$ is the trivial character,
$q$ is the cardinality of $k$ and 
$r_G^H (D)$ denotes $1/2 (\mathrm{dim} (\widetilde \FM_G)  - \mathrm{dim} (\widetilde \FM_H))$.

\subsection{}The structure of the paper is the following.
In Section \ref{sec1} we develop a version
of motivic
integration which is suited for the purpose of this paper, using
the general framework  from \cite{CL-2008}.
This can be viewed  as a relative version of 
the integrals constructed  in \cite{DL-PSF},
providing volumes in Grothendieck groups of \'etale motives
to  definable sets in the Denef-Pas language with pseudo-finite residue fields.
We show that the basic properties we need (change of variable, Fubini theorem, specialization to $p$-adic integration, formula for the orbifold volume)
still hold in this context.
In the short Section \ref{sec:geo-setup}, which is mostly borrowed from
Section 4 of \cite{ngo:fl}, we recall the preliminary material required in order  to formulate the Geometric Stabilization theorem. 
In Section \ref{sec:def:geo:stab}, we show how the various objects appearing in the statement of the Geometric Stabilization theorem can actually be encoded
using the Denef-Pas language as definable
objects  in the theory of pseudo-finite fields.
In particular Galois cohomology is encoded by cocycles so that  classical   isomorphisms between
Galois cohomology groups of finite fields 
lift to definable isomorphisms over pseudo-finite fields.
We also 
use  twists of definable sets by Galois cocycles to construct isotypical components of  motives associated to definable sets.
It is in this framework that we  formulate 
our  motivic version of Geometric Stabilization
as an equality between motivic integrals, Theorem \ref{th:geo:stab}.

Section \ref{sec:proofmt} is devoted to the proof of motivic 
Geometric Stabilization, Theorem \ref{th:geo:stab:a}.
Following \cite{GWZ18} we use the formalism of orbifold measures to reduce it to an identity between motivic integrals on Hitchin fibers for Langlands dual groups,
Theorem \ref{mainmint}. By a Fubini theorem one may further reduce to an identity involving only generically smooth Hitchin fibers, which then follows from Tate duality.

Finally, 
we explain in Section \ref{sec:fl} how the strategy introduced by  Ng\^o 
to deduce the Fundamental Lemma from Geometric Stabilization can be adapted to the motivic setting, allowing us to prove our motivic version of the 
Fundamental Lemma, Theorem \ref{th:fl}. In particular, it follows directly from 
Theorem \ref{th:fl} that the Fundamental Lemma \ref{FUL} holds for quasi-split reductive groups of large enough residue characteristic
(Corollary \ref{corfl}).

\section{Motivic integration with pseudo-finite residue field}\label{sec1}
In this section we recall what we need about motivic integration, specialized to the case of pseudo-finite residue field. We moreover extend the construction for motivic integrals to take values in the Grothendieck groups of \'etale motives, and show that many compatibility properties continue to hold in this context, such as the Fubini theorem. 

\subsection{Pseudo-finite fields}
A pseudo-finite field is an infinite field 
satisfying all first order sentences in the ring language that are true in every finite field. 
More concretely, by Ax's theorem \cite{ax} a field $K$ is pseudo-finite if and only if it is perfect, admits a unique extension of degree $n$ for each $n\in \NN$, and is pseudo-algebraically closed, which means that every geometrically irreducible $K$-variety admits a $K$-point. One can show that such properties are axiomatizable by first order formulas in the ring language and we denote $T_\psf$ the corresponding theory. For a field $k$, we also denote by $T_{\psf,k}$ the theory of pseudo-finite fields containing $k$. Denote by $\RDef_k$ the category of definable sets in the theory of pseudo-finite fields containing $k$, and let $\K{\RDef_k}$ be the corresponding Grothendieck ring\footnote{The terminology $\RDef_k$ will make sense in Section \ref{sec-DPlang}, where we consider valued fields with pseudo-finite residue fields.}. More generally, when $S$ in $\RDef_k$ is a definable set, we write $\RDef_S$ for the category of definable sets with parameters in $S$, and $\K{\RDef_S}$ its corresponding Grothendieck ring. 

An especially useful fact in connection with motivic integration is that this theory admits quantifier elimination in the language of Galois formulas, 
and by work of Denef and Loeser in \cite{DL-PSF}, one can associate a
canonical virtual motive to every quantifier free Galois formula.
Let us briefly recall the notion of quantifier-free Galois formulas, referring to 
 \cite{FJ:FA} and  \cite[Section 2]{DL-PSF} for more details.

A Galois stratification of a $k$-scheme $X$ is the data $\FA=(X, C_i/A_i,\Con(A_i))_{i\in I}$ where $(A_i)$ is a finite stratification of $X$ into locally closed normal and integral subschemes, $C_i\to A_i$ is a Galois cover, with group $G(C_i/A_i)$ and $\Con(A_i)$ is a family of subgroups of $G(C_i/A_i)$ which is stable under conjugation. Let $K$ be a field extension of $k$ and $x\in X(K)$ belonging to $A_i$. Let $\Ar(x)$ be the Artin symbol at $x$, that is the conjugacy class of the decomposition subgroup of $G(C_i/A_i)$ at $x$. We write $\Ar(x)\subseteq \Con(\FA)$ if $\Ar(x)\subseteq \Con(A_i)$. We now define
\[
\FA(K)=\set{x\in X(K)\mid \Ar(x)\subseteq \Con(\FA)}.
\]
If $X$ is affine over $k$, there is a ring formula $\varphi$ with parameters in $k$ such that $\varphi(K)=\FA(K)$ for every pseudo-finite field $K$ extending $k$. We call such a formula a quantifier-free Galois formula, see \cite[Remark 30.6.4]{FJ:FA} for details.  Reciprocally, every ring formula is equivalent to a Galois formula with quantifiers by \cite[Remark 30.5.1]{FJ:FA}.

Let $\mcL_{\psf,k}$ be the language consisting of constants of $k$ and an $n$-ary relation symbol for every Galois stratification of the affine space $\mathbb{A}^n_k$. Every field extension $K$ of $k$ is naturally an $\mcL_{\psf,k}$-structure, where we interpret the relation corresponding to a Galois stratification $\FA$ as $\FA(K)$. The theory $T_{\psf,k}$ admits quantifier elimination in the language $\mcL_{\psf,k}$, by \cite[Proposition 31.1.3]{FJ:FA} (note that pseudo-finite fields are Frobenius fields in the terminology of \cite{FJ:FA}).

\subsection{\'Etale motives}
Let $k$ be a field of characteristic zero, $\Lambda$ a field of characteristic zero containing all roots of unity and $S$ a quasi-projective $k$-scheme. We write $\DA(S,\Lambda)$  for the triangulated category of \'etale motives over $S$ with coefficients in $\Lambda$ introduced by Ayoub \cite{Ayoub_ENS}. He also endowed $\DA(S,\Lambda)$ with a six-functor formalism and the full subcategory $\DAC(S,\Lambda)$ of constructible objects is stable under these six functors. We refer the reader to \cite{ayoub_icm} for a nice introduction.

Let $\K{\Var_S}$ be the Grothendieck group of varieties over $S$ and $\cM_S$ its localization by the  class $\mathbb{L}$ of the affine line over $S$. Ivorra and Sebag prove in \cite[Lemma 2.1]{ivo_seb}
the existence of a unique ring morphism
\[
\chi_{S} :
\cM_S \longrightarrow \K{\DAC (S, \Lambda)}
 \]
which assigns to a quasi-projective $S$-scheme $p : X \to S$ the element $\chi_{S} ([X]) := [p_! (\mathds{1}_X)]$ in the Grothendieck ring
$\K{ \DAC (S, \Lambda)}$.

Recall the following lemma from \cite{LW19}.
\begin{lemma}[{\cite[Lemma 2.1.1]{LW19}}]\label{evmot} Let $S$ be a quasi-projective $k$-scheme.
Let $\alpha \in \K{\DAC (S, \Lambda)}$ be such  that, for every schematic point $i_x : \set{x} \hookrightarrow S$ we have
$i_x^{*} (\alpha) = 0$. Then $\alpha = 0$.
\end{lemma}

Using Galois stratifications, we will extend $\chi_S$ to a ring morphism 
\[
\chi_{\psf,S} \colon \K{\RDef_S}\longrightarrow \K{\DAC(S,\Lambda)}\otimes \QQ
\]
in Section \ref{sec:psrel}.

Let $G$ be a finite group. Let $G^*$ be the set of characters of irreducible $\Lambda$-representations of $G$ and $R(G)$ be the group of virtual characters of $G$ with $\QQ$-coefficients.  

Let $M\in \DAC(S,\Lambda)$ be endowed with a $G$-action, \emph{i.e.} endowed with a group morphism $\rho\colon G\to \Aut_{\DAC(S,\Lambda)}(M)$. Let $\alpha\in G^*$ be a character of an irreducible $\Lambda$-representation of $G$ of dimension $n_\alpha$. We consider the projector
\[p_\alpha=\frac{n_\alpha}{\abs{G}}\sum_{g\in G}\alpha(g^{-1}) \rho(g)\in \End(M).\]
Since $\DAC(S,\Lambda)$ is pseudo-abelian by \cite[Proposition 9.2]{Ayoub_ENS}, the image of $M$ by $p_\alpha$ is well-defined. We denote it by $M^\alpha$ and call it the $\alpha$-isotypical component of $M$. When $\alpha$ is the trivial character, we also write $M^\alpha=M^G$, the $G$-invariant component of $M$. Since $\bigoplus_{\alpha\in G^*} p_\alpha=\id\in \End(M)$, we have that $M\simeq \bigoplus_{\alpha\in G^*} M^\alpha$.

For $S$ a quasi-projective $k$-scheme we consider $\Sch_S^G$, the category of $S$-schemes endowed with a fiberwise good $G$-action, by which we mean schemes $p \colon X\to S$ with a $G$-action such that for every $x\in X$, the $G$-orbit of $x$ is contained in an affine subset of $p^{-1}(p(x))$. Notice that for $X \to S$ quasi-projective, any $G$-action on $S$ is good.

By functoriality, for an object $p\colon X\to S$ of $\Sch_S^G$ its motive $p_! (\mathds{1}_X)$ is endowed with a $G$-action. Hence for $\alpha\in G^*$, its $\alpha$-isotypical component $p_! (\mathds{1}_X)^\alpha$ is defined above. 

The following proposition is a relative version of \cite[3.1.1]{DL-PSF} or \cite[6.1]{BN-equiv-mot}.

\begin{proposition}
\label{prop-chialpha} Let $S$ be a quasi-projective $k$-scheme, $G$ a finite group and $\alpha\in R(G)$ a virtual character of $G$. There is a unique  group morphism\footnote{This is merely a group morphism, not a ring morphism, contrary to what is claimed in \cite[3.1.1]{DL-PSF}.}
\[
\chi_S(-,\alpha) : \K{\Sch_S^G} \longrightarrow \K{\DAC(S,\Lambda)}\otimes \QQ
\]
such that:
\begin{enumerate}
\item If $p\colon X\to S \in \Sch_S^G$ and $\alpha\in G^*$ is the character of an irreducible representation of dimension $n_\alpha$ of $G$, then 
\[
n_\alpha\chi_S([X],\alpha)=[p_! (\mathds{1}_X)^\alpha].
\]
\item For every $X\in \Sch_S^G$, 
\[
\chi_S(X)=\sum_{\alpha\in G^*} n_\alpha \chi_S ([X],\alpha),
\]
where the sum runs over characters $\alpha$ of irreducible representations of $G$ and $n_\alpha$ denotes their dimension.
\item For every $X\in \Sch_S^G$, the map 
\[
\alpha\in R(G)\longmapsto \chi_S(X,\alpha)\in\K{\DAC(S,\Lambda)}
\otimes \QQ\]
 is a group morphism. 
\end{enumerate}
\end{proposition}

\begin{proof}
Let $\alpha\in G^*$ be the character of an irreducible $\Lambda$-representation $V_\alpha$ of $G$  of dimension $n_\alpha$. View $V_\alpha$ as an object of $\DAC(S,\Lambda)$ endowed with a $G$-action, by identifying it with $\bigoplus_{1\leq i\leq n_\alpha}\mathds{1}_S$. Let $p\colon X\to S\in \Sch_S^G$. Define 
\[\chi_S([X],\alpha)=[\iHom(V_\alpha, p_!(\mathds{1}_X))^G],
\] 
where $\iHom$ is the internal Hom.

By additivity of $p_!$, as in the proof of \cite[Lemma 2.1]{ivo_seb}, this extends uniquely to a group morphism 
\[
\chi_S(-,\alpha) : \K{\Sch_S^G}\longrightarrow \K{\DAC(S,\Lambda)}.
\]
Property $(1)$ now follows from the standard fact that $\iHom(V_\alpha, M)^G\otimes V_\alpha$ is isomorphic to $M^\alpha$, 
see \cite[6.1]{BN-equiv-mot} for details. Property $(2)$ follows from the fact that $[M]=\sum_{\alpha\in G^*} [M^\alpha]$. We finally define $\chi_S(-,\alpha)$ for a virtual character $\alpha$ in the unique way such that $(3)$ holds. 
\end{proof}

The following proposition collects a few useful functorialities.
\begin{proposition}
\label{prop-chialpha-change-gps} 
 Let $G$ and $H$ be finite groups, $X\in \Sch^G_S$, $Y\in \Sch^H_S$,  $\alpha$ be a character of $G$, $\beta$ a character of $H$.
\begin{enumerate}
\item View $\alpha\cdot\beta$ as a character of $G\times H$. Consider $X \times_S Y$ in $\Sch^{G\times H}_S$. Then
\[
\chi_S([X \times_S Y],\alpha\cdot\beta)=\chi_S([X],\alpha)\chi_S([Y],\beta).
\]
\item Assume that $H$ is a normal subgroup of $G$, with quotient morphism $\rho\colon G\to G/H$, and let $\gamma$ be a character of $G/H$. Then
\[
\chi_S([X/H]),\gamma)=\chi_S([X],\gamma\circ\rho).
\]
\item Assume that $H$ is a subgroup of $G$. Then
\[
\chi_S([X],\mathrm{Ind}_H^G \beta)=\chi_S([X],\beta),
\]
where in the second term $X$ is viewed in $\Sch_S^H$. 
\item Assume that $H$ is a subgroup of $G$ and that $X$ is isomorphic as a $G$-variety to $\sqcup_{s\in G/H}sY$.  Then
\[
\chi_S([X],\alpha)=\chi_S([Y],\mathrm{Res}_H^G\alpha). 
\]
\end{enumerate}
\end{proposition}

\begin{proof}
The proof is similar to that of  Proposition 3.1.2  in \cite{DL-PSF}.
\end{proof}

\subsection{Pseudo-finite realization}
\label{sec:psrel} Let $A$ be an integral normal locally closed subscheme of an $S$-scheme $X$, $C\to A$ a Galois cover with Galois group $G$ and $\Con$ a family of subgroups of $G$ stable under conjugation. 
Let $\alpha_\Con$ be the central function on $G$ defined by
\[
\alpha_\Con(x)=\left\{ 
\begin{array}{ll}
1&\text{if }\langle x \rangle \: \in \Con,\\
0 &\text{otherwise},
\end{array}
\right.
\]
where $\langle x\rangle$ is the subgroup of $G$ generated by $x$. The function $\alpha_\Con$ can be written as a $\QQ$-linear combination of irreducible characters of $G$. One defines $\chi_S(C/A, \Con)=\chi_S(C, \alpha_\Con)\in\K{\DAC(S,\Lambda)}\otimes \QQ$, and  for a Galois stratification $\FA=(X, C_i/A_i,\Con(A_i))_{i\in I}$ one sets
\[
\chi_S(\FA)=\sum_i \chi_S (C_i,\alpha_{\Con_i}).
\]
If two Galois stratifications $\FA$ and $\FA'$ of $X$ are equivalent, in the sense that they define the same definable subset of $X$ in every pseudo-finite field, then by \cite[3.2.1]{DL-PSF}, $\chi_S(\FA)=\chi_S(\FA')$. Moreover, if two Galois stratifications $\FA$ and $\FA'$ of $X$ and $X'$ are such that there is a Galois formula defining a bijection between the interpretation of $\FA$ and $\FA'$ in every pseudo-finite field extending $k$, then by \cite[3.3.5]{DL-PSF}, $\chi_S(\FA)=\chi_S(\FA')$. Since $T_{\psf}$ admits quantifier elimination in the language of Galois formulas, one can define a group morphism
\[
\chi_{\psf,S} \colon \K{\RDef_{\psf,S}}\longrightarrow \K{\DAC(S,\Lambda)}\otimes \QQ,
\]
by sending the class of a definable set represented by a Galois stratification $\FA$ to $\chi_S(\FA)$. It follows from Proposition \ref{prop-chialpha-change-gps} that $\chi_{\psf,S}$ is a ring morphism. 

By construction, if $X$ is definable over $S$ in the ring language without quantifiers, for example if it is an algebraic variety over $S$, then $\chi_{\psf,S}([X])=\chi_S([X)]$, where the first class is in $\K{\Sch_S}$ and the second in $\K{\RDef_{\psf,S}}$.

\begin{lemma}
\label{lem-comp-chipsf-proj}
Let $p\colon S\to S'$ be a morphism between quasi-projective $k$-schemes. The following diagrams commute:
\begin{equation}
\xymatrixcolsep{1pc}
\xymatrix{
\K{\RDef_S} \ar[d]_{p_!} \ar[rr]^-{\chi_{\psf,S}} &&\K{\DAC(S,\Lambda)}\otimes \QQ \ar[d]^{p_!}  \\
\K{\RDef_{S'}} \ar[rr]^-{\chi_{\psf,S'}} &&\K{\DAC(S',\Lambda)}\otimes \QQ,
}
\end{equation}
\begin{equation}
\xymatrixcolsep{1pc}
\xymatrix{
\K{\RDef_{S'}} \ar[d]_{p^*} \ar[rr]^-{\chi_{\psf,S'}} &&\K{\DAC(S',\Lambda)}\otimes \QQ \ar[d]^{p^*}  \\
\K{\RDef_{S}} \ar[rr]^-{\chi_{\psf,S}} &&\K{\DAC(S,\Lambda)}\otimes \QQ.
}
\end{equation}
\end{lemma}
\begin{proof}
For the first diagram, fix $W\in\RDef_S$ and $\FA=(X, C_i/A_i,\Con(A_i))_{i\in I}$ a Galois stratification corresponding to $W$, where $X$ is an $S$-scheme. Let $\FA'$ be the same Galois stratification, but viewing $X$ as an $S'$-scheme via $p$. Then $\FA'$ is a Galois stratification of $W$ viewed in $\psf_{S'}$ via $p$, hence $\chi_{\psf,S'}(p_![W])=\chi_{S'}(\FA')$. It follows finally from Ayoub's six operations formalism that $\chi_{S'}(\FA')=p_!\chi_S(\FA)$. 

The proof of the commutativity of the second diagram is similar, using the fact that the base change along $p$ of a Galois stratification of an $S'$-scheme $X$ is a Galois stratification of $X\times_S S'$, and that the associated definable sets correspond. 
\end{proof}

\subsection{Motives with definable parameters}

We shall need to consider also definable sets with parameters, and to associate motives to those. Let $S\in \RDef_k$, say with $n$ variables, and let $\RDef_S$ be the category of definable sets with parameters in $S$, \emph{i.e.} definable sets $X\subseteq \BA^m\times S$ together with the projection to $S$. Morphisms are definable bijections preserving the projection to $S$.

Fix $X\in \RDef_S$. Since $S\subseteq \BA^n$, there is a canonical map $X\to \BA^n$ and we can consider the \emph{motive of $X$ with parameters} $\chi_{\psf,\BA^n}([X])\in\K{\DAC(\BA^n,\Lambda)}\otimes \QQ$.

Since the precise set of parameters is not important, and that we shall need to change it regularly, we adopt the following convention.

\begin{definition}
\label{def-mot-par}Let $S\subseteq \BA^n$  be a 
definable set and $M$ be an element of $\K{\DAC(\BA^n,\Lambda)}$ such that $M=M\cdot\chi_{\psf,\BA^n}(S)$.
We call such an $M$
a \emph{motive with parameters in $S$}, or a \emph{relative motive}. 

We define the ring of relative motives $\K{\DAC(\rel,\Lambda)}$ as the colimit of the groups  $\K{\DAC(\BA^n,\Lambda)}$ under the closed immersions
$\BA^n \hookrightarrow  \BA^{n+1}$ given by $x \mapsto (x, 0)$.
\end{definition}

We adopt the same definition for the various tensor products of $\K{\DAC(\rel,\Lambda)}$ that we shall consider, such as $\K{\DAC(\rel,\Lambda)}\otimes \QQ$. We also extend the construction to motives over some base scheme $A$, which gives $\K{\DAC(A,\rel,\Lambda)}$, with elements motives in some $\K{\DAC(A\times \BA^n,\Lambda)}$ for some $n$. 
Similarly, given a definable set $X\in \RDef_S$ with parameters $S\subseteq \BA^n$, we write $\chi_{\psf,\rel}([X])\in \K{\DAC(\rel,\Lambda)}\otimes \QQ$ for $\chi_{\psf,\BA^n}([X])$.

\subsection{Denef-Pas language and definable subassignments}
\label{sec-DPlang}
We use in this paper the framework for motivic integration developed by Cluckers and Loeser in \cite{CL-2008}. See also \cite{chl} for an introduction. 

Fix a field $k$ of characteristic zero. The theory takes place in the Denef-Pas 3-sorted language $\dplk$ in the sense of first order logic. It has one sort for the valued field with the ring language, one sort for the residue field with ring language, and one sort for the value group, with Presburger arithmetic language (0,1,+,$\leq$, $(\equiv_n)_{n\in \N})$. It also has two unary function symbols, $\ord$ from the valued field sort to the value group, interpreted as the valuation, and $\ac$ from the valued field sort to the residue field, interpreted as an angular component. It also has constant symbols for elements of $k \llp t \rrp$.

Given an extension $K$ of $k$, we view naturally the field $K\llp t \rrp$ as an $\dplk$-structure, as follows. The underlying sets are $(K\llp t \rrp, K,\ZZ)$. The function symbol $\ord$ is interpreted as the valuation, $\ac$ as the angular component, which is the first non-zero coefficient of $x\in K\llp t \rrp$ if $x\neq 0$ and 0 otherwise. The relation  $\equiv_n$ is interpreted as congruence modulo $n$.

One considers the category $\Def_k$ of definable objects in such structures, defined as follows. Let $\varphi$ be an $\dplk$-formula, with respectively $m$, $r$ and $s$ free variables in the three sorts. For every extension $K$ of $k$, we consider the subset $h_\varphi(K)\subseteq (K\llp t \rrp)^m\times K^r\times \ZZ^s$ consisting of points satisfying $\varphi$. An object $S\in \Def_k$, called a definable subassignment (or definable set), is a map from the set of extensions $K$ of $k$ to sets such that there is some $\dplk$-formula $\varphi$ such that $S(K)=h_\varphi(K)$ for every $K$. One also considers general definable subassignments $\GDef_k$, which objects are definable subassignments of algebraic varieties over $k$, defined similarly using affine charts. 

In this paper, we specialize this situation to the case of pseudo-finite residue fields. Let $T_{\psf,k}$ the theory of pseudo-finite fields containing $k$. We work in the category denoted by $\Def_k(\dpl,T_\psf)$ in \cite[Section 16]{CL-2008}, that we shall denote by $\Def_{\psf,k}$. Let $\mcL_{\DP,\psf,k}$ be the extension of the Denef-Pas language $\dplk$ where we add symbols for Galois formulas in the residue field. An object $S\in \Def_{\psf,k}$ is a map from the set of pseudo-finite fields $K$ extending $k$ to sets that are of the form $S(K)=h_\varphi(K)$ for some $\mcL_{\DP,\psf,k}$-formula $\varphi$. Combining Denef-Pas elimination of valued fields quantifiers with Fried-Jarden quantifier elimination for pseudo-finite fields, we see that in the definition of $S\in \Def_{\psf,k}$, we can choose the $\mcL_{\DP,\psf,k}$-formula $\varphi$ to be quantifier free. The category $\GDef_{\psf,k}$ is defined similarly.

We also write $\cC_{\psf, k}(S)$ or $\cC_{\psf}(S)$ for $\cC(S,(\mcL_\DP,T_{\psf,k}))$, the ring of motivic constructible functions on $S\in \GDef_{\psf,k}$, when there is no risk of confusion.

Let $*_k$ the final object in $\GDef_{\psf,k}$, whose value at every pseudo-finite extension of $k$ is the one point set.
By definition
\[
\cC_\psf(*_k)=\K{\RDef_{\psf, k}}\otimes_{\ZZ[\eL]} \mathbb{A},
\]
where $\mathbb{A}=\ZZ \left[  \eL,\eL^{-1}, \left( \frac{1}{1-\eL^{-i}}\right)_{i>0}\right]$. 

\subsection{Evaluation of functions}
We fix a Grothendieck universe $\mathcal{U}$ containing $k$. This is merely a matter of convenience,
it follows from  standard arguments relying on the reflection principle, cf. \cite{feferman},
that our main results are in fact valid within $\zfc$.
Let $S\in \GDef_{\psf,k}$. We define the set of points $\abs{S}$ as the set of pairs $(x,K)$ with $K\in \mathcal{U}$ a pseudo-finite field containing $k$ and $x\in S(K)$. We sometimes write abusively $x\in \abs{S}$. 

For such an $x\in S(K)$, let $k(x)$ be the definable closure of $x$ intersected with the residue field $K$ and consider the language $\mcL_{\DP,\psf,k,x}$, where we add constants for $k(x)$. Note that $\mcL_{\DP,\psf,k,x}$ is smaller than $\mcL_{\DP,\psf,k(x)}$, since we do not add constants symbols for $k(x)\llp t\rrp$ in $\mcL_{\DP,\psf,k,x}$. 

Consider the category of pseudo-finite fields $K'$ containing $k(x)$, and the category $\Def_{\psf,k,x}$ such that its objects $X\in \Def_{\psf,k,x}$ consists of maps associating to $K'$ as above $h_\varphi(K')$ for $\varphi$ an $\mcL_{\DP,\psf,k,x}$-formula. This new theory falls into the framework of \cite[Section 2.7]{CL-2008}, hence we define similarly $\GDef_{\psf,k,x}$, $\RDef_{\psf,k,x}$,  $\cC_{\psf,k,x}(X)$, \emph{etc}. Note that for the one point definable subassignment $*$, we have 
\[\cC_{\psf,k,x}(*)=\cC_{\psf,k(x)}(*_{k(x)})=\K{\RDef_{\psf, k(x)}}\otimes_{\ZZ[\eL]} \mathbb{A}.
\]

A constructible motivic function $\psi\in \cC_{\psf,k}(S)$ can be evaluated at $(x,K)\in \abs{S}$ as follows. 

Since $\mcL_{\DP,\psf,k,x}$ is an expansion of $\mcL_{\DP,\psf,k}$, by \cite[Proposition 16.1.1]{CL-2008},  every $\psi\in \cC_{\psf,k}(S)$ induces a unique function in $\cC_{\psf,k,x}(S)$, that we still denote abusively $\psi$. By cell decomposition, or more precisely \cite[Theorem 7.2.1]{CL-2008}, there exists an $\mcL_{\DP,\psf,k,x}$-definable subassignment $V_x\subseteq S$, such that $x\in V_x(K)$, a definable function $g\colon V_x\to *$ and $\psi_x\in \cC_{\psf,k,x}(*)$ such that $g^*(\psi_x)=\psi$. We define the evaluation of $\psi$ at $(x,K)$ to be $\psi(x)=\psi_x\in \cC_{\psf,k,x}(*)=\cC_{\psf,k(x)}(*_{k(x)})$ .

\begin{rmk}
Our notion of evaluation at a point differs from the one used by Cluckers and Halupczok in \cite{eval}. They instead define the evaluation at a point $(x,K)$ by working in the complete diagram with  parameters a tuple of constants for the tuple of coordinates of $x$.  More precisely, let $\mcL_{\DP,\psf,k}(x)$ be the expansion of $\mcL_{\DP,\psf,k}$ by a constant symbol for each coordinate of $x$. Then they consider fields $K'$ such that $(K'\llp t \rrp, K', \ZZ)$ is elementary equivalent to $(K\llp t \rrp, K, \ZZ)$ in the language $\mcL_{\DP,\psf,k}(x)$. Since this theory is an expansion of ours in the $\mcL_{\DP,\psf,k,x}$-language, our notion of evaluation will determine the notion of evaluation in \cite{eval}.
\end{rmk}

Hence the following proposition follows readily from \cite[Theorem 1]{eval}:

\begin{proposition}
\label{prop-eval}
A constructible motivic function is determined by its evaluation at points.  More precisely, let $S\in \GDef_{\psf,k}$ and $\varphi,\psi\in \cC_{\psf,k}(S)$. Then $\varphi=\psi$ if and only if, for every $(x,K)\in \abs{S}$, 
$\varphi(x)=\psi(x)$ in $\cC_{\psf,k(x)}(*_{k(x)})$.
\end{proposition}
\begin{proof}
Let $\varphi,\psi\in \cC_{\psf,k}(S)$ such that for every $(x,K)\in \abs{S}$, $\varphi(x)=\psi(x)\in \cC_{\psf,k(x)}(*_{k(x)})$. The evaluation of a constructible motivic function $\psi$ at a point $(x,K)\in \abs{S}$ in the sense of \cite{eval} is defined as $i_x^*(\psi)$, where $i_x\colon \set{x}\to S$ is the inclusion, which is definable in $\mcL_{\DP,\psf,k}(x)$. Using the notations $V_x, g$ in the definition of evaluation, since $x\in V_x(K)$, we have $i_x^*(\psi)=i_x^*(\psi_{\mid V_x})=i_x^*g^*\psi_x$. Hence the evaluations (in the sense of \cite{eval}) of $\varphi$ and $\psi$ at every point $(x,K)\in \abs{S}$ are equal, therefore by \cite[Theorem 1]{eval}, $\varphi=\psi$ in $\cC_{\psf,k}(S)$.
\end{proof}

\subsection{Definable varieties}
We shall need the following notions.
\begin{definition}
\label{def:defvar}
A \emph{definable algebraic variety} $\wX$ over $k$ is a definable set $\wX\subseteq \BP^n\times \BA^m$ in $\RGDef_{\psf,k}$ such that for every pseudo-finite field $K$ and $a\in \pi(\wX(K))\subseteq K^m$, where $\pi$ is the coordinate projection, the fiber $\wX_a$ is the set of $K$-points of a quasi-projective variety with a bounded number of equations of bounded degrees and parameters in $a$.

We shall usually consider such $\wX$ as a definable in $\BP^n$ with parameters $\pi(\wX)$. In particular, by ``$\wX$ is of pure dimension $d$" we mean that the fibers $\wX_a$ are points of varieties of pure dimension $d$.

Define similarly a \emph{definable algebraic group}, requiring in addition that the group law and inverse are the graphs of definable functions.

Similarly, a \emph{definable algebraic variety $\wX$ over $k\llp t\rrp$} is a definable $\wX\subseteq \BP^n\times \BA^m$ in $\GDef_{\psf,k}$ such that for every pseudo-finite field $K$ and $a\in \pi(\wX(K))\subseteq K\llp t\rrp^m$, where $\pi$ is the coordinate projection, the fiber $\wX_a$ is the set of $K\llp t\rrp$-points of a quasi-projective variety with a bounded number of equations of bounded degrees and parameters in $a$. 

Define also a  \emph{definable algebraic variety $\wX$ over $k\llb t\rrb$} in a similar way, but asking the fibers $\wX_a$ to be the set of $K\llb t\rrb$-points of a quasi-projective variety over $K\llb t\rrb$. We can associate to it definable algebraic variety $\wX_{k\llp t\rrp}$ over $k\llp t\rrp$ by simply taking the $K\llp t\rrp$-points of the fibers, and a definable inclusion $\wX\to \wX_{k\llp t\rrp}$ and a definable algebraic variety $\wX_{k}$ over $k$ by taking the residue map, together with a surjection $\wX\to \wX_k$.
\end{definition}

\begin{example}
For example, there is a definable algebraic group $\wU_1$ representing the unitary group:  \[\wU_1=\set{(x_1,x_2,a)\in \BA^2\times \BA\mid x_1^2-ax_2^2=1, a \text{  not a square }}.\] 
This definable set is not the same as the definable set associated to the unitary group $U_1$ over $k$:
\[
U_{1}=\set{(x_1,x_2)\in \BA^2\mid x_1^2-a_0x_2^2=1},
\]
where $a_0$ is not a square in $k$.
We have  $U_1(K)=K^\times$ whenever $K$ is an extension of $k$ where $a_0$ is a square, whereas $\wU_{1,a}(K)\neq K^\times$.
\end{example}

\subsection{Fubini theorem}
We say that $S\in \GDef_{\psf,k}$ is of dimension $d$ if $\mathrm{Kdim}(S)=d$, in the sense of \cite[Section 3]{CL-2008}. Roughly speaking, this means that the valued field part of $S$ is of dimension $d$. Fix $S\in \Def_{\psf,k}$ of dimension $d$ and let $\varphi\in \cC_{\psf}(S)$. Let $\abs{\omega_0}_S$ be the canonical volume form, in the sense of \cite[15.1]{CL-2008}. Let $C_\psf^d(S)$ the quotient of $\cC_\psf(S)$ by the ideal of constructible functions with support included in a definable of dimension $d-1$. If the class $[\varphi]\in C_\psf^d(S)$ of $\varphi$ is integrable, then $\int_S [\varphi] \abs{\omega_0}_S$ is defined in \cite[15.1]{CL-2008} and belongs to $\cC_{\psf}(*_k)$. Say that $\varphi\in \cC_\psf(S)$ is integrable if $[\varphi]$ is integrable, and set
\[
\int_S \varphi \abs{\omega_0}_S=\int_S [\varphi] \abs{\omega_0}_S.
\]

Let $\wX$ be a smooth  definable algebraic variety over $k\dpar{t}$ of pure dimension $d$. Denote by $\abs{\omega_\wX}$ the definable volume form on $\wX$ determined by the algebraic differential form of degree $d$ on $\wX_a(K\dpar{t})$ for every pseudo-finite field. Let $\varphi\in \cC_\psf(\wX)$.  We  say that $\varphi\abs{\omega_\wX}$ is integrable if $[\varphi]\abs{\omega_\wX}$ is integrable and set 
\[
\int_{\wX} \varphi \abs{\omega_\wX}=\int_{\wX} [\varphi] \abs{\omega_\wX}.
\]

We will use the following versions of the Fubini theorem, which follows from Theorem 10.1.1, Theorem 15.2.1 and Proposition 15.4.1 in \cite{CL-2008}.

\begin{proposition}
\label{prop-CL-fubini}The following properties hold:
\begin{enumerate}
\item Let $\wX$ be a smooth definable algebraic variety over $k\dpar{t}$ of pure dimension and $\wY$ a definable algebraic variety over $k$. Let $f \colon \wX \to \wY$ be a morphism in $\GDef_{\psf,k}$. Let $\abs{\omega_\wX}$ be the volume form associated to a top degree form on $\wX$. Let $\varphi\in \cC_\psf(\wX)$ such that $\varphi\abs{\omega_\wX}$ is integrable.  For every point $y\in \abs{\wY}$, denote by $\wX_y$ the fiber of $f$ over $y$. Then for every $y\in \abs{\wY}$,  $\varphi_{\vert \wX_y}\abs{\omega_\wX}_{\vert \wX_y}$ is integrable on $\wX_y$ and there exists a constructible function $\psi\in \cC_\psf(\wY)$ such that for every $y\in \abs{\wY}$, 
\[
\psi(y)= \int_{\wX_y} \varphi_{\vert \wX_y}\abs{\omega_X}_{\vert \wX_y}
\]
and
\[
\int_{\wX} \varphi \abs{\omega_\wX}=\int_{\wY} \psi.
\]
\item Let $f \colon \wX\to \wY$ be a smooth definable morphism between smooth definable $k\dpar{t}$-varieties of pure dimension. Let $\abs{\omega_\wX}$ and $\abs{\omega_\wY}$ be the volume forms on $\wX$ and $\wY$ Let $\varphi\in \cC_\psf(\wX)$ such that $\varphi \abs{\omega_\wX}$ is integrable. Then for every $y\in \abs{\wY}$, $\varphi_{\vert \wX_y}\abs{\omega_\wX/f^*(\omega_\wY)}_{\vert \wX_y}$ is integrable on $\wX_y$ and there exists a constructible function $\psi\in \cC_\psf(\wY)$ such that for every $y\in \abs{\wY}$, 
\[
\psi(y)=\int_{\wX_y}\varphi_{\vert \wX_y}\abs{\omega_\wX/f^*(\omega_\wY)}_{\vert \wX_y}
\]
and
\[
\int_{\wX} \varphi \abs{\omega_\wX}=\int_{\wY} \psi\abs{\omega_\wY}.
\]
\end{enumerate}
\end{proposition}
Note that the functions $\psi$ in the proposition are uniquely determined, since a constructible function is determined by its evaluation at points by Proposition \ref{prop-eval}.

\subsection{Constructible functions with values in motives}
\label{sec-const-mot}
Fix $S\in \GDef_{\psf,k}$ and $x\in \abs{S}$. 
Consider the morphism 
\[
\chi_{\psf,k(x)}\colon \K{\RDef_{\psf,k(x)}}\longrightarrow \K{\DAC(\Spec k(x),\Lambda)}\otimes \QQ.
\]

Recall that 
\[\cC_\psf(*_{k(x)})=\K{\RDef_{\psf,k(x)}}\otimes_{\ZZ[\eL]} \mathbb{A}.
\]
The morphism $\chi_{\psf,k(x)}$ then induces a morphism
\[
\chi_{\psf,k(x)}\colon \cC_\psf^{x\in S}(*_{k(x)})\longrightarrow \K{\DAC(\Spec k(x),\Lambda)}\otimes_{\ZZ[\eL]} \mathbb{A}\otimes \QQ.
\]

We now have a natural morphism 
\[
\vartheta_S\colon \cC_\psf(S)\longrightarrow \prod_{x\in \abs{S}} \K{\DAC(\Spec k(x),\Lambda)}\otimes_{\ZZ[\eL]} \mathbb{A}\otimes \QQ,
\]
which sends $\psi\in \cC_{\psf}(S)$ to $ \left(\chi_{\psf,k(x)}(\psi(x))\right)_{x\in \abs{S}}$. Set $\cC_{\mot}(S)=\vartheta_S(\cC_\psf(S))$ and still denote by $\vartheta_S$ the induced morphism.

Let $X$ be a quasi-projective $k$-scheme, we still write $X$ for the associated definable subassignment.
\begin{rmk}Note
that working with the theory of pseudo-finite fields is  crucial in this paper and is quite different from working with the theory
of algebraically closed fields as in
\cite{LW19}. Indeed, given $n$ a positive integer,  if $Z_n$ is the set defined by the formula
$x \not=0 \wedge \exists y (x = y^n)$ in $\mathbb{A}^1_k$,
we have $\chi_{\psf} (Z_n) = \frac{\mathbb{L} - 1}{n}$, while, with the notation from loc. cit.,
$\chi_{\acf} (Z_n) = \mathbb{L} - 1$. This reflects the fact that in a finite field of characteristic prime to $n$, non-zero
$n$th powers have density $1/n$.
As a consequence, motivic functions considered here are quite different from the ones used in loc. cit.
\end{rmk}

Let $X$ be a quasi-projective $k$-scheme. By definition the ring $\cC_\psf(X)$ is equal to $\K{\RDef_{X}}\otimes_{\ZZ[\eL]}\mathbb{A}$.  
\begin{lemma}\label{lem-cCmoti}
Let $X$ be a quasi-projective $k$-scheme, and $X$ the associated definable subassignment. The map sending $\vartheta_{X}(\varphi)\in \cC_\mot(X)$ to $\chi_{\psf,X}(\varphi)$ is well-defined and provides an isomorphism between 
$\cC_\mot(X)$ and the image of
\begin{equation}
\label{eqn-cCmot-quasiproj}
\chi_{\psf,X}\colon \K{\RDef_{\psf,X}}\otimes_{\ZZ[\eL]}\mathbb{A}\longrightarrow \K{\DAC(X,\Lambda)}\otimes_{\ZZ[\eL]}\mathbb{A}\otimes \QQ.
\end{equation}
\end{lemma}
\begin{proof}
Fix $\varphi, \psi\in \cC_\psf(X)$ such that $\vartheta_{X}(\varphi)=\vartheta_{X}(\psi)$. By Lemma \ref{lem-comp-chipsf-proj}, we have that the pullbacks of $\chi_{\psf,X}(\varphi)$ and $\chi_{\psf,X}(\psi)$ to every point $x\in \abs{X}$ are equal. Since points in pseudo-finite fields cover all scheme points, by  Lemma \ref{evmot} we have 
\[
\chi_{\psf,X}(\varphi)=\chi_{\psf,X}(\psi).\qedhere
\]
\end{proof}

Let $p\colon X\to Y$ be a morphism between quasi-projective $k$-schemes.  Using Lemma \ref{lem-cCmoti}, Lemma \ref{lem-comp-chipsf-proj} can be restated  in terms of  commutative diagrams
\begin{equation}
\label{eqn:diagpf}
\xymatrixcolsep{1pc}
\xymatrix{
\cC_\psf(X) \ar[d]_{p_!} \ar[rr]^-{\vartheta_{X}} &&\cC_\mot(X)  \ar[d]^{p_!} \\
\cC_\psf(Y) \ar[rr]_{\vartheta_{Y}} &&\cC_\mot(Y)
}
\end{equation}
and
\begin{equation}
\xymatrixcolsep{1pc}
\xymatrix{
\cC_\psf(Y) \ar[d]_{p^*} \ar[rr]^-{\vartheta_{Y}} &&\cC_\mot(Y)  \ar[d]^{p^*} \\
\cC_\psf(X) \ar[rr]_{\vartheta_{X}} &&\cC_\mot(X).
}
\end{equation}

\begin{propdef}
\label{propdef-intmot}
Let $S\in \Def_{\psf,k}$ and $\varphi,\varphi'\in \cC_{\psf}(S)$. Assume that $\varphi$ and $\varphi'$ are integrable and that $\vartheta_S(\varphi)=\vartheta_S(\varphi')$. Then 
\[
\vartheta_{*_k}\left( \int_S\varphi \abs{\omega_0}_S\right)=\vartheta_{*_k}\left( \int_S\varphi' \abs{\omega_0}_S\right).
\]
In this case, we say that $\vartheta_S(\varphi)$ is integrable and set
\[
\int_S^\mot\vartheta_S(\varphi)\abs{\omega_0}_S=\vartheta_{*_k}\left( \int_S\varphi \abs{\omega_0}_S\right).
\]
\end{propdef}
\begin{proof}
Let $S\in \GDef_{\psf,k}$ be of dimension $d$. Let $\cC^{\leq k}_\psf(S)$ be the subring of $\cC_\psf(S)$ consisting of constructible functions with support contained in a definable subassignment of dimension at most $k$, and $C^{\leq k}_\psf(S)=\cC^{\leq k}_\psf(S)/\cC^{\leq {k-1}}_\psf(S)$. Let $C_\psf(S)=\bigoplus_{k\geq 0} C^k_\psf(S)$. 
Recall from \cite[Section 10]{CL-2008} that the integral is constructed by defining for each morphism $f \colon X\to Y$ in $\Def_{\psf,k}$ a group homomorphism $f_!\colon IC_\psf(X)\to C_\psf(Y)$, where $IC_\psf(X)$ is the subgroup of $C_\psf(X)$ consisting of integrable functions. The integral of $[\varphi]\in C_\psf(X)$ is then defined as $f_![\varphi]$, where $f \colon X\to *_k$ is the map to the final object. To prove Proposition \ref{propdef-intmot}, it is thus enough to prove that for every definable map $f \colon X\to Y$, $f_!\colon IC_\psf(X)\to C_\psf(Y)$ induces a map $IC_\mot(X)\to C_\mot(Y)$, where $IC_\mot(X)$ is the image of $IC_\psf(X)$ by $\vartheta_X$. 

If $f$ is a bijection or an injection, this is clear. Hence using the Denef-Pas cell decomposition theorem, it suffices to treat the cases where $f$ is either a projection of residue field variables, value group variables, or one valued field variable. Hence we can assume $f\colon Y\times W\to Y$, where $W$ is either the residue field, the value group, or the valued field.  

By the main theorem of \cite{CR-cpb}, integration commutes with base change, in the sense that for every $x\in \abs{Y}$ and $\varphi\in C_\psf(Y\times W)$, we have $f_!(\varphi)(x)=f_{\vert \set{x}\times W !}(\varphi_{\vert \set{x}\times W})$. Hence up to replacing $k$ by $k(x)$, we can assume $Y=*_{k}$. When $W$ is the value group, $f_!$ is defined by counting, hence the result is clear. When $W$ is the residue field, $f_!$ is induced the the natural map $f_! \colon\K{\RDef_{\psf,W}}\to \K{\RDef_{\psf,k}}$ obtained by composition. Hence using Lemma \ref{lem-comp-chipsf-proj} and \ref{eqn-cCmot-quasiproj}, the result follows from the commutative diagram \ref{eqn:diagpf}.

To treat the last case, using again cell decomposition we can assume that $W$ is a cell adapted to $\varphi$, which means that $W$ is a ball and there exists a function $\psi\in C_\psf(*_{k})$ such that $f^*\psi=\varphi$. In this case, $f_!\varphi$ is defined as $\psi \eL^{-\alpha}$, where $\alpha$ is the valuative radius of the ball $W$. In this situation, it is again clear that $f_!$ induces a map $f_!\colon IC_\mot(W)\to C_\mot(*_k)$.
\end{proof}

The proof of Proposition \ref{propdef-intmot} implies that for every morphism $f\colon S\to S'$ in $\GDef_{\psf,k}$, the following diagrams commute:
\begin{equation}
\label{diag-fshriek-psf-mot}
\xymatrixcolsep{1pc}
\xymatrix{
IC_\psf(S) \ar[d]_{f_!} \ar[rr]^-{\vartheta_{S}} &&IC_\mot(S)  \ar[d]^{f_!} \\
\cC_\psf(S') \ar[rr]_{\vartheta_{S'}} &&\cC_\mot(S')
}
\end{equation}
and
\begin{equation}
\label{diag-fpull-psf-mot}
\xymatrixcolsep{1pc}
\xymatrix{
\cC_\psf(S') \ar[d]_{f^*} \ar[rr]^-{\vartheta_{S'}} &&\cC_\mot(S')  \ar[d]^{p^*} \\
\cC_\psf(S) \ar[rr]_{\vartheta_{S}} &&\cC_\mot(S).
}
\end{equation}

Let $\wX$ be a smooth definable algebraic $k\dpar{t}$-variety of pure dimension and $\abs{\omega_\wX}$ the associated volume form on $\wX$. Fix $\varphi\in \cC_\psf(\wX)$ such that $\varphi\abs{\omega_\wX}$ is integrable. Using Proposition \ref{propdef-intmot} on affine charts, we see that $\vartheta_{*_k}(\int_{\wX} \varphi \abs{\omega_\wX})$ depends only on $\vartheta_{\wX}(\varphi)$. In this case, we say that $\vartheta_{\wX}(\varphi)\abs{\omega_\wX}$ is integrable and we set
\[
\int_{\wX}^\mot\vartheta_{\wX}(\varphi)\abs{\omega_\wX} =\vartheta_{*_k} \Bigl(\int_{\wX} \varphi \abs{\omega_\wX}\Bigr).
\]

Proposition \ref{prop-CL-fubini} implies the following version of Fubini:
\begin{proposition}
\label{prop-fubini-mot}The following properties hold:
\begin{enumerate}
\item Let $\wX$ be a smooth definable algebraic variety over $k\dpar{t}$ of pure dimension and $\wY$ a definable algebraic variety over $k$. Let $f \colon \wX\to \wY$ be a morphism in $\GDef_{\psf,k}$. Let $\abs{\omega_\wX}$ be the associated volume form on $\wX$. Let $\varphi\in \cC_\mot(\wX)$ such that $\varphi\abs{\omega_\wX}$ is integrable.  For every point $y\in \abs{\wY}$, denote by $\wX_y$ the fiber of $f$ at $y$. Then for every $y\in \abs{\wY}$,  $\varphi_{\vert \wX_y}\abs{\omega_\wX}_{\vert \wX_y}$ is integrable on $\wX_y$ and there exists a constructible function $\psi\in \cC_\mot(Y)$ such that for every $y\in \abs{\wY}$, 
\[
\psi(y)= \int_{\wX_y}^\mot \varphi_{\vert \wX_y}\abs{\omega_\wX}_{\vert \wX_y}
\]
and
\[
\int_{\wX}^\mot \varphi \abs{\omega_\wX}=\int_{\wY}^\mot \psi.
\]
\item Let $f \colon \wX\to \wY$ be a smooth morphism between smooth definable $k\dpar{t}$-varieties of pure dimension. Let $\abs{\omega_\wX}$ and $\abs{\omega_\wY}$ be volume forms on $\wX$ and $\wY$. Let $f\colon \wX\to \wY$ be the induced morphism in $\GDef_{\psf,k}$. Let $\varphi\in \cC_\mot(\wX)$ such that $\varphi \abs{\omega_\wX}$ is integrable. Then for every $y\in \abs{\wY}$, $\varphi_{\vert \wX_y}\abs{\omega_\wX/f^*(\omega_\wY)}_{\vert \wX_y}$ is integrable on $\wX_y$ and there exists a constructible function $\psi\in \cC_\mot(\wY)$ such that for every $y\in \abs{\wY}$, 
\[
\psi(y)=\int_{\wX_y}^\mot\varphi_{\vert \wX_y}\abs{\omega_\wX/f^*(\omega_\wY)}_{\vert \wX_y}
\]
and
\[
\int_{\wX}^\mot \varphi \abs{\omega_\wX}=\int_{\wY}^\mot \psi\abs{\omega_\wY}.
\]
\end{enumerate}
\end{proposition}

\subsection{Specialization}
We assume here that the base field $k$ is the field of fractions of a normal domain $R$ which is of finite type over $\BZ$. Let $x\in U=\Spec(R)$ be a closed point and $\BF_x$ the residue field at $x$, which is a finite field. 

Let $\varphi$ be a ring formula with parameters in $k$ and $n$ free variables. Since $k$ is the field of fractions of $R$, we can consider that the parameters of $\varphi$ are in $R$, hence consider $\varphi(\BF_x)\subseteq \BF_x^n$, the set of elements in $\BF_x^n$ satisfying $\varphi$. 

\begin{proposition}[\cite{DL-PSF}]
\label{prop:spec-for}
If $\varphi$ and $\varphi'$ are two formulas, we have $\varphi(\BF_x)=\varphi'(\BF_x)$ for every $x$ in a non-empty open subset of $U$ if and only if  $\varphi$ and $\varphi'$ are equivalent in the theory of pseudo-finite fields.
\end{proposition}

From now on, we suppose that the field $\Lambda$ of coefficients for motives is $\Qlb$.
Let $G_k$ be the absolute Galois group of $k$ and $\K{\Qlb,G_k}$ the Grothendieck group of $\Qlb$-vector spaces endowed with a continuous $G_k$-action, for some prime $\ell$. Consider the $\ell$-adic realization map $\mathrm{Et}_\ell \colon\K{\DAC(\Spec k,\Qlb)}\to \K{\Qlb,G_k}$.

For $x\in U$, let $\Frob_x$ the geometric Frobenius at $x$. Taking its trace on the invariants by inertia defines a ring morphism $\K{\Qlb,G_k}\to \Qlb$.  Precomposing with $\mathrm{Et}_\ell$, we get a ring morphism
\[
\Tr \Frob_x\colon \K{\DAC(\Spec k,\Qlb)}\otimes \QQ\longrightarrow \Qlb.
\]
By \cite[Proposition 3.3.1]{DL-PSF}, this morphism is compatible with specialization:
\begin{proposition}\label{prop:spec-mot} Let $k$ be the field of fractions of a normal domain $R$ which is of finite type over $\BZ$, and $U=\Spec(R)$. Let $X$ be a definable set in the theory of pseudo-finite fields with parameters in $k$. There exists a non-empty open subset of $U$ such that, for every closed point $x$ in $U$, 
\[
\Tr \Frob_x \Bigl( \chi_\psf([X]) \Bigr)=\# X(\BF_x).
\]
\end{proposition}

Consider now $X\in \GDef_{\psf,R\llb t\rrb}$, and $\varphi\in \FC_\mot(X)$. For every non-archimedean local field $L$ with a map $R\llb t\rrb\to L$ sending $t$ to a uniformizer of $L$, we can consider the $L$-points of $X$, which do not depend on the precise choice of a formula, if we restrict the residue field of $L$  to be the residue field of a closed point of a non-empty open subscheme of $U$.

Given $\varphi\in \FC_{\mot}(X)$, for every $L$ as above, with residue field $\BF_x$ in an open subscheme of $U$, by applying $\Tr \Frob_x$ to $\varphi$, we get a map $\varphi_L\colon X(L)\to \Qlb$. If $\abs{\omega}$ is a volume form on $X$, similarly we get a volume form $\abs{\omega}_L$ on $X(L)$ for the Haar measure on $L$.

Combining the above proposition with the specialization of integrals of Cluckers-Loeser \cite[Theorem 9.1.4]{cl_annals}, we get a specialization principle for motivic integrals:

\begin{proposition}
\label{prop:spec-int}
Let $k$ be the field of fractions of a normal domain $R$ which is of finite type over $\BZ$, and $U=\Spec(R)$. Let $X\in \GDef_{\psf,R\llb t\rrb}$, and $\varphi\in \FC_\mot(X)$. Assume that $\varphi$ is integrable with respect to a volume form $\abs{\omega}$ on $X$. Then for every $x$ in a non-empty open subset of $U$, and for every local field $L$ with a map $R\llb t\rrb\to L$ sending $t$ to a uniformizer of $L$ and residue field $\BF_x$, 
$\varphi_L$ is integrable with respect to $\abs{\omega}_L$ and 
\[
\Tr\Frob_x \Bigr( \int^\mot_X \varphi\abs{\omega}\Bigr)=\int_{X(L)} \varphi_L\abs{\omega}_L,
\]
where the integral on the right hand side is for the Haar measure on $L$.
\end{proposition}

\subsection{Fourier transform}
\label{sec:ft}

In this section we show that given a finite abelian group $G$, one can define at the level of motives a Fourier transform such that given a variety $X$ with an action of $G$, the function of the twisted varieties of $X$ is sent to the function mapping a character $\kappa$ to the $\kappa$-isotypical component of the motive of $X$. This motivates our Definition \ref{def:chi:iso:mcP} of isotypical component. 

Write $G$ as a product of cyclic groups $G=\ZZ/n_1\ZZ\times\dots\times \ZZ/n_r\ZZ$.

Let $T_i\to B_i$ be a $\ZZ/n_i\ZZ$-torsor over a base $B_i$ which is versal in the sense of~\cite[Definition 5.1]{serre-cohinv}. Let $S_i\subseteq B_i$ be the definable subset consisting of the points $x\in B_i$ such that the fiber of $T_i$ over $x$ is a primitive $\ZZ/n_i\ZZ$-torsor, \emph{i.e.} is a torsor of order $n_i$. Since $T_i$ is versal, $S_i(K)$ is non-empty for every pseudo-finite field $K$. 

The product $T=T_1\times \dots \times T_r$ is a $G$-torsor over $B=B_1\times \dots \times B_r$ and over $x\in S=S_1\times \dots \times S_r$, it is a primitive $G$-torsor.

Let $X$ be a variety over $B$ with a $G$-action. Define the twist of $X$ by $T$ as $X^T=X\times_B T/G$, where the quotient is by the anti-diagonal action.

In the special case where $X$ is a $G$-torsor, this corresponds to addition in $\rmH^1(B,G)$. 

For $e=(e_1,\dots,e_r)\in \BN^r$, define by induction $T_i^{e_i}$ by the relation $T_i^{j+1}=({T_i^{j})}^{T_i}$, and $T^{e}=T_1^{e_1}\times\dots \times T_r^{e_r}$.
Since the group $\rmH^1(B_i,\ZZ/n_i\ZZ)$ is torsion of order divisible by $n_i$, $T_i^{e_i}$ depends only on the class of $e_i$ modulo $n_i$. For $e\in G$, we then define $T^e$ as $T^{e'}$ where $e'\in \BN^r$ is a representative of $e$.

Let $G^*$ be the group of characters of $G$ with values in $\Lambda$.
\begin{definition}
Given a map $f\colon G \to \DAC(B,\Lambda)$, we define its Fourier transform as 
\[\widehat f\colon \chi\in  G^*\longmapsto \frac{1}{\abs{G}}\sum_{e\in G} f(e)\: \chi_{B}(T^e,\chi^{-1}).
\]
\end{definition}

Let $X\in \Var_{B}^G$.  
Consider the map $f_X\colon e\in G\mapsto \chi_{B}(X^{T^e})$. We will show that its Fourier transform is equal to isotypical components of the motive of $X$, once multiplied by some idempotent.

Consider the motive  $\chi_{B}(X^T)\in\K{\DAC(B)}$ of $X^{T}$, together with its $G$-action. 
Let $\chi\in G^*$ be a character of $G$, and $\rho\colon (x,y)\in G\times G \mapsto xy^{-1}\in G$. Using Proposition \ref{prop-chialpha-change-gps} (2) and (1), since $G$ is commutative, we have that 
\begin{equation}
\label{eqn-chi-iso-tordu}
\chi_{B}(X^T,\chi)=\chi_{B}(X\times_{B} T,\chi\circ \rho)=\chi_{B}(X,\chi) \cdot \chi_{B}(T,\chi) 
\end{equation}
and
\begin{equation}
\label{eqn-chi-iso-tordu2}
\chi_{B}(T^e,\chi)=\chi_{B}(T,\chi^e),
\end{equation}
where $\chi^e=\chi_1^{e_1}\dots\chi_r^{e_r}$ and $\chi_i$ is a character of $\BZ/n_i\BZ$. 

By Proposition~\ref{prop-chialpha} (2), we have $f_X(e)=\chi_{B}(X^{T^e})=\sum_{\psi\in G^*}\chi_{B}(X^{T^e},\psi)$.
Injecting this relation, we have for $\chi\in G^*$, 
\begin{align*}
\widehat f_X(\chi)&=\frac{1}{\abs{G}}\sum_{e\in G}\sum_{\psi\in G^*}\chi_{B}(X^{T^e},\psi) \: \chi_{B}(T^e,\chi^{-1})\\
&=\frac{1}{\abs{G}}\sum_{\psi\in G^*} \sum_{e\in G}\: \chi_{B}(X,\psi) \cdot \chi_{B}(T^e,\psi) \: \chi_{B}(T^e,\chi^{-1}) \: \ \ (\text{using } (\ref{eqn-chi-iso-tordu}))\\
&=\frac{1}{\abs{G}}\sum_{\psi\in G^*}\: \chi_{B}(X,\psi)\sum_{e\in G} \: \chi_{B}(T^e,\psi\chi^{-1})\\
&=\chi_{B}(X,\chi)\: \chi_{B}(T,1)+\frac{1}{\abs{G}}\sum_{\psi\in G^*, \psi\neq 1}\: \chi_{B}(X,\psi\chi)\sum_{e\in G}\: \chi_{B}(T,\psi^e),
\end{align*}
where the last line used (\ref{eqn-chi-iso-tordu2}). 
The second term is non-zero, but we will show that by restricting to a suitable definable set, it vanishes.

Recall that $S$ is the definable subset of $B$ consisting of points $x\in B$ over which $T$ is a primitive $G$-torsor.

Consider $\chi_{\psf,B}([S])\in \K{\DAC(B)}\otimes \BQ$. For example, if $r=1$ and $n_1$ is prime, we have 
\[
\chi_{\psf,B}([S])=\chi_{B}(T,1)-\frac{1}{\abs{G}}\sum_{\chi\in G^*} \: \chi_{B}(T,\chi).
\]
Since $S\subseteq B$, $S\times_{B} S=S$, the element $\chi_{\psf,B}([S])$ is an idempotent of $\K{\DAC(B)}\otimes \BQ$.

For every strict subgroup $H$ of $G$, since the fiber of $T$ over $x\in S$ is a primitive $G$-torsor, the fiber of $T/H$ over $x$ is not the trivial torsor, hence $S\times_B (T/H)(K)=\emptyset$ for every pseudo-finite field $K$. Thus the product of the classes of $S$ and $T/H$ in $\K{\RDef_{\psf,B}}$ is 0.  Thus we have 
\begin{equation}\label{eqn-AtimesT}
\chi_{\psf,B}([S])\cdot \chi_{B}([T/H])=0
\end{equation} in $\K{\DAC(B)}\otimes \BQ$ or $\cC_\mot(B)$. 

Let $\psi$ be a character of $G$. There exists a (possibly trivial) subgroup $H$ of $G$, and a primitive character $\bar \psi$ of $G/H$ such that $\psi=\bar \psi\circ \rho$, where $\rho\colon G\to G/H$ is the quotient morphism. Using Proposition \ref{prop-chialpha-change-gps}(2), we get that
\[
\chi_{B}(T,1)=\chi_{B}(T/G)=[\un_{B}]=1 
\]
in $\K{\DAC(B,\Lambda)}\otimes \BQ$.

Assume now that $\psi$ is not the trivial character. Using again Proposition \ref{prop-chialpha-change-gps}(2), we also have
\[
\sum_{e\in G}\chi_{B}(T,\psi^e)=\abs{H}\sum_{\bar e\in G/H} \chi_{B}(T/H,\bar \psi^{\bar e}).
\]
Since $\bar \psi$ is a primitive character of $G/H$, $\psi^{\bar e}$ varies over all characters of $G/H$, hence by Proposition \ref{prop-chialpha}(2), 
\[
\sum_{e\in G}\chi_{B}(T,\psi^e)=\abs{H}\chi_{B}(T/H).
\]
Since $H\neq G$, from (\ref{eqn-AtimesT}) we deduce that 
\begin{equation}
\label{eqn:sumchar}
\chi_{\psf,B}([S])\cdot\sum_{e\in G}\chi_{B}(T, \psi^e)=0,
\end{equation}
and finally
\[
\widehat f_X(\chi)\cdot\chi_{\psf,B}([S])=\chi_{B}(X,\chi)\cdot\chi_{\psf,B}([S]). 
\]
We thus have proven:
\begin{proposition}
\label{lem:FT:mu:n}
The Fourier transform of the map
\[
e\in G\longmapsto \chi_{B}(X^{T^e})\cdot\chi_{\psf,B}([S])
\]
is the map
\[
\chi\in  G^* \longmapsto \chi_{B}(X,\chi)\cdot\chi_{\psf,B}([S]).
\]
In particular, we have
\[
\frac{1}{\abs{G}}\sum_{e\in G}\chi_{B}(X^{T^e})\cdot\chi_{\psf,B}([S])=\chi_{B}(X,1)\cdot\chi_{\psf,B}([S]).
\]
\end{proposition}


The preceding discussion took place inside $\K{\DAC (B,\Lambda)}\otimes \BQ$, but carries over verbatim to $\cC_\mot(B)$, using $(\ref{eqn-cCmot-quasiproj})$.

\begin{rmk}
Note that the above discussion is not vacuous, since even if $\chi_{B}(T)\cdot\chi_{\psf,B}([S])=0$, $\chi_{B}(T,\psi)\cdot\chi_{\psf,B}([S])$ is not zero. The underlying reason for that being that $\chi_{B}(-,\psi)$ is a group morphism but not a ring morphism. 
\end{rmk}

\subsection{Motivic characters}
\label{sec:motchar}
Let $G$ be a constant commutative group over $k$. Given the above discussion on Fourier transform, we shall consider a motivic version of the characters of $G$, as functions in $\cC_\mot(G,\rel)$, as follows. For a character $\alpha\in G^*$ and $g\in G$, we then define $\alpha_\mot(g)=\chi_{\rel}({T^g},\alpha)\cdot \chi_{\psf,\rel}([S])$, where $T$ and $S$ are as in the previous section. A general motivic character is a product of $e$ copies of $\alpha_\mot$, for some integer $e$. Up to the choice of a primitive character of $G$, we get a bijection between the set of characters of $G$, and the set of motivic characters of $G$.

We recover the classical result about the sum of values of a non-trivial character.
\begin{lemma}
\label{sumchar0}
Let $\alpha \in G^*$ be a non-trivial character $G$. Then
\[
\sum_{x\in G} \alpha_\mot(g)=0.
\]
\end{lemma}
\begin{proof}
This is Equation (\ref{eqn:sumchar}).
\end{proof}

\subsection{Orbifold volume}
\label{sec:orbifold}
We need to consider Deligne-Mumford stacks in a definable way. We shall treat them using atlases, but we need to introduce a notion of definable groupoid beforehand. 

\begin{definition}
A definable groupoid is the data of two definable sets $V$ and $A$, as well as definable maps $s, t\colon A \to V$, $c\colon C\subseteq A\times A\to A$, $e\colon V\to A$, $i\colon A\to A$. Here, $V$ and $A$ are respectively the vertices and arrows, $s$ the source, $t$ the target, $c$ the composition, defined on the set of composable maps, $e$ the identity arrows and $i$ the inverse. Those are required to satisfy the usual compatibilities.

A definable functor between $(V,A)$ and $(V',A')$ is the data of definable bijections $V\simeq V'$ and $A\simeq A'$ respecting the structure of of definable groupoid.
\end{definition}

Given a definable groupoid $(V,A)$, one can form the quotient space $V/A$ of $V$ by the equivalence relation given by isomorphisms. The space $V/A$ is a priori an imaginary set. 

\begin{definition}
\label{def:DMstack}
A \emph{definable smooth Deligne-Mumford stack over $k\llb t\rrb$} is given by the following data: a definable groupoid $(V,A)$ such that $V/A$ is a definable algebraic variety $\wM$ over $k\llb t\rrb$ together with a finite definable cover $\wM=\cup \wM_i$ such that the $\wM_i$ form a Zariski cover of $\wM$ into open pieces when specialized to a pseudo-finite field. We require as well that for all $i$, there is a smooth definable algebraic variety $\wX_i$ and a definable finite algebraic group $\Gamma_i$ acting on $\wX_i$ such that every orbit is contained in an affine subset, and such that the restriction of $(V,A)$ to $\wM_i$ is identified with the quotient stack $[\wX_i/\Gamma_i]\to \wX_i/\Gamma_i\simeq \wM_i$, and there is a dense open subset $\wU_i\subseteq \wM_i$ over which the action is  free.
\end{definition}

Let $\wM$ be a definable smooth Deligne-Mumford stack over $k\llb t\rrb$.
Chose the opens $\wU_i\subseteq \wM_i$ maximal with the property that the action of $\Gamma_i$ is free over $\wU_i$, which exist since $\Gamma_i$ is finite, and set $\wU=\cup \wU_i$.

Consider the definable subassignment $\wM^\natural$, determined by  
\[\wM^\natural(K)=\wM(K\llb t \rrb)\cap \wU(K\llp t \rrp),\]
for any pseudo-finite field $K/k$. One defines similarly $\wM^\natural_i$. 

After shrinking the $\wX_i$ if necessary, we may assume that their canonical bundles are trivial. We fix non-vanishing top degree forms $\omega_i$ on $X_i$, which induce, possibly pluricanonical, forms $\omega_{i,\orb}$ on $\wM^\natural_i$ as in \cite[Lemma 2.8]{GWZ18} and the measures induced by $\omega_{i,\orb}$ glue to define a definable form $\abs{\omega_{\orb}}$ on $\wM^\natural$. We shall compute 
$\int_{\wM^\natural}\abs{\omega_{\orb}}$ in terms of the twisted inertia stack of $\mcM_K$. 

Recall that given a stack $\mcM$, the inertia stack $I\mcM$ parametrizes pairs of objects in $\mcM$ together with an automorphism. When $\mcM$ is a Deligne-Mumford stack, hence has  finite automorphism groups, we have an equivalence
\[I \mcM \cong \mathrm{colim}_n\Hom(B\ZZ/n\ZZ,\mcM).\]
It turns out, that if $k$ does not contain all roots of unity, the correct space in our context is the \emph{twisted inertia stack of $\mcM$} defined as 
\[\Imu \mcM=\mathrm{colim}_n\Hom(B\mu_n,\mcM).\] 
We write abusively $\Imu M$ for the coarse moduli space of $\Imu \mcM$. Notice, that if $\mcM$ is of finite type, then the order of its automorphism groups are bounded and hence the above colimit stabilizes.

Thus given a definable Deligne-Mumford stack $\wM$, we can associate to it a definable Deligne-Mumford stack $\Imu \wM$ that parametrizes the twisted inertia stacks of $\mcM_K$ for pseudo-finite field $K$, \emph{i.e.} such that for every pseudo-finite field $K$, $\Imu \wM(K)=\Imu M_K(K)$. 

We now define the specialization map 
\begin{equation}\label{spmap} e\colon \wM^\natural \longrightarrow \Imu \wM.\end{equation}
Fix a pseudo-finite field $K/k$ and $x\in \wM^\natural(K)=\wM(K\llb t \rrb)\cap \wU(K\llp t \rrp)$. Let $\mcM_x$ be the pullback of $\mcM_K$ along $x$:
\[
\xymatrixcolsep{1pc}
\xymatrix{
\mcM_x \ar[d]_{} \ar[r]^-{} &\mcM \ar[d]_{} \\
\Spec(K\llb t \rrb) \ar[r]^-{x} &M,
}
\]
where $\mcM\to M$ is the map from $\mcM$ to its coarse moduli space.

Let $\widetilde \mcM_x$ be the normalization of $\Spec(K\llb t \rrb)$ in $\mcM_x$. By \cite[Proposition 2.12]{GWZ18}, there is a finite totally ramified extension $L$ of $K\llp t\rrp$ of degree $N$, with valuation ring $\mcO_L$ and residue field $k_L$, such that $[\Spec(\mcO_L)/\mu_N]\simeq \widetilde \mcM_x$ and the induced isomorphism 
\[
[\Spec(K)/\mu_N]=[\Spec(k_L)/\mu_N]\simeq \widetilde \mcM_{x,K}
\]
in the special fiber is independent of $L$ up to isomorphism. Hence we get a map $[\Spec(K)/\mu_N]\to \widetilde \mcM_{x,K}\to \mcM_{x,K}\to \mcM_K$, which gives a point $e(x)\in \Imu M_K(K)=\Imu \wM(K)$.

\begin{proposition} \label{edef}
The map $e$ is definable. 
\end{proposition}

Before starting the proof, recall from \cite[Section 2.5]{GWZ18} some facts about $\Gamma$-torsors over $\Spec(K\llp t \rrp)$, with $K$ pseudo-finite. A $\Gamma$-torsor $T=\Spec(\prod_i L_i)$ over $\Spec(K\llp t \rrp)$ is said to be unramified if each $L_i$ is unramified. It is strongly ramified if each $L_i$ is totally ramified. 

In that case, $L_i\simeq L_j$ and 
$\Gamma$ acts transitively on the connected components of $T$, which correspond to the totally ramified extensions $L_i$. Let $I_T$ be the stabilizer of $\Spec(L_1)$. It is a subgroup of $\Gamma$, well-defined up to conjugation, and isomorphic to $\mu_N$, where $N=\mathrm{deg}(L_1/ K\llp t \rrp)$. 

For any $T$,  from the description of $\mathrm{H}^1(K\llp t \rrp,\Gamma)$, there exists a (unique) unramified $\Gamma$-torsor $P$ such that the twist $T\times^\Gamma P$ is strongly ramified, see \cite[Lemmas 2.14, 2.16]{GWZ18} for details. 

Let $T$ a $\Gamma$-torsor over $\Spec(K\llp t \rrp)$ and $\widetilde T$ be the normalization of $\Spec(K\llb t \rrb)$ in $T$. From the valuative criterion of properness, $\widetilde T$ is endowed with a $\Gamma$-action. We can form the stack quotient $[\widetilde T/\Gamma]$. Let $P$ be an unramified torsor such that $T\times^\Gamma P$ is strongly ramified. Since smooth base change commutes with normalization, we have $[\widetilde T\times^\Gamma P/\Gamma_P]\simeq [\widetilde T/\Gamma]$. Since the special fiber of $[\widetilde T\times^\Gamma P/\Gamma_P]$ is $B I_T$ and $I_T \cong \mu_N$, we get a morphism 
\begin{equation}
\label{eqn-strongramtor}
{B \mu_N}_K\longrightarrow [\widetilde T/\Gamma]_K.
\end{equation}

The totally ramified torsor $T\times^\Gamma P$ can be  described more concretely as follows. Let $\gamma\colon \mu_N\to \Gamma$ be the inclusion of the inertia group of $T$ into $\Gamma$. Define a $\Gamma$-torsor by
\begin{equation}
T_\gamma= \Spec(K\llp t^{1/N} \rrp)\times_{\Spec(K\llp t \rrp)} \Gamma/\mu_N.
\end{equation}
This torsor is strongly ramified and thus must be equal to $T\times^\Gamma P$. 

\begin{proof}[Proof of Proposition \ref{edef}]
We can assume that $\wM$ is parametrizing a global quotient $[X/\Gamma]$, where $\wX$ and $\Gamma$ are as in the beginning of the section. Let $\pi\colon \wX\to \wX/\Gamma$ the quotient morphism. Recall the $\wU\subseteq \wX/\Gamma$ is defined as the maximal open subset such that $\pi\colon V=\pi^{-1}(\wU)\to \wU$ is a $\Gamma$-torsor, which amounts to define $\wV$ as the subset of $\wX$ where the $\Gamma$-action in free.  In this situation, $e$ admits the following more explicit description (see \cite[Construction 2.18]{GWZ18}). Let $x\in \wM^\natural(K)=\wM(K\llb t \rrb)\cap \wU(K\llp t \rrp)$. By definition of $\wU$, the point $x$ defines a $\Gamma$-torsor $T$ over $\Spec(K\llp t \rrp)$. Let $\widetilde T$ be the normalization of $\Spec(K\llb t \rrb)$ in $T$. By the valuative criterion for properness, one gets a $\Gamma$-equivariant map $\widetilde x\colon\widetilde T\to X$. We have the following diagram:
\[
\xymatrixcolsep{1pc}
\xymatrix{
T\ar[d]_{} \ar[r]^-{} &\widetilde T \ar[d]_{}\ar[r]^-{\widetilde x} &X \ar[d]_{\pi} \\
\Spec(K\llp t \rrp) \ar[r]_-{} &\Spec(K\llb t \rrb) \ar[r]_-{x} &X/\Gamma.
}
\]
From (\ref{eqn-strongramtor}), we then get a morphism 
\[
{B \mu_N}_K\longrightarrow [\widetilde T/\Gamma]_K\longrightarrow [X/\Gamma]_K, 
\]
hence a point in $\Imu [X/\Gamma]_K$. By \cite[Construction 2.18]{GWZ18}, this point in $\Imu M_K$ is $e(x)$. 

From the description of the inertia stack above, to show that $e$ is definable, it suffices to show that the subset of $x\in \wU$ where the totally ramified torsor $T\times^\Gamma P$ is isomorphic to $T_\gamma$, for $\gamma\colon \mu_N\to \Gamma$ (up to conjugacy), is definable. But those two torsors are isomorphic if over an unramified extension, the fiber $\pi^{-1}(x)$ is isomorphic to $T_\gamma$ as sets with a definable $\Gamma$-action. If no extension is needed, then it is clear that this expressed by a formula $\varphi_{x,\gamma}$. If one needs an extension, say of degree $m$, then one encodes the extension of degree $m$ of $K$, $K_m$, as a definable set with parameters, and use also a parameter $\tau$  for a fixed choice of generator of the Galois group $\Gal(K_m/K)$. See Section \ref{sec:def:geo:stab} below for details. Then working in $K_m$, the condition that $\pi^{-1}(x)$ is isomorphic to $T_\gamma$ as $\Gamma$-set is again definable by a formula $\varphi_{x,\gamma,m}$. Since $U$ is covered by the formulas $\varphi_{x,\gamma,m}$, by compactness it is covered by a finite number of those, which shows that $e$ is definable. 
\end{proof}

Let us recall the definition of weights from \cite[Definition 2.20]{GWZ18}.
Let $k$ be a field and let 
$(\chi_1, \cdots, \chi_r) \in (\mathbb{Q}/\mathbb{Z})^r$. Define $w (\chi_1, \cdots, \chi_r) = \sum_{i= 1}^r c_i$,
with $c_i$
the unique rational representative of $\chi_i$ with $0 < c_i\leq 1$.
For  any $r$-dimensional $k$-vector space $V$ with an algebraic action $a : \hat \mu \to \mu_N \to \Aut (V)$, 
one has a character decomposition
$V_{\bar k} \simeq \oplus_{i = 1}^r V_{\bar k} (\chi_i)$ with $\chi_i \in \mathbb{Q}/\mathbb{Z}$ characters of $\widehat \mu$ and one defines the weight of $a$ as
$w(a) =w (\chi_1, \cdots, \chi_r)$.
Let $\wM$ be a definable  smooth Deligne-Mumford stack  over $k\llb t\rrb$ and $(x, \alpha) \in \Imu \wM(K)$.
We define $w(x, \alpha)$ as the weight of $\widehat \mu$ acting on the tangent space $T_x \mcM_K$.

\begin{theorem}[Orbifold formula]
\label{th:orbi}
Let $\wM$ be a definable smooth Deligne-Mumford stack over $k\llb t\rrb$. Let $(x,\alpha)\in \Imu \wM(K)$. Then the volume of the fiber of $e$ with respect to the orbifold measure on $\wM$ is 
\[
\int_{e^{-1}(x,\alpha)} \abs{\omega_\orb}=\eL^{-w(x,\alpha)}.
\]
\end{theorem}

\begin{proof}This is the exact analogue for  motivic integrals of  \cite[Theorem 2.21]{GWZ18} which deals with $p$-adic integrals.
The proof in loc. cit., which ultimately reduces to computations done in \cite{DL2002}, adapts without difficulty to the motivic setting, using \cite[Proposition 3.2.3]{LW19} instead  of 
\cite[Lemma 4.19]{GWZ18}. Note that the apparent discrepancy  with 
  \cite[Theorem 2.21]{GWZ18} (the presence of a denominator in the formula)
is an artefact coming from the fact that in \cite{GWZ18}
 $(x,\alpha)$ is a point of $\Imu \mcM$ while here it is a point in the coarse moduli space.

The theorem also follows from a more general orbifold formula for Artin stacks, that will appear in the authors' upcoming work in preparation.
\end{proof}

\section{Geometric setup}
\label{sec:geo-setup}
In this section we introduce the relevant constructions for the formulation of the Geometric Stabilization theorem \cite[Théorème 6.4.1]{ngo:fl}. The main reference is Section 4 of \textit{loc. cit.}

Throughout this section let $k$ be a field of characteristic $0$.
Since a first order sentence in the language of rings holds in every pseudo-finite field of characteristic zero if and only it holds for all finite fields of large enough characteristic, see \cite{ax}, we shall use freely that
first order statements proved in \cite{ngo:fl,GWZ18} for  finite fields,  sometimes with a lower bound on the characteristic, will hold for any pseudo-finite field containing $k$.
We fix $X$ a smooth projective geometrically connected curve over $k$ and $D$ a line bundle of even degree $d$ on $X$. By abuse of notation we also write $D$ for the $\BG_m$-torsor associated with $D$. We will also assume the existence of a rational point $\infty \in X(k)$. 

\subsection{Reductive group schemes}
Let $\BG$ be a split reductive group scheme over $k$ and $(\BT,\BB,s)$ a split pinning of $\BG$. We denote by $\BX^*(\BT)$ and $\BX_*(\BT)$ the root and coroot lattice of $\BG$.
 We define a quasi-split form of $\BG$ over $\X$ by fixing a $\Out(\BG)$-torsor $\rho$ over $\X$ and setting 
\[G = \BG \times^{\Out(\BG)} \rho.\]
We further denote by $T = \BT \times^{\Out(\BG)} \rho$ the induced maximal torus and by $\bfg$,$\bft$,$\g$ and $\t$ the Lie algebras of $\BG$, $\BT$, $G$ and $T$ respectively. The superscripts $\reg$ and $\rs$ will denote the open dense subschemes of regular and regular semi-simple elements respectively. 

\subsection{Higgs bundles}
 \label{sec:higgs-bdl}
\begin{definition} A $G$-Higgs bundle (with coefficients in $D$) is a pair $(E, \theta)$ with $E$ a $G$-torsor on $X$ and $\theta$ a global section of $\ad(E) \otimes D$.\\
Equivalently a $G$-Higgs bundle is given by a morphism $X \rightarrow [\g_D/G]$ over $X$, where $\g_D = \g \times^{\BG_m}_X D$ and $G$ acts via the adjoint action. 
\end{definition}

Let $\BM_G = \BM_G(\X,D)$ be the moduli stack parametrizing $G$-Higgs bundles on $\X$. The global sections $Z(\X,G)$ of the center $Z(G)$ over $\X$ act as automorphisms on any $G$-Higgs bundle and we denote by $\FM_G$ the rigidification of $\BM_G$ by $Z(\X,G)$.

\begin{definition} A $G$-Higgs bundle is regular if the morphism $X \rightarrow [\g_D/G]$ factors through $[\g^{\reg}_D/G]$. We denote by $\BM_G^{\reg}$ and $\FM_G^{\reg}$ the corresponding open substacks. 
\end{definition}

In Section \ref{sec:coendo} below we will also need a slight modification of the notion of a Higgs bundle introduced in \cite{GLWZ}. Given a $Z(G)$-gerbe $\beta$ on $X$ we define a $\beta$-twisted $G$-Higgs bundle to be a morphism $X \to [\g_D/G]^\beta$, where $[\g_D/G]^\beta$ is the $\beta$-twist of the $Z(G)$-gerbe $ [\g_D/G] \to  [\g_D/(G/Z(G))]$. We write $\BM_G^{\beta}$ for the corresponding moduli stack.

\subsection{The Hitchin fibration}
 Let $\bfc = \bfg / \BG $ be the Chevalley base of $\bfg$ and $\mfc / \X$ its relative version over $\X$. The scaling action of $\BG_m$ on $\g$ descends to $\mfc$ and we write $\mfc_D = \mfc \times_{\X}^{\BG_m} D$. 

\begin{definition}
\label{def:hfib} The Hitchin base $\A=\A_G=\A_G(X,D)$ is the affine space of global sections $H^0(X,\mfc_D)$. The morphism
\[\chi: \BM_G \longrightarrow \A \]
induced by the quotient $[\g_D/G] \to \mfc_D$ is called the Hitchin fibration. 
\end{definition}

The Kostant section $\bfc \to \bfg$ together with a choice of square root of $D$ gives rise to a section $\A \to \BM^{\reg}_G$ of $\chi$, which we still call the Kostant section. 

\subsection{Symmetries}
Let $\BP_G \to \A$ be the Picard stack of torsors for the regular centralizer $J \to \mfc_D$, also called the abstract Prym. It naturally acts on $\BM_G / \A$ and $\BM_G^{\reg}$ is a $\BP_G$-torsor trivialized by the Kostant section. As before we denote by $\FP_G$ the rigidification of $\BP_G$ by $Z(\X,G)$.

\subsection{$\pi_0(\BP_G)$ and the anisotropic locus}
\label{sec:ani}
Of central importance for Geometric Stabilization is the group scheme of connected components $\pi_0(\BP_G)$.

The anisotropic locus $\A^{\ani} \subseteq \A$ is the open subscheme whose points over an algebraic closure $\bar{k}$ are given by 
\[\A^{\ani}(\bar{k}) = \{ a \in \A^{\heartsuit}(\bar{k}) \ |\ |\pi_0(\BP_{G,a})| < \infty\},\]
with $\A^{\heartsuit}$
the open subset of
$\A$ where the cameral cover is reduced, see \cite[4.5]{ngo:fl}.

Evaluation at our fixed point $\infty \in X(k)$ defines a morphism $\ev_\infty: \A \to \mfc_{D,\infty}$, where $\mfc_{D,\infty}$ denotes the fiber of $\mfc_{D}$ over $\infty$. 
We define \'etale open subschemes $\wtA$ and $\wtA^{\ani}$ of $\A$ and $\A^{\ani}$ by the cartesian squares
\[\xymatrix{ \widetilde{\A}^{\ani} \ar[r] \ar[d] & \wtA \ar[d] \ar[r] & \t^{\reg}_{D,\infty} \ar[d] \\ \A^{\ani} \ar[r] & \A \ar[r]^{\ev_\infty} & \mfc_{D,\infty}.}\]

The pullback of $\BM_G$ to $\wtA$ and $\wtA^{\ani}$ is denoted by $\wtBM_G$ and $\wtBM^{\ani}$ respectively, and similarly for $\FM_G, \BP_G$ and $\FP_G$.

Next we further assume that the fiber of the $\Out(\BG)$-torsor $\rho$ over $\infty$ is trivial and we fix a point $\infty_\rho \in \rho_\infty(k)$. With this choice, the Abel-Jacobi map induces over $\wtA$ a morphism 
\begin{equation}\label{ajmap} \phi: \wtA \times \BX_*(\BT) \longrightarrow \wtP_G. \end{equation}

By \cite[Proposition 4.38]{GWZ18} the composition of $\phi$ with the projection onto $\pi_0(\wtP_G)$ gives for every $a\in \wtA^{\ani}(k)$ a surjection 
\begin{equation}\label{latact} \BX_*(\BT) \longrightarrow \pi_0(\wtP_{G,a}). \end{equation}
In particular $\pi_0(\wtP_{G,a})$ is a finite constant group scheme.

\subsection{Coendoscopic groups and twisted inertia}
\label{sec:coendo} This is a summary of \cite[Section 5]{GWZ18} and \cite[Section 2]{GLWZ}. Roughly speaking a coendoscopic group of $G$ is the Langlands-dual to an endoscopic group of $G$.

 A coendoscopic datum for $G$ over $k$ is a triple $\FE=(\kappa,	\rho_\kappa, \rho_\kappa \to \rho)$, where $\kappa: \widehat{\mu} \to \BT$ is a homomorphism over $k$ and $(\rho_\kappa, \rho_\kappa \to \rho)$ is a certain reduction of the torsor $\rho$. We call $\kappa$ the type of $\FE$. The split coendoscopic group $\BH_\kappa \subseteq \BG$ associated with $\kappa$ is simply the neutral component of the centralizer of the image of $\kappa$. The coendoscopic group $H_{\FE}$ is a quasi-split form of $\BH_\kappa$ over $X$ defined by $(\rho_\kappa, \rho_\kappa \to \rho)$. We write $T_\FE, \mfh_\FE, \mft_\FE...$ for the maximal torus, Lie algebra, Cartan algebra... associated with $H_\FE$.

Coendoscopic data appear naturally in the study of the twisted inertia stack $I_{\hmu} \BM_G$ as follows. Given $\FE$ one obtains as in \cite[Section 2.3]{GLWZ} a $Z(H_\FE)$-gerbe $\beta_\FE$ and we write $\BM_\FE$ for the moduli stack $\BM_{H_{\FE}}^{\beta_\FE}$ of $\beta_\FE$-twisted $H_\FE$-Higgs bundles. As in \cite[Construction 2.21]{GLWZ} we obtain compatible morphisms $\mu_\FE:\BM_\FE \to \BM_G $ and $\nu_{\FE}:\A_\FE = \A_{H_\FE} \to \A_G$. 

Furthermore $\FE$ induces for any $\beta_\FE$-twisted $H_\FE$-Higgs bundle $F$ a morphism $\hmu \to \Aut(F)$, which allows us to lift $\mu_\FE$ to a morphism to $I_{\hmu} \BM_G$. In summary we have a commutative diagram
\begin{equation}\label{square1}
\xymatrix{
\BM_{\FE} \ar[d] \ar[r]^{\mu_\FE} & I_{\hmu} \BM_G \ar[d]   \\
\A_{\FE} \ar[r]^{\nu_\FE} &\A_G,
}
\end{equation}
where the vertical arrows are the Hitchin fibrations. Furthermore $\nu_\FE$ respects the anisotropic locus \emph{i.e.} $\nu^{-1}_\FE(\A^{\ani}_G) = \A^{\ani}_{H_\FE} = \A^{\ani}_{\FE}$.

By \cite[Lemma 2.25]{GLWZ} the fiber of $\rho_\kappa$ over $\infty$ has a $k$-rational point $\infty_{\rho_\kappa}$. We write $\A^{G-\infty}_{\FE} = \nu^{-1}_\FE(\A_G^\infty)$. One can check that $\A^{G-\infty}_{\FE}$ is contained in $\A^\infty_{\FE}$ and we define $\wA^{G-\infty}_{\FE} = \A^{G-\infty}_{\FE} \times_{\A^\infty_{\FE}} \wA_{{\FE}}$. 

The chosen point $\infty_{\rho_\kappa}$ induces an isomorphism $\mft_{\FE,\infty} \cong \mft_\infty$ which allows us to lift $\nu_\FE$ to a morphism $\wnu: \wA^{G-\infty}_{\FE} \to \wA_G$.

We further restrict the above constructions to the anisotropic locus, which we indicate by adding a superscript $\ani$. For the various exponents $*$ appearing above we will also denote by $\wBM_G^* \to \wA^*_G$ and $\wBM^*_{\FE} \to \wA^*_{\FE}$ the pullbacks of  $\BM_G \to \A_G$ and $\BM_{\FE} \to \A_{\FE}$ to $\wA^*$ and $\wA^*_{\FE}$. In particular, from \eqref{square1} we get the commutative square 
\begin{equation*}
\xymatrix{
\wBM_{\FE}^{G-\infty,\ani} \ar[d] \ar[r]^{\wmu_\FE} & I_{\hmu} \wBM_G^{\ani} \ar[d]   \\
\wA_{\FE}^{G-\infty,\ani} \ar[r]^{\wnu_\FE} & \wA_G^{\ani}.
}
\end{equation*}

Conversely, if we assume the existence of a $k$-rational point $\infty_\rho$ in the fiber of $\rho$ at $\infty$, we get for any anisotropic $G$-Higgs bundle $F$ an inclusion $\Aut(F) \to \BT$ \cite[Construction 4.36]{GWZ18}. Thus any $k$-point in $I_{\hmu} \wBM_G^{\ani}$ defines a morphism $\kappa: \hmu \to \BT$ and in fact a coendoscopic datum $\FE= (\kappa, \rho_\kappa, \rho_\kappa \to \rho)$ \cite[Proposition 5.12]{GWZ18}. We say that a coendoscopic datum $\FE$ occurs in $I_{\hmu} \wBM_G^{\ani}(k)$ if it arises from a point in $I_{\hmu} \wBM_G^{\ani}(k)$ via this construction and write $I_{\hmu} \wBM_G^{\ani}(k)_\FE$ for the component with given occurring endoscopy datum $\FE$.

The above construction defines an equivalence \cite[Corollary 2.27]{GLWZ}
\begin{equation}\label{nreq}  \wmu_\FE:  \wBM_{\FE}^{G-\infty,\ani}(k) \xrightarrow{\: \sim\:}  I_{\hmu} \wBM_G^{\ani}(k)_\FE. \end{equation}

For our applications we need a variant of this equivalence for the inertia stack of the rigidification $I_{\hmu} \wFM_G^{\ani}$, which will have fewer components indexed by orbits $[\FE]$ of occuring coendoscopy data under a natural $Z(X,G)$-action \cite[Construction 5.20]{GWZ18}. Since the coendoscopic groups for two coendoscopy data in the same orbit are isomorphic we will ignore this subtlety and write $I_{\hmu} \wFM_G^{\ani}(k)_\FE$ for the corresponding component. If we write $\wFM_{\FE}^G$ for the $Z(X,G)$-rigidification of $\wBM_{\FE}^{G-\infty,\ani}$, we deduce from  \eqref{nreq}  the equivalence \cite[Corollary 2.29]{GLWZ}
\begin{equation}\label{rigineq} \overline{\mu}_{\FE}: \wFM_{\FE}^G(k) \xrightarrow{\: \sim\:} I_{\hmu} \wFM_{G}^{\ani}(k)_\FE. \end{equation}

Notice also that  because of \eqref{square1} the equivalences \eqref{nreq} and \eqref{rigineq} are compatible  with the restriction to individual Hitchin fibers.

\section{Definability of Geometric Stabilization}
\label{sec:def:geo:stab}
In this section we show that the various objects appearing in the statement of the Geometric Stabilization theorem are definable in the theory of pseudo-finite fields and that its statement is expressible as an equality between motivic integrals. Part of this discussion is similar to \cite{chl} (see also \cite{gh}), where it is proven that the statement of the Fundamental Lemma is expressible as an equality between motivic integrals. 

\subsection{Field extensions}
\label{sec:field:ext}
Let $K$ be a pseudo-finite field. We shall need to express conditions related to finite Galois extensions $L$ of $K$. The ring language does not allow to do this directly, however, as long as the degree $\deg(L/K)=r$ is fixed (or at least bounded), we can emulate it within the ring language as follows. 

We describe $L$ by using the minimal polynomial $m$ of an element $\alpha$ generating $L/K$. The polynomial $m_b=x^r+ b_{r- 1} x^{r- 1}+ \dots + b_0$ is described via its tuple of coefficients $b =(b_0,\dots, b_{r-1})\in K^r$. The following conditions can be described with ring formulas with free parameter $b$. We can express the fact that $m_b$ is irreducible in $K[x]$. We then use the explicit basis $\set{1,x,\dots,x^{r-1}}$ of $K[x]/(m_b)=L$ to identify $L$ with $K^r$. Since $K$ is pseudo-finite, $L/K$ is Galois, with Galois group cyclic of order $r$. The Galois group $\Gal(L/K)$ can now be described as well, viewing its elements as linear automorphisms of $K^r$ respecting the field structure. We write the set of elements of $\Gal(L/K)$ as $\Gal(L/K)=\set{\sigma_1,\dots,\sigma_r}$.

Hence we find that there is a ring formula $\varphi(b,\tau)$ such that for every pseudo-finite field $K$, $\varphi(b,\tau)$ is satisfied exactly when $m_b$ is irreducible in $K[x]$ and $\tau$ is a generator of the Galois group of 
$K_{b,r}=K[x]/m_b$ over $K$. 

Given a definable set $X$, there is a definable set $X_{r}$ with parameters $b,\tau$ such that the $K$-points of $X_{r}$ are identified with $X(L)$ for $L$ a Galois extension of degree $r$.

\subsection{Galois cohomology}
\label{sec:galois:coh} 
We can also encode part of Galois cohomology as follows. Let $K$ be a pseudo-finite field and $L$ a finite Galois extension of $K$. Let $A$ be a definable $\Gal(L/K)$-module, by which we mean that $A$ is definable and the $\Gal(L/K)$-action as well. For example $A=X(L)$, where $X$ is definable. We represent the elements of $\mathrm{H}^1(\Gal(L/K),A)$ using 1-cocycles. We view a 1-cocycle as a tuple $(c_1,\dots,c_r)\in A^r$ corresponding to elements $(\sigma_1,\dots,\sigma_r)$ of $\Gal(L/K)$. The cocycle condition is then expressed by the definable condition
\[
\sigma_i \sigma_j=\sigma_k\implies c_i\sigma_i(c_j)=c_k.
\] 
We encode the fact that two cocycles $c,c’$ are cohomologous if there exists a 1-coboundary $d$ such that $c=c’d$. Since this is a definable equivalence relation on 1-cocycles, 
this allows us to consider $\mathrm{H}^1(\Gal(L/K),A)$ as an imaginary set (with parameters $(b,\tau)$). 

If $A\to B$ is a definable map between $\Gal(L/K)$-modules, then we get a definable map between $1$-cocyles for $A$ and $1$-cocycles for $B$.

Basic facts about Galois cohomology are now expressible in a first-order way. 

\begin{proposition}
\label{prop:gal:co1}
Let $G$ be a finite commutative group of order $r$. There is a definable bijection with parameters $(b,\tau)$ such that for every pseudo-finite field $K$,
\begin{equation}
\label{eq:H1G}
\mathrm{H}^1(\Gal(K_{b,r}/K),G)\simeq G.
\end{equation}
determined by the chosen generator $\tau$ of $\Gal(K_{b,r}/K)$. 
\end{proposition}
\begin{proof}
The image of the generator $\tau$ of $\Gal(K_{b,r}/K)$ determines a definable map
$\mathrm{H}^1(\Gal(K_{b,r}/K),G)\to G$.

To show that it is a bijection, it is enough to show that it is bijective for every finite field $\Ff_q$. In this case, the map becomes
\[
\mathrm{H}^1(\Gal(\Ff_{q^r}/\Ff_q),G)\longrightarrow G,
\]
determined by the image of a generator of $\Gal(\Ff_{q^r}/\Ff_q)$. Since $r$ is the order of $G$, it is well-known that it is bijective.
\end{proof}

\begin{proposition}
\label{prop:lang}
Let $\wP$ be a definable commutative algebraic group. Assume that $\wP$ is connected. Then for every pseudo-finite field $K$ and $r\geq 1$,
\[
\mathrm{H}^1(\Gal(K_{b,r}/K),\wP(K_{b,r}))\simeq\set{0}.
\]
\end{proposition}
\begin{proof}
Since for an algebraic group being connected is equivalent to being irreducible, it follows from the following Proposition \ref{irrdef} that
the condition  for $\wP$ to be connected is definable.

As in the previous proposition, to show the result it is enough to show that for every finite field $\Ff_q$,
\[
\mathrm{H}^1(\Gal(\Ff_{q^r}/\Ff_q),\wP(\Ff_{q^r}))\simeq\set{0}.
\]
Since $\wP(\Ff_{q^r})$ is the set of $\Ff_{q^r}$-points of a connected commutative algebraic group over $\Ff_q$, this holds by Lang's theorem.
\end{proof}

\begin{proposition}\label{irrdef}
Let $f : X \to \mathbb{A}_{\mathbb{Z}}^m$ be a quasi-projective morphism. Then the locus of irreducible fibers is definable, namely, there exists a formula
$\varphi (x_1, \cdots, x_m)$ in the (first-order) language of rings such that, for any field $K$ and $b= (b_1, \cdots, b_m)$ in $K^m$,
the fiber $X_b$ is irreducible if and only $\varphi (b_1, \cdots, b_m)$ holds in $K$.
\end{proposition}

\begin{proof}
By working in  affine charts it is enough to consider the case where $f$ is affine (of finite type). In this case the statement follows directly from statements (I) and (II) in the introduction of \cite{vdds}, see  page 79 of \cite{vdds} for an explicit statement.
\end{proof}

\begin{proposition}
\label{prop:gal:co2}
Let $\wP$ be a definable commutative algebraic group with constant group $\pi_0(\wP)$ of connected components  of order $r$. There is a definable bijection with parameters $(b,\tau)$  such that for every pseudo-finite field $K$ 
\begin{equation}
\label{eq:H1PpiO}
\mathrm{H}^1(\Gal(K_{b,r}/K),\wP(K_{b,r}))\simeq \mathrm{H}^1(\Gal(K_{b,r}/K),\pi_0(\wP)).
\end{equation}
\end{proposition}
\begin{proof}
We have a definable map
\[
\mathrm{H}^1(\Gal(K_{b,r}/K),\wP(K_{b,r}))\longrightarrow \mathrm{H}^1(\Gal(K_{b,r}/K),\pi_0(\wP)(K_{b,r})).
\]
Since this map is surjective when $K$ is replaced by a finite field, it is also surjective for pseudo-finite fields. It is injective by Proposition \ref{prop:lang}.
\end{proof}

 Combining Propositions \ref{prop:gal:co1} and \ref{prop:gal:co2} above, we finally get:

\begin{proposition}\label{deflthm} Let $\wP$ be a definable commutative algebraic group, with constant group $\pi_0(\wP)$ of connected components of order $r$. There are definable bijections with parameters $(b,\tau)$ such that for every pseudo-finite field $K$,
\begin{equation}
\label{eq:H1pi0}
\mathrm{H}^1(\Gal(K_{b,r}/K),\wP(K_{b,r}))\simeq \mathrm{H}^1(\Gal(K_{b,r}/K),\pi_0(\wP))\simeq \pi_0(\wP).
\end{equation}
Those bijections are explicitly described at the level of cocycles. For $g\in \pi_0(\wP)$, one builds a 1-cocycle $f_g$ by sending the generator $\tau$ of $\Gal(K_{b,r}/K)$ to $g$.  Then this cocycle can be lifted to $\wP(K_{b,r})$: there exists an element of $x\in\wP(K_{b,r})$ such that sending $\tau$ to $x$ determines a $1$-cocycle that is sent to $f_g$ by the morphism $\wP(K_{b,r})\to \pi_0(\wP)$. 
\end{proposition}

\subsection{Twisted varieties}
\label{sec:twisted}
Let $\wX$ be a definable variety and $\wP$ a definable commutative group acting definably on $\wX$.

Fix a pseudo-finite field $K$. Suppose we are given a $1$-cocycle
\[f\colon \Gal(K_{b,r}/K)\longrightarrow \wP(K_{b,r}).\] By composing this cocycle with the action of $\wP(K_{b,r})$ on $\wX(K_{b,r})$, we get for each $\sigma\in \Gal(K_{b,r}/K)$ an automorphism of $\wX(K_{b,r})$. Define $\wX^f(K)$ as the set of fixed points of $\wX(K_{b,r})$ under this system of automorphisms.  This gives a formula with parameters $b,\tau,f$ for a definable set $\wX^f$ such that for every pseudo-finite field $K$, ${\wX}^f(K)$ is the set of fixed points of the system of automorphisms.

The system of automorphisms is a Galois descent datum for the variety corresponding to $\wX$ over $K$, hence ${\wX}^f$ is a definable algebraic variety.

\subsection{Isotypical components}
\label{sec:isocomp}
Keep notations as above, with $\wX$ a definable algebraic variety with an action of a definable commutative algebraic group $\wP$. Let $\alpha$ be a character of $H=\mathrm{H}^1(\Gal(K_{b,r}/K),\pi_0(\wP)(K_{b,r}))$, where $r$ is large enough such that $\pi_0(\wP)(K_{b,r})=\pi_0(\wP)(K^\mathrm{alg})$. Consider the motivic version $\alpha_\mot$ of $\alpha$ defined in Section \ref{sec:motchar}.

\begin{definition}
\label{def:chi:iso:mcP}
The $\alpha$-isotypical component of the motive of $\wX$ is
\[
\chi_{\psf,\rel}(\wX,\alpha):=
\frac{1}{\abs{H}}
\sum_{t\in H} \chi_{\psf,\rel}({\wX}^t) \, \alpha_\mot(t)
\]
in $\K{\DAC(\rel,\Lambda)}\otimes \BQ$.
\end{definition}

Thanks to the following remark, it is consistent with the notion introduced in Proposition~\ref{prop-chialpha}, at the cost of multiplying by the motive of a set of parameters. 
Note that the multiplicities $n_\alpha$ appearing in Proposition~\ref{prop-chialpha} are here equal to 1 since $H$ is commutative.

\begin{rmk}
Assume that $H$ is a constant finite commutative group. Let $\alpha$ be a character of $H$, and $X$ be a $k$-variety with action of $H$. Then Proposition \ref{lem:FT:mu:n} states that
\[
\frac{1}{\abs{H}}\sum_{g\in H} \chi_{\psf,\rel}(X^g) \, \alpha_\mot(g) = \chi_{\psf,\rel}(X,\alpha)\, )  \chi_{\psf,\rel}([S])
\]
in  $\K{\DAC(\rel,\Lambda)}\otimes \BQ$, where $S$ is the set of parameters.
\end{rmk}

We extend the definition to the category of definable sets with action of $\wP$ over a base $S$, to get a morphism $\chi_{\psf,S,\rel}(\cdot,\alpha)$. 
When $\alpha$ is the trivial character, we write $\chi_{\psf,S,\rel}(\cdot,\mathrm{stab})=\chi_{\psf,\rel}(\cdot,\alpha)$. 

Finally let us record the following lemma for later use.
\begin{lemma}\label{lem:square} Let $\wX$ a definable algebraic variety with an action of a definable commutative algebraic group $\wP$. Assume further that we have a subgroup $\wP^\square \subseteq \wP$ consisting of a union of connected components of $\wP$, that is $\wP/\wP^\square \cong \pi_0(\wP)/\pi_0(\wP^\square)$, and a subvariety $\wX^\square \subseteq \wX$ such that $\wX \cong \wX^\square \times^{\wP^\square} \wP$ as $\wP$-varieties. Let $r$ be a large enough integer such that $\pi_0(\wP)(K_{b,r})=\pi_0(\wP)(K^{\alg})$.

Then we have for any character $\alpha$ of $H=\mathrm{H}^1(\Gal(K_{b,r}/K),\pi_0(\wP)(K_{b,r}))$ the equality
\[  \chi_{\psf,\rel}(\wX,\alpha) = \chi_{\psf,\rel}(\wX^\square,\alpha_{|H^\square}),\]
where $H^\square = \mathrm{H}^1(\Gal(K_{b,r}/K),\pi_0(\wP^\square)(K_{b,r}))$.
\end{lemma}
\begin{proof}
 The isomorphism $\wX \cong \wX^\square \times^{\wP^\square} \wP$  implies that we have for any $t \in H$ a definable morphism $\wX^t \to \pi_0(\wP)^t/\pi_0(\wP^\square)$. The $\pi_0(\wP)/\pi_0(\wP^\square)$-torsor $\pi_0(\wP)^t/\pi_0(\wP^\square)$ is trivial if and only if $t \in H^\square$ and thus $\wX^t$ can only admit a $K$-rational point for $t \in H^\square$. This implies that $ \chi_{\psf,\rel}(\wX^t)=0$ for all $t\notin H^\square$ and 
\[ \chi_{\psf,\rel}(\wX^t) = \chi_{\psf,\rel}(\wX^{\square,t} \times \pi_0(\wP)/\pi_0(\wP^\square)) = \chi_{\psf,\rel}(\wX^{\square,t}) |\pi_0(\wP)/\pi_0(\wP^\square)|.\]

Since the absolute Galois group $\Gal(K) = \widehat{\BZ}$, say with topological generator $\sigma$, we have for any finite commutative group scheme $\Gamma$ an exact sequence
\[  0 \longrightarrow \Gamma(K) \longrightarrow \Gamma(K^\mathrm{alg}) \xrightarrow{\:1-\sigma\:} \Gamma(K^\mathrm{alg}) \longrightarrow  \rmH^1(K,\Gamma) \longrightarrow 0,  \]
and thus $|\Gamma(K)| = |\rmH^1(K,\Gamma)|$. Putting all this together we get 
\begin{align*}\chi_{\psf,\rel}(\wX,\alpha) &= \frac{1}{\abs{H}}
\sum_{t\in H} \chi_{\psf,\rel}({\wX}^t)\alpha_\mot(t)\\
& =  \frac{1}{\abs{H}}
\sum_{t\in H^\square} \chi_{\psf,\rel}({\wX}^{\square,t}) |H/H^\square|\alpha_\mot(t)= \chi_{\psf,\rel}(\wX^\square,\alpha_{|H^\square}).\qedhere \end{align*}
\end{proof}

\subsection{Split reductive groups}
Recall that the classification of split connected reductive groups is independent of the base field and that their isomorphism classes are in bijection with root data $\mathcal{D}=(\BX^*,\Phi,\BX_*,\Phi^\vee)$, which are the character group of a split torus, the set of roots, the cocharacter group, and the sets of coroots. Each split reductive group can be realized over $k$. Fix a faithful representation $\delta\colon \BG\to \GL(V)$ for a $k$-vector space $V$, and $\BG$ the group associated to the root datum $\mathcal{D}$. Fix also a basis of $V$. Then there exists a ring formula describing $\delta(\BG)\subseteq \GL(V)$.

\subsection{Quasi-split groups}
\label{sec:qsgroup}
Let $X$ be a smooth projective curve over $k$, viewed as a definable set. For later purposes, suppose that we are also given a rational point $\infty\in X(k)$. Let $D$ be a line bundle of large enough even degree $d$ on $X$, in particular such that $D$ is ample. Hence a power of $D$ defines a closed embedding of $X$ into $\mathbb{P}^n_k$ for some $n$. We identify $X$ with its image in $\mathbb{P}^n_k$. For later purposes, we also assume that $D$ is a square, \emph{i.e.} that there exists $D'$ such that ${D'}^{\otimes 2}\simeq D$. 

Let $\rho$ be an $\Out(\BG)$-torsor over $X$ that is trivial over $\infty\in X$. By \cite[Lemma 4.8]{GWZ18} there is a finite \'etale cover $X'$ of $X$ over which $\rho$ becomes trivial, and we can moreover assume that $X'\to X$ is Galois. We then view $\rho$ as a morphism $\Aut(X'/X)\to \Out(\BG)$, which amounts to choose a finite number of points of $\Out(\BG)$ satisfying some conditions. 
By Section \ref{sec:galois:coh}, if we fix an integer $r$ we get a definable set $\mathrm{H}^1(\Aut(X'/X),\Out(\BG))$ representing the subset of $\Out(\BG)$-torsors on $X$ parametrized by morphisms $\Aut(X'/X)\to \Out(\BG)$ with $X'$ a finite \'etale Galois cover of $X$ of degree $r$. Fix one such parameter $f$.

A quasi-split form of $\BG$ over $X$ is by definition a group of the form $G=\BG \times^{\Out(\BG)} \rho$. The group $G$ can be viewed as a definable set as follows. Let $r$ be the degree of $X'\to X$. Let $K$ be a pseudo-finite field containing $k$. Let $x\in X(K)$. The map $X'\to X$ induces a finite Galois extension $L_x/k(x)=K$ of degree at most $r$. The extension $L_x/K$ can then be described as above, using a parameter $\tau$ for a fixed generator of $\Gal(L_x/K)$.

The torsor $\rho$ induces at $x$ an element of $\mathrm{H}^1(\Gal(L_x/K),\Out(\BG))$, and the choice of a generator $\tau\in \Gal(L_x/K)$ identifies this group with $\Out(\BG)$. Let $\rho_x\in \Out(\BG)$ be the automorphism defined by $\rho$ at $x$. 

The set $G(K)_x$ of $K$-points of $G=\BG \times^{\Out(\BG)} \rho$ above $x$ is identified with the set of fixed points of $\BG(L)$ under the endomorphism $\rho_x\circ \tau$ of $\BG(L)$. Using the fixed representation $\delta$ of $\BG$ and our conventions regarding encoding of $L$ in Section \ref{sec:field:ext}, the set $G(K)_x$ is therefore defined by a formula with parameters. Since this formula is uniform in $x$, we get a definable $\wG$ with parameters $X, b, \tau,\rho$  such that $\wG_x(K)=G(K)_x$ for every pseudo-finite field $K$ and $x\in X(K)$. 

\bigskip

Summarizing the above discussion, we can state:
\begin{proposition}
\label{prop:qsplit:grp}
Fix a smooth projective curve $X$ over $k$, a split reductive group $\BG$ determined by its root datum, a fixed faithful representation of $\BG$, and an integer $r$. Consider an $\Out(\BG)$-torsor over $X$ that is trivialized over a finite étale cover of $X$ of degree $r$, represented by a 1-cocycle $f$. Then there is a definable algebraic group $\wG$ with parameters $X, b, \tau, f$ such that for every pseudo-finite field $K$ extending $k$ and $x\in X(K)$, $\wG_x(K)=G(K)_x$, where $G$ is the quasi-split form of $\BG$ over $X$ determined by $f$.
\end{proposition}

The lattice of characters $\BX^*=\BX^*(\BT)$ of $\BG$ is similarly endowed with a Galois action. Its sublattice of fixed points is the lattice of characters of $T$, the fixed (non-split) maximal torus of $G$. We identify $\BX^*$ with $\BZ^n$ by fixing a basis. Since the Galois action on $\BZ^n$ is definable (in the sense that the matrix defining it in the chosen basis is definable), we get a finite definable set whose points form a basis of the sublattice $\BX^*(T)$.  

The fixed maximal torus is also represented by a definable algebraic group $\wT$. The Lie algebras $\wg$ and $\wt$ are also defined as a set of fixed points using the $\Out(\BG)$ torsor, hence definable as well.

\subsection{Higgs bundles}\label{higgsb} 
In this section, we encode definably the various objects related to the moduli space of Higgs bundles and Hitchin fibration recalled in Section \ref{sec:geo-setup}.

Given a split reductive  group $\BG$ on $k$, and a smooth projective curve $X$ with a $k$-point, we so far constructed a definable group $\wG$. Fix in addition an even integer $d$ larger or equal to $2g-2$, where $g$ is the genus of $X$. 
To define Higgs bundles we need to fix a line bundle on $X$.  Since the space of classes of line bundles of degree $d$ on $X$ is representable by a quasi-projective scheme $\mathrm{Pic}^d_{X/k}$, we add $\mathrm{Pic}^d_{X/k}$ to the set of parameters, and in what follows the definable set will depend on a parameter $D\in \mathrm{Pic}^d_{X/k}$ representing a line bundle on $X$ of degree $d$.

Fix $K$ a pseudo-finite field extending $k$. Recall from  Definition \ref{def:hfib} the notion of Hitchin fibration: a map $\chi: \FM_{G_K} \rightarrow \A_{G_K}$, where $G_K$ is the quasi-split group scheme over $X_K$ which has fiberwise the same $K$-points as $\wG(K)$. We would like to encode $\FM_{G_K}$ as a definable set, but it is not true in this generality, we need to restrict to an open subscheme of $\A_{G_K}$, the anisotropic locus. 

To do so, first recall from Definition \ref{def:hfib} that $\A_{G_K}$ is the affine space of global sections $\rmH^0(X,\bfc_D)$, and by \cite[Lemme 4.13.1]{ngo:fl}, its dimension (under the hypothesis $\deg(D)\geq 2g-2$) depends only on $g$, $\deg(D)$, and the rank and the number of roots of $\BG$. Hence there is a definable set $\wwA$ such that $\wwA(K)=\A_{G_K}(K)$ for every pseudo-finite field $K$. Using affine charts for $X$ provided by the projective embedding given by $D$, we can interpret elements of $\wwA$ as global sections hence for every pseudo-finite field $K$ and $a\in \wwA(K)$, $a$ determines a map $X(K)\to \wc_D(K)$.

Recall that we are given as part of the data a point $\infty\in X(k)$. Let $\mathrm{ev_\infty}\colon \A_G\to c_{D,\infty}$ the evaluation at $\infty$. Let $\mft^{\mathrm{rs}}_{D,\infty}$ be the set of regular semi-simple elements. Define the \'etale open $\widetilde{\A}_G\to \A_G$ by the cartesian diagram
\begin{equation}\label{agtild}
\xymatrixcolsep{1pc}
\xymatrix{
\widetilde{\A}_G \ar[d]_{} \ar[rr]^-{} && \mft^{\mathrm{rs}}_{D,\infty}\ar[d]   \\
\A_G \ar[rr] &&\mfc_{D,\infty}.
}
\end{equation}
Since the arrows in the cartesian diagram are definable, there is a definable variety $\widetilde{\wwA}_G$ such that $\widetilde{\wwA}_G(K)=\A_G(K)$ for every pseudo-finite field $K$ extending $k$.

Let $a\in \A_{G_K}$. The cameral cover $\mcC_a$ is the $W$-equivariant morphism $\mcC_a\to X$, where $W$ is the twisted Weil group as in \cite[1.3.3]{ngo:fl}, fitting into the cartesian diagram of $K$-schemes
\begin{equation}
\xymatrixcolsep{1pc}
\xymatrix{
\mcC_a \ar[d]_{} \ar[rr]^-{} &&\mft_D \ar[d]   \\
X \ar[rr]^-{a} &&\mfc_D.
}
\end{equation}

The universal cameral cover is denoted by $\mcC\to X\times \A$. Since the arrows in the cartesian diagram are definable, there is a definable set $\wC$ and a definable map $\wC\to X\times \wwA$ such that the fiber over $a\in \wwA(K)$ is the $K$-points of the cameral cover  $\mcC_a(K)\to X(K)$.

Recall from Section \ref{sec:ani} that the anisotropic locus $\A_{G_K}^\ani\subseteq \A_{G_K}$ is the locus where the cameral curve is reduced and $\pi_0(\FP_{G_K,a})$ is finite. By \cite[Proposition 4.10.3]{ngo:fl}, the second condition is equivalent to the fact that ${\mft_D}^{W_{\widetilde{a}}}=0$, where $\widetilde{a}=(a,\tilde{\infty})\in \tilde{\A}_G$ is in the preimage of $a$ and $W_{\widetilde{a}}$ is the finite subgroup of $W\rtimes\Out(\BG)$ defined as follows, following \cite[1.3.6, 5.4.1]{ngo:fl}. 

Let $U$ be the open subset of $X$ over which the cameral curve $\mcC_a$ is a $W$-torsor. Let $X'$ the finite cover of $X$ of degree $r$ and $\rho\colon \Aut(X'/X)\to \Out(\BG)$ the torsor that defines the quasi-split form of $\BG$. On the fiber product $C_a\times_U X'$, we have an action of $\Aut(X'/X)$ and $W$, that combine in an action of $W\rtimes\Out(\BG)$. Hence we have a morphism $\Aut(X'_U/U)\to W\rtimes\Out(\BG)$ and we let $W_{\widetilde{a}}$ be its finite image. We thus have the following commutative diagram

\begin{equation}
\xymatrixcolsep{1pc}
\xymatrix{
\Aut(X'_U/U) \ar[d]_{} \ar[rr]^-{} &&W_{\widetilde{a}}\subseteq W\rtimes\Out(\BG) \ar[d]   \\
\Aut(X'/X) \ar[rr] &&\Out(\BG).
}
\end{equation}

 Since the degree $r$ is fixed, the cover $X'$ is definable, and the open subset $U$ as well. Hence the above description of $W_{\widetilde{a}}$ depends definably on $\widetilde{a}\in \widetilde{\wwA}_G$. Hence the subset $\widetilde{\wwA}_G^\ani$ of $\widetilde{a}\in\widetilde{\wwA}_G$ such that $\mcC_a$ is reduced and $\wt^{W_{\widetilde{a}}}=0$ is definable. The anisotropic locus ${\wwA}_G^\ani$ is the image of $\widetilde{\wwA}_G^\ani$ in $\wwA_G$, therefore definable as well. By \cite[6.1]{ngo:fl}, $\A^{\ani}_G$ is a non-empty open subset of $\A_G$.

The Prym $\mcP_{G,a}$ acts on the fiber $\FM_{G,a}$ and using the Kostant section \cite[Section 4.2.4]{ngo:fl}, is identified with an open subset of $\FM_{G,a}$. 
Let $\mcP^{\ani}$ and $\FM_G^{\ani}$ be the restriction of $\mcP_{G_K}$ and $\FM_{G_K}$ to $\A^{\ani}_{G_K}$. By \cite[Proposition 6.1.3]{ngo:fl}, $\mcP^{\ani}$ is a smooth separated Deligne-Mumford stack of finite type over $\A^{\ani}_{G_K}$ and $\FM_{G_K}^{\ani}$ is a smooth separated Deligne-Mumford stack of finite type over $K$. Moreover, the coarse moduli space $M_G^{\ani}$ of $\FM_G^{\ani}$  is projective over $\A^{\ani}_G$. 

We claim that there exists a definable set $\wM_G$ such that $\wM_G^\ani(K)=M_G^{\ani}(K)$. By \cite[Lemme 6.1.2]{ngo:fl}, every $G$-Higgs bundle over $A^{\ani}$ is stable. Hence one can follow the construction by Faltings \cite[Theorem II.5]{falt:GB} of the coarse moduli space of stable $G$-Higgs bundles $M_G^{\mathrm{s}}$. This gives a quasi-projective embedding of $M_G^{\mathrm{s}}$ into some $\BP^N\times X(G)^*$, where $N$ depends only on the dimension of $\A_G$ and the given representation $\delta$ of $\BG$, and $X(G)^*$ is the dual of the character lattice. The fibers over $X(G)^*$ are then definable, therefore the same is true for the restriction of $M_G^{\mathrm{s}}$ to the restriction to the anisotropic locus $\A^{\ani}_G$. From the construction, one also gets that the stack structure is also definable in the sense of Definition \ref{def:DMstack}.

\bigskip

We have thus established:
\begin{proposition}
\label{prop:hitchin:def}
Fix a smooth projective curve $X$ over $k$, a split reductive group $\BG$ determined by its root datum, a fixed faithful representation of $\BG$, and an integer $r$. Consider an $\Out(\BG)$-torsor over $X$ that is trivialized over a finite étale cover of $X$ of degree $r$, represented by a 1-cocycle $f$. Then there is a definable Deligne-Mumford stack $\wM_G^\ani$ over $k$ with parameters $b, \tau, f$ such that for every pseudo-finite field $K$ extending $k$, $\wM_G^\ani(K)=M_{G}^\ani(K)$, where $G$ is the quasi-split form of $\BG$ over $X$ determined by the choice of $b,\tau,f$ as in Proposition \ref{prop:qsplit:grp}. 
\end{proposition}
 We sometimes use the notation $\wM_G$ instead of $\wM_G^\ani$ since we have the definability only above the anisotropic locus.

One gets also that $\mcP_G^{\ani}$, as an open subset in $M_G^{\ani}$, is the points of a definable set $\wP_G^\ani$ which acts definably on $\wM_G^{\ani}$. 

Let $\widetilde{\wM}_G^{\ani}$ be the base change of $\wM_G^{\ani}$ along $\widetilde{\wwA}_G\to \wwA_G$, which is then also definable. By Proposition \cite[Prop. 4.38]{GWZ18} there is a natural surjective morphism 
\begin{equation}
\label{eqn:surjpiO}
\BX_*(\BT) \times  \widetilde{\A}_G^{\ani} \to \pi_0(\widetilde{\mcP}^{\ani}_{G}),
\end{equation}
hence $\pi_0(\widetilde{\mcP}^{\ani}_{G,a})$ is a constant finite group scheme for every $a\in \widetilde{\A}_G^{\ani}$.

Let $\widehat G$ the Langlands dual group of $G$. One has also the Hitchin fibration for $\widehat G$, $\wM_{\widehat G}^{\ani}\to \A_{\widehat G}^{\ani}$. The bases $\wwA_{G}^{\ani}$ and $\wwA_{\widehat G}^{\ani}$ are canonically isomorphic, so we write simply $\wwA^{\ani}$ and  $\widetilde{\wwA}^{\ani}$, when the group $G$ is understood.

\subsection{Definability of moduli of $\beta$-twisted Higgs bundles}


Recall from Section~\ref{sec:coendo} that given a smooth projective curve $X$, $G$ a quasi-split reductive group on $X$ and $\beta$ a a $Z(G)$-gerbe on $X$, one can consider the moduli stack of $\beta$-twisted Higgs bundles $\FM_\beta$, which is a smooth Deligne-Mumford stack over $A^\ani$. 

Suppose now that we are in the situation of Section~\ref{higgsb}, with a smooth projective curve $X$ over $k$, a definable quasi-split reductive group $\wG$ over $X$, and a line bundle $D$ considered as a point in the parameter space. Fix in addition a $Z(\wG)$-gerbe $\beta$ represented by a 3-cocycle on a finite étale cover of $X$. 

\begin{proposition}
There is a definable Deligne-Mumford stack $\wM_G^{\beta,\ani}$ such that for every pseudo-finite field $K$ extending $k$, $\wM_G^{\beta,\ani}(K)=M_G^{\beta,\ani}(K)$.
\end{proposition}

\begin{proof}

We first recall some constructions from \cite{GLWZ}.

Let $\bar \wG=\wG_\FE/Z(\wG)$. From \cite[Construction 2.11]{GLWZ},  consider the universal Higgs bundle $(E,\theta)$ on $X\times_k \BM_{\bar G}$. Associate to it the $Z(G)$-gerbe $[E/G]$, this defines a global section $\beta^{\mathrm{univ}}\in \rmH^0(\BM_{\bar G}, \rmH^2(X,\pi_0(Z(G))))$. Then let $\BM_{\bar G}^\beta$ the largest open subset of $\BM_{\bar G}$ where $\beta^{\mathrm{univ}}=\beta$. This yields a finite decomposition
\begin{equation}
\label{eqn-MbG}
\BM_{\bar G}=\cup_{\beta\in \rmH^2(X,\pi_0(Z(G)))}\BM_{\bar G}^\beta(K).
\end{equation}
By \cite[Lemma 2.12]{GLWZ}, there is a canonical morphism 
\[
F\colon \BM_G^{\beta,\ani}\to \BM_{\bar G}^{\beta,\ani}\times_{A_{\bar G}}A_G,
\]
making $\BM_G^{\beta,\ani}$ a $\mathrm{Sect}_X(BZ(G))$-torsor. Taking the rigidification of this morphism yields a morphism
\[
F\colon \FM_G^{\beta,\ani}\to \FM_G^{\beta,\ani}\times_{A_{\bar G}}A_G,
\]
which is a $\pi_0(Z(G))$-torsor. 

We now consider the definable case. From Proposition~\ref{prop:hitchin:def}, there is a definable $\wM_{\bar G}^{\ani}$. Taking the rigidification of (\ref{eqn-MbG}), we get definable subsets $\wM_{\bar G}^\beta\subset\wM_{\bar G}$ since the condition $\beta^{\mathrm{univ}}=\beta$ is definable. 

By \cite[Lemma 2.12]{GLWZ}, $\FM_G^{\beta,\ani}$ is a $\pi_0(Z(G))$-torsor over $\FM_G^{\beta,\ani}\times_{A_{\bar G}}A_G$. Each $\pi_0(Z(\wG))$-torsor over $\wM_{\bar G}^{\beta,\ani}\times_{\wwA_{\bar G}}\wwA_G$ is definable. To see that $\wM_G^{\beta,\ani}$ is definable, is suffices then to see that the choice of the torsor is definable, \emph{i.e.} that the assignment $K\mapsto F\in \rmH^1(\FM_G^{\beta,\ani}\times_{A_{\bar G}}A_G, \pi_0(Z(G))$ is definable. But this amounts to check which twist of $\FM_G^{\beta,\ani}$ has a section, which is a definable condition.
\end{proof}

\begin{theorem}
\label{stabcomp}
For every definable gerbe $\beta$ on $X$, we have
\[
\chi_{\psf,\widetilde{\wwA}_G^\ani,\rel}(\widetilde{\wM}_{G}^{\beta,\ani},\mathrm{stab})=\chi_{\psf,\widetilde{\wwA}_G^\ani,\rel}(\widetilde{\wM}_{G}^{\ani},\mathrm{stab}).
\]
\end{theorem}
\begin{proof}
By \cite[Section 3.1]{GLWZ}, there exists a $\FP_G$-torsor $\tau$ such that $\FM_G^{\beta,\mathrm{reg}}$ is the twist of $\FM_G^{\mathrm{reg}}$ by $\tau$, \emph{i.e.} $\FM_G^{\beta,\mathrm{reg}}\simeq \FM_G^{\mathrm{reg},\tau}$, where $\reg$ means restriction to the regular locus. 

As in the proof of the previous proposition, checking which twist is the correct one is a definable condition, hence given $\wG$, and a definable gerbe $\beta$, there is a definable $\wP_G$-torsor $\tau$ and a definable bijection
\[
\wM_G^{\beta,\mathrm{reg}}\simeq \wM_G^{\mathrm{reg},\tau}.
\]

This induces a definable bijection such that for every pseudo-finite field $K$,
\[
\widetilde{\wM}_G^{\beta,\ani,\natural}(K)\simeq \widetilde{\wM}_G^{\ani,\tau,\natural}(K),
\]
where as in Section~\ref{sec:orbifold}, $\wX^\natural(K)=\wX(K\llb t\rrb)\cap \wX^\mathrm{reg}(K\llp t\rrp)$.

We know that $\wM_G^{\beta,\ani}$ is a definable Deligne-Mumford stack. So for every pseudo-finite field $K$, the map to the coarse moduli space
\[\FM_G^{\beta,\ani}(K\llb t\rrb)\to M_G^{\beta,\ani}(K\llb t\rrb)\]
is definable.

We now claim that $x\in M_G^{\beta,\ani,\natural}(K)$ lifts to $\FM_G^{\beta,\ani}(K\llb t\rrb)$ if and only if its image in $M_G^{\ani,\tau,\natural}(K)$ lifts to $\FM_G^{\ani,\tau}(K\llb t\rrb)$. This is a first order property and when $K$ is a finite field is the statement of \cite[Lemma 3.4]{GLWZ}. So the property holds also for pseudo-finite fields. 
So the preimages of $\widetilde{\wM}_G^{\beta} \subset I_{\hat{\mu}}\widetilde{\wM}_G^{\beta} $ and $\widetilde{\wM}_G^{\tau} \subset I_{\hat{\mu}}\widetilde{\wM}_G^{\tau} $ under the specialization map $e$ \eqref{spmap} are identified. 

Restricting to the Hitchin fiber  over $a\in \widetilde{\wwA}_G^\ani$ we then obtain by the orbifold formula \ref{th:orbi}
\[ [\widetilde{\wM}^\beta_{G,a}] = \BL^{\dim \widetilde{\wM}^\beta_{G,a}} \int_{e^{-1}(\widetilde{\wM}^\beta_{G,a})}|\omega| = \BL^{\dim \widetilde{\wM}^\tau_{G,a}} \int_{e^{-1}(\widetilde{\wM}^\tau_{G,a})}|\omega|=[\widetilde{\wM}^\tau_{G,a}].\]
This equality holds in $\cC_\psf(\widetilde{\wwA}_G^\ani\times S)$, where $S$ is the set of parameters. We then take its realization in $\K{\DAC(\widetilde{\wwA}_G^\ani,\rel,\Lambda)}$. 

To obtain the comparison of the stable parts, we  consider, for every $t$ in $H^1(K, \pi_0(\wP_{G,a}))$,  the unramified twists $\left(\widetilde{\wM}^\beta_{G,a}\right)^t$ and $\widetilde{\wM}^{\tau+t}_{G,a}$ and repeat the above argument to obtain $\left[\left(\widetilde{\wM}^\beta_{G,a}\right)^t\right] = [\widetilde{\wM}^{\tau+t}_{G,a}]$. Taking realizations and summing up over all $t$ gives the result. 
\end{proof}

\subsection{Geometric Stabilization}
We can now give the precise statement of the motivic version of Geometric Stabilization. 

We fix a field $k$ of characteristic zero, a smooth projective curve $X$ of genus $g$ over $k$ with a rational point $\infty\in X(k)$, an even integer $d$, larger or equal to $2g-2$ and an integer $r$. We view $X$ embedded in projective space as in Section \ref{higgsb}. We consider an ample line bundle $D$ on $X$ of degree $d$, given by trivializations on open charts. As in Section \ref{sec:qsgroup}, we consider a definable quasi-split reductive group $\wG$ over $X$, with parameters $b,\tau,f$.

We fix a definable coendoscopic datum of $\widehat \wG$, $\FE=(\kappa,\rho_\kappa,\rho_\kappa\to \rho)$ which determines a coendoscopic group  $ \wH_\FE$ of $\widehat \wG$. We assume that $\rho_\kappa$ is trivial over the fixed rational point $\infty\in X$, which implies that $\rho$ is trivial as well over $\infty$.

We consider the Hitchin fibrations $\widetilde{\wM}_G\to {\widetilde{\wwA}}^\ani$ and $\widetilde{\wM}_{H_\FE}\to {\widetilde{\wwA}}_{\FE}^\ani$, as well as the immersion ${\widetilde{\wwA}}_{\FE}^\ani \to {\widetilde{\wwA}}^\ani$, which are definable over the same set of parameters. In the following we identify ${\widetilde{\wwA}}_{\FE}^\ani$ with its image in $\widetilde{\wwA}^\ani$.

The homomorphism $\kappa: \hat \mu \to \widehat \BT$ corresponds to an element in $\BX_*(\widehat \BT) \otimes \BQ/\BZ= \Hom(\BX_*(\BT),\BQ/\BZ)$. We consider $\kappa$ as a character of $\BX_*(\BT)$ by identifying $\BQ/\BZ$ with the roots of unity in $\Lambda$ via  $\BQ/\BZ \xrightarrow{e^{2\pi i (\cdot)}}\Lambda$. If follows from \cite[Sections 6.2.2 \& 6.3]{ngo:fl} that $\kappa$ factors through $\pi_0(\widetilde{\mcP}^{\ani}_{G})$ when restricted to ${\widetilde{\wwA}}_{\FE}^\ani$. Thus we obtain a well-defined $\kappa$-isotypical component $\chi_{\psf,\widetilde{\wwA}^{\ani}_{H_\FE},\rel}(\widetilde{\wM}_{G{\vert \widetilde{\wwA}^{\ani}_{H_\FE}}},{\kappa})$ in the sense of Section \ref{sec:isocomp}.

\begin{theorem}[Geometric Stabilization]
\label{th:geo:stab}
Given the above data, we have the equality 
\[
\chi_{\psf,\widetilde{\wwA}^{\ani}_{H_\FE},\rel}(\widetilde{\wM}_{G{\vert \widetilde{\wwA}^{\ani}_{H_\FE}}},{\kappa})=[\mathds{1}( -r_G^{H_\FE}(D))] \, \chi_{\psf,\widetilde{\wwA}^{\ani}_{H_\FE},\rel}(\widetilde{\wM}_{H_\FE},{\mathrm{stab}})
\]
in the relative Grothendieck groups of motives $\K{\DAC(\widetilde{\wwA}^{\ani}_{H_\FE},\rel,\Lambda)}\otimes \BQ$, where $r_G^H(D)=\frac{1}{2} \dim(\widetilde{M}_G-\widetilde{M}_{H_\FE})$. 
\end{theorem}

We also prove a motivic version of the non-standard Fundamental Lemma:

\begin{theorem}
\label{th:geo:ns}
We have the equality 
\[
\chi_{\psf,\widetilde{\wwA}^{\ani},\rel}(\widetilde{\wM}_{G},{\mathrm{stab}})
=\chi_{\psf,\widetilde{\wwA}^{\ani},\rel}(\widetilde{\wM}_{\widehat G},{\mathrm{stab}})
\]
in the relative Grothendieck groups of motives $\K{\DAC(\widetilde{\wwA}^{\ani},\rel,\Lambda)}\otimes \BQ$.
\end{theorem}

\begin{rmk}
In the original Geometric Stabilization theorem, it is the endoscopic group $\widehat{H}_{\FE}$ that appears on the right hand side. We recover this version as well, by applying the non-standard Fundamental Lemma \ref{th:geo:ns} to the group $H_\FE$.
\end{rmk}

\begin{proposition}
\label{prop:spec:geostab}
Let $\BG$ be a split reductive group and $X$ a smooth projective curve on a normal domain $R$ of finite type over $\BZ$ with a rational point, an even integer $d\geq 2g-2$ and an integer $r$. Then there exists a non-empty open subset $U$ of $\Spec(R)$ such that for each closed point of $U$ with residue field $\BF$, the following holds.

For every quasi-split form of $\BG$ over $X_\BF$, $G\to X_\BF$, that splits over a Galois cover of degree $r$, the specialization of $\widetilde{\wM}_{G}$ to $\BF$ for some choice of parameters is the set of $\BF$-points of $\widetilde{M}_G$, the coarse moduli space of $\widetilde{\FM}^\ani_G$. The same holds for the various other geometric object constructed above, such as $\wP_G$ and $\wwA$. 

Moreover, every definable coendoscopic datum $\FE=(\kappa,\rho_\kappa,\rho_\kappa\to \rho)$ of $\widehat \wG$ specializes to a coendoscopic datum  of $\widehat G$, written $\FE$ as well, such that the function $\Tr \Frob \chi_{\psf,\widetilde{\wwA}^{\ani}_{H_\FE},\rel}(\widetilde{\wM}_{G{\vert \widetilde{\wwA}^{\ani}_{H_\FE}}},{\kappa})$ specializes for some choice of parameters in $\BF$ to the function 
\[a\in \widetilde{\A}^\ani_{H_\FE}(\BF) \longmapsto \#^\kappa \widetilde{\FM}_{G,a}(\BF).\]
\end{proposition}

\begin{proof}
For the first part, at each step of the construction in this Section, up to some choice of parameters the specialization, of the definable set to a finite field $\BF$ is the set of points of the geometric object we are encoding. Proposition \ref{prop:spec-for} ensures that the specialization is independent of the choice of formula, up to restricting $\BF$ to be the residue field of a closed point in a non-empty open subset of $\Spec(R)$.

For the second part, by using Proposition \ref{prop:spec-mot} and the first part, we find that for some choice of parameters, $\Tr \Frob \chi_{\psf,\widetilde{\wwA}^{\ani}_{H_\FE},\rel}(\widetilde{\wM}_{G{\vert \widetilde{\wwA}^{\ani}_{H_\FE}}},{\kappa})$ is a sum of the $\BF$-points counts of the twisted stacks:
\[\Tr \Frob_x \chi_{\psf,\widetilde{\wwA}^{\ani}_{H_\FE},\rel}(\widetilde{\wM}_{G{\vert \widetilde{\wwA}^{\ani}_{H_\FE}}})(a)=\frac{1}{\abs{\pi_0(\FP_{G,a})}}\sum_{t\in \pi_0(\widetilde{\FP}_{G,a})} \#\widetilde{\FM}^t_{G,a}(\BF)\kappa(t^{-1}),
\]
since this is the way we defined the $\kappa$-isotypical components for motives. By \cite[Lemma 6.6]{GWZ18}, this sum is equal to $\#^\kappa \widetilde{\FM}_{G,a}(\BF)$.
\end{proof}

Recall the Geometric Stabilization theorem of Ngô \cite[Theorem 6.4.2]{ngo:fl}, as well as the geometric version of Waldspurger non-standard Fundamental Lemma \cite[Theorem 8.8.2]{ngo:fl}. See also \cite[Theorems 1.1 and 1.3]{GWZ18}.
\begin{corollary}
The Geometric Stabilization theorem and the geometric non-standard Fundamental Lemma hold in every finite field of large enough characteristic.
\end{corollary}
\begin{proof}
It suffices to combine Proposition \ref{prop:spec:geostab} with the motivic statements, Theorems \ref{th:geo:stab} and \ref{th:geo:ns}.
\end{proof}

\section{Proof of the main theorems}\label{sec:proofmt}
In this section we prove Theorems \ref{th:geo:stab} and \ref{th:geo:ns} by reducing them gradually to a duality statement about motivic integrals on generic Hitchin fibers, Theorem \ref{thm:mainintfibers}.

\subsection{Reduction to a point}
By Proposition \ref{prop-eval}, Theorems \ref{th:geo:stab} and \ref{th:geo:ns} follow from their following fiberwise versions.

\begin{theorem}
\label{th:geo:stab:a}
With notations as in Theorem  \ref{th:geo:stab}, for every pseudofinite field $K$, for every $a\in \widetilde{\wwA}^{\ani}_{H_\FE}(K)$, we denote by $\kappa_a$ the character of $\pi_0(\widetilde{\wP}_{G,a})$ induced by $\kappa$.  Then the equality 
\[
\chi_{\psf,\rel}(\widetilde{\wM}_{G,a},{\kappa_a})=[\mathds{1}( -r_G^{H_\FE}(D))]\chi_{\psf,\rel}(\widetilde{\wM}_{H_\FE,a},{\mathrm{stab}})
\]
holds in $\K{\DAC(\Spec(K),\rel,\Lambda)}\otimes \BQ$,
where $r_G^{H_\FE}(D)=\frac{1}{2} \dim(M_G-M_{H_\FE})$. 
\end{theorem}

\begin{theorem}
\label{th:geo:ns:a}
For every $a\in \widetilde{\wwA}^{\ani}(K)$,  the equality \[
\chi_{\psf,\rel}(\widetilde{\wM}_{G,a},{\mathrm{stab}})
=\chi_{\psf,\rel}(\widetilde{\wM}_{\widehat G,a},{\mathrm{stab}})
\]
holds in $\K{\DAC(\Spec(K),\rel,\Lambda)}\otimes \BQ$.
\end{theorem}

\subsection{Reduction to inertia stacks}
For now on, we fix a pseudo-finite field $K$ and $a\in \widetilde{A}^{\ani}(K)$. Then $\pi_0(\widetilde{\FP}_{G,a})$ is a finite constant group scheme by \cite[Prop. 4.38]{GWZ18}. 

For $b$ generic, $\widetilde{\FP}_{G,b}$ is a proper commutative group scheme, with group of connected components isomorphic to the center $Z(X,\widehat G)$ of $\widehat G$. 
Recall from Section \ref{sec:qsgroup} that we consider the group of characters $\BX^*(\widehat \BT)$ via a basis of this lattice, which is a definable set. Since $\widehat \BT/Z(X,\widehat G)$ is also a torus, we also have a finite definable set consisting of a basis of $\BX^*(\widehat \BT/Z(X,\widehat G))$. Let $\lambda$ be the composition of the inclusion $\BX^*(\widehat \BT/Z(X,\widehat G))\subseteq\BX^*(\widehat \BT)$ with the natural surjection $ \widetilde{A}^{\ani} \times \BX^*(\widehat \BT)=\BX_*(\BT) \to \pi_0(\widetilde{\mcP}_{G})$ from \cite[Prop. 4.38]{GWZ18} relative over $\widetilde{A}^{\ani}$. Let $\square$ be the image of $\lambda$ in $\pi_0(\widetilde{\mcP}_{G})$. For any $a \in \widetilde{A}^{\ani}(K)$ the fiber $\square_a$ is definable, since it is generated by the image of the basis of $\BX^*(\widehat \BT/Z(X,\widehat G))$ that is definable. 

Hence for generic $b$, $\square_b$ is the trivial group, and for every $a\in \widetilde{A}^{\ani}(K)$, $\square_a$ is a finite constant group scheme. We then define by fiber product $\widetilde{\mcP}^\square_{G}=\square\times_{\pi_0(\widetilde{\mcP}_{G})} \widetilde{\mcP}_{G}$, and similarly $\widetilde{\mcM}^\square_{G}$. The coarse moduli spaces of $\widetilde{\mcM}^\square_{G}$ and $\widetilde{\mcP}^\square_{G}$ can be seen as definable sets denoted by $\widetilde{\wM}^\square_G$ and $\widetilde{\wP}^\square_{G}$ respectively.

Let $r$ be the order of $\pi_0(\widetilde{\FP}^\square_{G,a})=\square_a$. By Proposition \ref{deflthm}, we get identifications of definable sets
\[ \pi_0(\widetilde{\wP}^\square_{G,a}) = H^1(\Gal(K_r/K),\pi_0(\widetilde{\wP}^\square_{G,a})) =   H^1(\Gal(K_r/K),\widetilde{\wP}^\square_{G,a}), \]
where $K_r$ is the extension of degree $r$ of $K$. In particular every $t\in \pi_0(\widetilde{\FP}^\square_{G,a})$ corresponds to a $\widetilde{\FP}^\square_{G,a}$-torsor $T_t$, which we can use to define the twisted Hitchin fiber
\[\widetilde{\mcM}^{\square,t}_{G,a} := \widetilde{\mcM}^{\square}_{G,a} \times^{\widetilde{\FP}^\square_{G,a}} T_t. \]
By Section \ref{sec:twisted} its coarse moduli space gives a definable set $\widetilde{\wM}^{\square,t}_{G,a}$. 

Next we define for $s\in \BX^*(\BT/Z(X,G))$ a natural motivic constructible function $\chi_s$ on $\Imu\widetilde{\wM}^{\square,t}_{G,a}$ using the results of Section \ref{sec:coendo}:

First define the substack $\widetilde{\FM}^\lozenge_{\FE}\subseteq \widetilde{\FM}_{\FE}^G$ as the maximal open substack such that the following diagram commutes:

\begin{equation*}
\xymatrix{
\widetilde{\FM}^\lozenge_{\FE} \ar[d] \ar[r]^{} & \Imu\widetilde{\FM}_G^{\square} \ar[d]   \\
\widetilde{\FM}^{G}_{\FE} \ar[r]^{} & \Imu \widetilde{\FM}_G.
}
\end{equation*}

As in Section \ref{higgsb}, $\widetilde{\FP}_{H_\FE}$ embeds into $\widetilde{\FM}_{H_\FE}$ by means of the Kostant section and we define $\widetilde{\FP}^\lozenge_{H_\FE}$ as the intersection of $\widetilde{\FP}_{H_\FE}$ with $\widetilde{\FM}^\lozenge_{H_\FE}$. As before, we get associated definable sets $\widetilde{\wM}^\lozenge_{\FE}$ and $\widetilde{\wP}^\lozenge_{H_\FE}$.

There is a canonical morphism $\pi_0(\widetilde{\FP}^\square_{G,a})\to  \pi_0(\widetilde{\FP}^\lozenge_{H_\FE,a})$, allowing to define for $t\in \pi_0(\widetilde{\FP}^\square_{G,a})$ twists $\widetilde{\FM}^{\lozenge,t}_{\FE,a}$, with associated definable  $\widetilde{\wM}^{\lozenge,t}_{\FE,a}$, which are compatible with twists of $\Imu\widetilde{\FM}_G^{\square}$ in the following sense.

\begin{proposition}\label{finteq} There is a finite definable partition 
\[
\Imu\widetilde{\wM}_{G,a}^{\square,t}=\bigcup_{[\FE]} \Imu\widetilde{\wM}_{G,a,\FE}^{\square,t}
\]
indexed by orbits of definable coendoscopic data $\FE=(\tau,\rho_\tau, \rho_\tau \to \rho)$ under the natural $Z(X,G)$-action.
 Moreover, the morphism $\overline{\mu}_{\FE}$ from \eqref{rigineq} induces a definable bijection $\widetilde{\wM}^{\lozenge,t}_{\FE,a} \cong \Imu\widetilde{\wM}^{\square,t}_{G,a,\FE}$.
\end{proposition}

\begin{proof}
Recall from Section \ref{sec:coendo} the notion of coendoscopic datum $\FE= (\tau, \rho_\tau, \rho_\tau \to \rho)$. We get a notion of definable coendoscopic datum for $\wG$, where $\tau$ is described  by the image a topological generator of $\widehat \mu$, and $\rho_\tau$ and $\rho$ are defined using 1-cocycles as in Section \ref{sec:galois:coh}. To each point of $\Imu\widetilde{\wM}^{\square,t}_{G,a}$ we associate a morphism $\hat \mu \to \BT$, hence an orbit of a definable coendoscopic datum, see the discussion below (\ref{nreq}) for details. We then get a definable partition of $\Imu\widetilde{\wM}^{\square,t}_{G,a}$ into finitely many definable pieces $\Imu\widetilde{\wM}^{\square,t}_{G,a,\FE}$ indexed by orbits of definable coendoscopic data.

Moreover, the construction (\ref{rigineq}) can be twisted by $t$, yielding a definable map 
\[
\widetilde{\wM}^{\lozenge,t}_{\FE,a} \longrightarrow \Imu\widetilde{\wM}^{\square,t}_{G,a,\FE}.
\]
This map is a bijection in every finite field of large enough residue characteristic by an argument similar to \cite[Lemma 6.8]{GWZ18}, hence in every pseudo-finite field.
\end{proof}

\begin{definition}\label{chioncomp} Let $s \in \BX^*(\BT/Z(X,G))$ and $\FE=(\tau,\rho_\tau,\rho_\tau\to \rho)$ a coendoscopic datum of $G$. Define the function $\chi_s$ on $\Imu\widetilde{\wM}^{\square,t}_{G,a,\FE}$ to be identically equal to the motive corresponding to the composition $s_\mot(\tau)$ of $\tau$ and $s_\mot$, as explained in Section \ref{sec:motchar}. 
\end{definition}

Since the decomposition of  $\Imu \widetilde{\wM}^{\square,t}_{G,a}$ is definable, the function $\chi_s$ on $\Imu\widetilde{\wM}^{\square,t}_{G,a}$  is constructible, \emph{i.e.} an element of $\mcC_\mot\left(\Imu\widetilde{\wM}^{\square,t}_{G,a}\right)$.

Given $t\in \BX^*(\widehat \BT/Z(X,\widehat G))$, we have canonical maps
\[
\BX^*(\widehat \BT/Z(X,\widehat G))\longrightarrow \BX^*(\widehat \BT)=\BX_*(\BT)\longrightarrow \BX_*(\BT/Z(X,G)).
\]
We also have the map $\lambda_a\colon\BX_*( \BT/Z(X, G))\to \pi_0(\widetilde{\FP}^\square_{G,a})$. We write abusively also $t$ for the image of $t$ in $\pi_0(\widetilde{\FP}^\square_{G,a})$ through this composition of maps. In particular, given $t\in \BX^*(\widehat \BT/Z(X,\widehat G))$, we write $\widetilde{\FM}^{\square,t}_{G,a}$ for the twist of $\widetilde{\FM}^{\square}_{G,a}$ by the image of $t$ in $\pi_0(\widetilde{\FP}^\square_{G,a})$. 

\medskip

The main theorem we need to prove is the following. 

\begin{theorem}
\label{th:true:LF}
For every pseudo-finite field $K$, $a\in \widetilde{\wwA}^{\ani}(K)$, $t \in \BX^*(\widehat \BT/Z(X,\widehat G))$ and $s \in \BX^*(\BT/Z(X,G))$, we have
\[
\int^\mot_{\Imu\widetilde{\wM}^{\square,t}_{G,a}} \eL^{-w}\chi_s=\int^\mot_{\Imu\widetilde{\wM}^{\square,s}_{\widehat G,a}} \eL^{-w}\chi_t,
\]
where $w$ is the locally constant weight function from Section \ref{sec:orbifold}.
\end{theorem}

\begin{rmk}
We use integral symbols in the theorem, but since we integrate on a definable subset over the residue field, this is simply a short notation for finite sums over the classes of level sets  of the functions $\eL^{-w}\chi_s$, $\eL^{-w}\chi_t$, which are definable.
\end{rmk}

We will show that Theorems \ref{th:geo:stab:a} and \ref{th:geo:ns:a} follow from Theorem \ref{th:true:LF}, for which we need some lemmas.

\begin{lemma}
\label{lem:wconst}
The weight function $w$ is constant on each piece of the decomposition of  $\Imu \widetilde{\wM}^{\square,t}_{G,a}$, hence can be written as a function $w(\kappa)$.
\end{lemma}
\begin{proof}
By Lemma 6.10 in \cite{GWZ18}, the result holds over finite fields. Since the level sets of $w$ are definable sets, by Proposition \ref{prop:spec-for}, the result holds in pseudo-finite fields as well. 
\end{proof}

\begin{lemma}\label{instre}
Theorems \ref{th:geo:stab:a} and \ref{th:geo:ns:a} follow from Theorem \ref{th:true:LF}.
\end{lemma}

\begin{proof}[Proof of Lemma \ref{instre}]
For the non-standard Fundamental Lemma, Theorem \ref{th:geo:ns:a}, choose lifts $s,t$ of elements of $\pi_0(\widetilde{\wP}^\square_{G,a})$ and $\pi_0(\widetilde{\wP}^\square_{\widehat G,a})$, and sum the equality of Theorem \ref{th:true:LF} over those. We get on the left hand side

\begin{multline*}
\sum_{s,t}\int^\mot_{\Imu\widetilde{\wM}^{\square,t}_{G,a}} \eL^{-w}\chi_s = \sum_{\FE}\sum_{s,t}s(\tau_\FE) \eL^{-w(\tau_\FE)} \chi_{\psf,\rel}(\Imu\widetilde{\wM}_{G,a}^{\square,t}(K)_{\FE})\\
= \eL^{-\dim \widetilde{\wM}_G}\sum_{s,t} \chi_{\psf,\rel}(\widetilde{\wM}_{G,a}^{\square,t}),
\end{multline*}
where we use Lemma \ref{lem:wconst} for the first equality. For the second equality, we use that if $\tau_\FE$ is non-trivial, then $s\mapsto s(\tau_\FE)$ is a non-trivial character, hence the sum over its values is zero by Lemma~\ref{sumchar0}, $w(1)=\dim \widetilde{\wM}_G$ by definition, and Proposition \ref{finteq} to get that $\chi_{\psf,\rel}(\Imu\widetilde{\wM}_{G,a}^{\square,t}(K)_{\FE_{\mathrm{triv}}})=\chi_{\psf,\rel}(\widetilde{\wM}_{G,a}^{\lozenge,t})=\chi_{\psf,\rel}(\widetilde{\wM}_{G,a}^{\square,t})$.

By definition of the isotypical component, we now have
\[\sum_{s,t} \chi_{\psf,\rel}(\widetilde{\wM}_{G,a}^{\square,t})=\abs{\pi_0(\widetilde{\wP}^\square_{G,a})}\abs{\pi_0(\widetilde{\wP}^\square_{\widehat G,a})}\chi_{\psf,\rel}(\widetilde{\wM}_{G,a}^{\square},\mathrm{stab}).
\]

Exchanging the role of $s$ and $t$ and working on the right hand side, we finally have that 
\begin{multline*}\abs{\pi_0(\widetilde{\wP}^\square_{G,a})}\abs{\pi_0(\widetilde{\wP}^\square_{\widehat G,a})}\chi_{\psf,\rel}(\widetilde{\wM}_{G,a}^{\square},\mathrm{stab})=\\
\abs{\pi_0(\widetilde{\wP}^\square_{G,a})}\abs{\pi_0(\widetilde{\wP}^\square_{\widehat G,a})}\chi_{\psf,\rel}(\widetilde{\wM}_{\widehat G,a}^{\square},\mathrm{stab}).
\end{multline*}
We now conclude using Lemma \ref{lem:square}.

For Geometric Stabilization, we proceed similarly, by summing over $s,t$ the equality of Theorem \ref{th:true:LF} multiplied by $\kappa_{a,\mot}(t)^{-1}$.

On the left hand side, we get
\begin{multline*}
\sum_{s,t}\int^\mot_{\Imu\widetilde{\wM}^{\square,t}_{G,a}} \eL^{-w}\chi_s\kappa_{a,\mot}(t)^{-1}=\sum_{\FE}\sum_{s,t}s(\tau_\FE) \eL^{-w(\tau_\FE)} \kappa_{a,\mot}(t)^{-1} \chi_{\psf,\rel}(\Imu\widetilde{\wM}_{G,a,\FE}^{\square,t})\\
=\eL^{-\dim \widetilde{\wM}_G}\sum_{s,t} \kappa_{a,\mot}(t)^{-1}\chi_{\psf,\rel}(\widetilde{\wM}_{G,a}^{\square,t}).
\end{multline*}
By definition of the isotypical component, we now have
\[\sum_{s,t} \kappa_{a,\mot}(t)^{-1} \chi_{\psf,\rel}(\widetilde{\wM}_{G,a}^{\square,t})=\abs{\pi_0(\widetilde{\wP}^\square_{G,a})}\abs{\pi_0(\widetilde{\wP}^\square_{\widehat G,a})}\chi_{\psf,\rel}(\widetilde{\wM}_{G,a}^{\square},\kappa_a).
\]

For the right hand side, we have 
\begin{multline*}
\sum_{s,t}\int^\mot_{\Imu\widetilde{\wM}^{\square,t}_{\widehat G,a}} \eL^{-w}\chi_t\kappa_{a,\mot}(t)^{-1}=\sum_{\FE}\sum_{s,t}t(\kappa_\FE) \eL^{-w(\kappa_\FE)} \kappa_{a,\mot}(t)^{-1} \chi_{\psf,\rel}(\Imu\widetilde{\wM}_{G,a,\FE}^{\square,s})
\\ =\eL^{-w(\kappa_a)}\sum_{s,t} \chi_{\psf,\rel}(\Imu\widetilde{\wM}_{\widehat G,a,\FE_H}^{\square,s}),
\end{multline*}
where we now use that $t\mapsto  t(\kappa_\FE)\kappa_{a,\mot}(t)^{-1}$ is a non-trivial character if and only if $\FE$ is the coendoscopic datum $\FE_H$ corresponding to $H$ (recall from the statement of Theorem~\ref{th:geo:stab:a} that $\kappa_a$ is the character  defined by $H$). 

We now have using Proposition \ref{finteq} 
\[ \chi_{\psf,\rel}(\Imu\widetilde{\wM}_{\widehat G,a,\FE_H}^{\square,s})=\chi_{\psf,\rel}(\widetilde{\wM}_{\FE,a}^{\lozenge,s}).
\]
Using again the definition of isotypical component, 
\[\sum_{s,t} \chi_{\psf,\rel}(\widetilde{\wM}_{\FE,a}^{\square,s})=\abs{\pi_0(\widetilde{\wP}^\square_{G,a})}\abs{\pi_0(\widetilde{\FP}^\square_{\widehat G,a})}\chi_{\psf,\rel}(\widetilde{\wM}_{\FE,a}^{\lozenge},\mathrm{stab}).
\]
From Lemma 6.10 in \cite{GWZ18}, $\dim(\widetilde\wM_G)-w(\kappa_a)=r_G^H(D)$, so we now conclude using Lemma \ref{lem:square} and Theorem \ref{stabcomp}.
\end{proof}

\subsection{Reduction to motivic integrals}
\label{sec:redint}
While Theorem \ref{th:true:LF} is a statement about virtual motives over the residue field, its proof passes through a reformulation in terms of integrals over valued fields. For this, consider $B_a$ the definable assignment of points in $\widetilde{\wwA}(K\llb t\rrb)$ that specialize to $a$, which is an open unit ball and represents the $K\llb t\rrb$-points of $U_a$, the spectrum of the Henselization of $\widetilde{\A}$ at $a$. Similarly, write $U_{a,r}$ for the Galois cover of $U_a$ corresponding to the extension $K_r\to K$, where $r$ is the order of $\pi_0(\wP_a)$.

The restrictions of $\FM_G$ and $\FP_G$ to $U_a$ are denoted by $\FM_{G,U_a}$ and $\FP_{G,U_a}$. Essentially by definition of Henselization, see for example \cite[Theorem I.4.2]{MEC}, we have
\[H^1(\Gal(U_{a,r}/U_a),\FP_{G,U_a}) = H^1(\Gal(K_r/K),\FP_{G,a}) \cong \pi_0(\FP_a).\]
 In terms of cocycles, it means that cocycles representing elements of $H^1(K,\FP_{G,a})$ can be lifted to $\FP_{G,U_a}$, hence used to defined twisted definable sets  as in Section \ref{sec:twisted}. 

For $t \in \pi_0(\widetilde{\FP}^\square_{G,a})$ we then get a smooth DM-stack $\widetilde{\FM}^{\square,t}_{G,U_a}$ and a corresponding definable $\widetilde{\wM}^{\square,t}_{G,B_a}$.

Recall from Section \ref{sec:orbifold} that we have a definable evaluation morphism
\begin{equation}\label{spezz}e:\widetilde{\wM}^{\square,t,\natural}_{G,B_a} \longrightarrow \I_{\widehat{\mu}}\widetilde{\wM}^{\square,t}_G,\end{equation}
and Theorem \ref{th:true:LF} is equivalent to 

\begin{theorem}\label{mainmint}Let $K$ be a  pseudo-finite field $K$. For every $a\in \widetilde{\wwA}^{\ani}(K)$, every  $t \in \BX^*(\widehat \BT/Z(X,\widehat G))$ and every $s \in \BX^*(\BT/Z(X,G))$, we have
\[ \int^\mot_{\widetilde{\wM}^{\square,t,\natural}_{G,B_a}} \chi_s \circ e \ \mu_{\orb} = \int^\mot_{\widetilde{\wM}^{\square,s,\natural}_{\widehat G,B_a}} \chi_t \circ e \ \mu_{\orb}. \]
\end{theorem}

\begin{proof}[Proof that Theorems \ref{mainmint} and \ref{th:true:LF} are equivalent]
This is a direct application of the orbifold formula, Theorem \ref{th:orbi}, which implies that 
\[
\int^\mot_{\widetilde{\wM}^{\square,t,\natural}_{G,B_a}} \chi_s \circ e \ \mu_{\orb} = \int^\mot_{\Imu\widetilde{\wM}^{\square,t}_{G,a}} \eL^{-w}\chi_s
\]
and
\[
\int^\mot_{\widetilde{\wM}^{\square,s,\natural}_{\widehat G,B_a}} \chi_t \circ e \ \mu_{\orb}=\int^\mot_{\Imu\widetilde{\wM}^{\square,s}_{\widehat G,a}} \eL^{-w}\chi_t.
\]
Hence Theorems \ref{mainmint} and \ref{th:true:LF} are equivalent.
\end{proof}

\subsection{Integration on generic fibers} 
A key insight from \cite{GWZ18} is that using a Fubini theorem, Theorem \ref{mainmint} can be proved one generically smooth Hitchin fiber at a time.

More precisely we let $B_a^\flat$ denote the definable assignment of points in $\wwA(K\llb t\rrb)$ that specialize to $a$ and whose generic fiber lies in the smooth locus of the Hitchin fibration. 

In the case of the Hitchin fibration, the measure $\mu_\orb$ is given by a global volume form $\abs{\omega_\orb}$, as explained in \cite[Lemma 6.13]{GWZ18}. For each $b\in B_a^\flat$, write $\omega_b$ for the volume forms on the smooth fibers $\widetilde{M}^{\square,t}_{G,b}$ and $\widetilde{M}^{\square,s}_{\widehat{G},b}$ obtained as a quotient of $\omega_\orb$ by a fixed volume form $\eta$ on the Hitchin base $\widetilde{\FA}^{\ani}$. Proposition \ref{prop-fubini-mot}, a version of Fubini's theorem, shows that 
\[
\int_{\widetilde{\wM}^{\square,t}_{G,B_a}} \chi_s \circ e \ \mu_{\orb}=\int_{b\in B_a^\flat} \abs{\eta}\int_{\widetilde{\wM}^{\square,t}_{G,b}} \chi_s \circ e |\omega_{b}|.
\]
We have a similar equality for the dual group, hence Theorem \ref{mainmint} is implied by the following theorem. 
\begin{theorem}
\label{thm:mainintfibers} For any $b \in B_a^\flat$ we have an equality 
\[ \int_{\widetilde{\wM}^{\square,t}_{G,b}} \chi_s\circ e |\omega_{b}| = \int_{\widetilde{\wM}^{\square,s}_{\widehat G,b}}\chi_t\circ e |\omega_{b}|. \]
\end{theorem}

\subsection{Proof of Theorem \ref{thm:mainintfibers}} Certainly the integral over the definable set $\widetilde{\wM}^{\square,t}_{G,b}$ is zero if for every pseudofinite field $K'$ containing $k(b)$, the set $\widetilde{\wM}^{\square,t}_{G,b}(K')$ is empty. By construction, $\widetilde{\wM}^{\square,t}_{G,b}(K')$ is the set of $K'$-points of the $\widetilde{\FP}^{\square,t}_{G,b}$-torsor $\widetilde{M}^{\square,t}_{G,b}$.

Let $\FN(\widetilde{\FP}^{\square}_{G,b})$ be the N\'eron model of $\widetilde{\FP}^{\square}_{G,b}$ and $\overline{\FN}(\widetilde{\FP}^{\square}_{G,b})$ its special fiber. For fixed $b$, we get a corresponding definable set $\FN(\widetilde{\wP}^{\square}_{G,b})$, but not uniformly in $b$. 

By the N\'eron mapping property we have a morphism $\widetilde{\FP}^{\square}_{G,a} \to \overline{\FN}(\widetilde{\FP}^{\square}_{G,b})$, which induces one on component groups
\[  i: \pi_0(\widetilde{\FP}^{\square}_{G,a}) \longrightarrow \phi_{\widetilde{\FP}^{\square}_{G,b}} = \pi_0(\overline{\FN}(\widetilde{\FP}^{\square}_{G,b})).  \]
By the following lemma, this map detects whether $\widetilde{M}^{\square,t}_{G,b}$ has rational points.

\begin{lemma}
\label{lem:itzero}
The $\widetilde{\FP}^{\square}_{G,b}$-torsor $\widetilde{M}^{\square,t}_{G,b}$ is trivial if and only if $i(t)=0$. 
\end{lemma}
\begin{proof} By construction $\widetilde{M}^{\square,t}_{G,b}$ is unramified, \emph{i.e.} splits after a finite extension of the residue field $K'$.  By \cite[Cor. 6.5.3]{NeronModels}, this torsor extends uniquely to an $\FN\left(\widetilde{\FP}^{\square}_{G,b}\right)$-torsor $\FT$. 
Since the extension is unique and $\FT / K'\llb t \rrb$ is smooth,  $\widetilde{M}^{\square,t}_{G,b}$ is trivial if and only if its restriction to its special fiber $\FT_{K'}$ is. The latter is the unramified $\overline{\FN}\left(\widetilde{\FP}^{\square}_{G,b}\right)$-torsor corresponding to $i(t)$ and thus trivial if and only if $i(t)=0$.
\end{proof}

Using Lemma \ref{lem:itzero} we will identify $\widetilde{M}^{\square,t}_{G,b}$ with $\widetilde{M}^{\square}_{G,b}$ whenever $i(t) = 0$.

\begin{lemma}
\label{lem:chis:triv}
Assume $i(t)=0$. Then the function $\chi_s\circ e$ is constant on $\widetilde{\wM}^{\square}_{G,b}$ if and only if $i(s)=0$.
\end{lemma}
\begin{proof}
From its definition, $\chi_s$ takes a finite number of values, say indexed by a finite group $\mu$. Since it is a constructible function, there is a $b$-definable partition of $\widetilde{\wM}^{\square}_{G,b}$ into parts $(X_g)_{g\in \mu}$ such that on $X_g$, the function $\chi_s$ is equal to $g$. The statement of the lemma is then equivalent to the truth of a first-order sentence $\phi(b)$ with parameter $b$. We then need to show that $\phi(b)$ holds for generic $b$. By Propositions \ref{prop:spec-for} and \ref{prop:spec:geostab}, the specialization $\phi_L(b)$ of $\phi(b)$ to a local field $L$ of large enough residue characteristic is equivalent to the equivalent lemma for an instance of the 
Geometric Stabilization theorem over a reductive group over $L$. For generic $b$, the fact that $\phi_L(b)$ holds is the key ingredient in the proof of Lemma 6.14 in \cite{GWZ18}. By Proposition \ref{prop:spec-for}, we conclude that $\phi(b)$ holds for generic $b$. 
\end{proof}

\begin{lemma}
\label{lem:chis:neron}
Assume $i(t)=0$. Identify $\widetilde{\wM}^{\square}_{G,b}$ with $\widetilde{\wP}^\square_{G,b}$ by means of the Kostant section. The function $\chi_s\circ e$ is a motivic character on $\widetilde{\wP}^\square_{G,b}$ and factors through the group of connected components of the N\'eron model of $\widetilde{\wP}^\square_{G,b}$.
\end{lemma}
\begin{proof}
Similarly to Lemma \ref{lem:chis:triv}, the multiplicativity of $\chi_s\circ e$  is implied by the multiplicativity of its specialization to local fields, where it follows from a reinterpretation  $\chi_s\circ e$ in terms of Tate duality, see  \cite[Section 6.5]{GWZ18}. Since the N\'eron model is not in general definable in families, we cannot readily deduce the factorisation property. However, using multiplicativity, it suffices to show that $\chi_s\circ e$ is constant to 1 on the neutral component of the N\'eron model. Assume not, let $x$ in the neutral component of the N\'eron model be such that $\chi_s\circ e(x)\neq 1$. Since elements in the neutral component of a N\'eron model are divisible, up to replacing $K'$ by a finite extension $K'_N$, we can assume that there exists $y\in \widetilde{\FP}^\square_{G,b}(K'_N\llp t \rrp) $ such that $y^N=x$, where $N$ is the order of the image of $\chi_s\circ e$, and the extension $K'_N$ is of degree depending only on $N$. Since this extension is pseudo-finite, the function $\chi_s\circ e$ is defined as well on this extension, multiplicative, and compatible with the restriction to $K'$. By compatibility of the N\'eron model with unramified base change, $x$ lies in the neutral component of the N\'eron model of $\widetilde{\FP}^\square_{G,b}(K'_N\llp t \rrp)$. Hence we get a contradiction, since if $\chi_s\circ e(x)\neq 1$, then $\chi_s\circ e(y)$ is of order more than $N$. 
\end{proof}

\begin{proposition}
\label{prop:calcint}
We have
\[ \int^\mot_{\widetilde{\wM}^{\square,t}_{G,b}} \chi_s\circ e \abs{\omega_{b}} = \begin{cases} 0 \ \ \ \text{ if } i(t) \neq 0 \text{ or } i(s) \neq 0  \\ \chi_{\psf,\rel}(\overline{\FN}(\widetilde{\wP}^\square_{G,b})\eL^{-\ord(\omega_b)}) \ \ \ \text{ if } i(t) = 0 \text{ and } i(s) = 0 .\end{cases} \]
\end{proposition}
Note that by Lemma \ref{lem:itzero}, the conditions $i(t) \neq 0$ and $i(s) \neq 0$  translate into definable conditions.

\begin{proof}
Assume first that $i(t) = 0$ and $i(s) = 0$. By Lemma \ref{lem:itzero} we have an isomorphism
\[\widetilde{\wM}^{\square,t}_{G,b}(K'\llp t\rrp) \cong \widetilde{\wP}^\square_{G,b}(K'\llp t\rrp).\]

By Lemma \ref{lem:chis:triv}, $\chi_s\circ e$ is constant on $ \widetilde{\wP}^\square_{G,b}(K'\llp t\rrp)$, hence necessarily constant to 1 since it is a group morphism. Hence we have
\[\int^\mot_{\widetilde{\wM}^{\square,t}_{G,b}} \chi_s\circ e \abs{\omega_{b}} = \int^\mot_{\widetilde{\wM}^{\square,t}_{G,b}} \abs{\omega_{b}} = \chi_{\psf,\rel}(\overline{\FN}(\widetilde{\wP}^\square_{G,b})\eL^{-\ord(\omega_b)}).\]
The last equation follows since the volume of an abelian variety with respect to any volume form equals the class of the special fiber of the N\'eron model up to a shift by the order of vanishing of the volume form along the special fiber, see for example \cite{Loeser-Sebag} or \cite[Remark 4.1.1]{LW19}.

For the other case, if $i(t)\neq 0$, then by Lemma \ref{lem:itzero} the integral is zero. Thus assume finally $i(t)=0$ and $i(s)\neq 0$.

By Lemma \ref{lem:chis:neron}, $\chi_s\circ e$ factors as 
\begin{equation}\label{nicefact} 
\widetilde{\FP}^\square_{G,b}(K'\llp t\rrp) \xrightarrow{r_{\widetilde{\FP}^\square_{G,b}}} \phi_{\widetilde{\FP}^\square_{G,b}}(K') \xrightarrow{\langle \cdot, i(s) \rangle} \BZ/n\BZ.   
\end{equation}

Each fiber of $r_{\widetilde{\FP}^\square_{G,b}}$
is a $\FN(\widetilde{\FP}^\square_{G,b})^0$-torsor over $K'\llp t\rrp$ and thus trivial, since $K'$ is pseudo-finite. Again as in the first case, the integral of $\chi_s\circ e$ over this fiber is equal to the constant value of $\chi_s\circ e$ multiplied by $\chi_{\psf,\rel}(\overline{\FN}^0(\widetilde{\wP}^\square_{G,b})\eL^{-\ord(\omega_b)})$.

We are given the integral of a constant function on a subset of an abelian variety defined by the preimage of a constructible subset of the special fiber of its N\'eron model. By smoothness of the N\'eron model, on each fiber, the integral is equal to the constant value of $\chi_s\circ e$ multiplied by $\chi_{\psf,\rel}(\overline{\FN}^0(\widetilde{\wP}^\square_{G,b})\eL^{-\ord(\omega_b)})$. Since the induced function on $\phi_{\widetilde{\FP}^{\square}_{G,b}}$ is non-trivial by Lemma \ref{lem:chis:triv}, the sum of all values is zero by Lemma~\ref{sumchar0}. This concludes the proof of the proposition.
\end{proof}

We can now finish the proof of Theorem \ref{thm:mainintfibers}. By Proposition \ref{prop:calcint}, if either $i(t)\neq 0$ or $i(s)\neq 0$, then both sides are zero hence the equality holds. In the case $i(t)=i(s)=0$, we need to prove that $\chi_{\psf,\rel}(\overline{\FN}(\widetilde{\wP}^\square_{G,b}))=\chi_{\psf,\rel}(\overline{\FN}(\widetilde{\wP}^\square_{\widehat G,b}))$. By the main result of \cite{Donagi:fk}, $\widetilde{\FP}^\square_{\widehat G,b}$ and ${\widetilde{\FP}^\square}_{G,b}$ are dual abelian varieties. Since the Grothendieck pairing on Néron models is non-degenerate in residue-characteristic $0$ \cite[11.3, Expos\'{e} IX]{GRR67} and $K'$ is pseudo-finite we deduce
\[ |\phi_{\widetilde{\FP}^{\square}_{G,b}}(K')| = |\phi_{{\widetilde{\FP}^{\square}}_{\widehat G,b}}(K')|,\]
as in the proof of \cite[Proposition 4.3]{Lo11}. Finally, the isogeny between $\widetilde{\FP}^\square_{G,b}$ and $\widetilde{\FP}^\square_{\widehat G,b}$ induces an isogeny between their N\'eron models by \cite[Proposition 7.3.6]{NeronModels}. The motives of two isogenous connected commutative groups are equal by \cite[Theorem 3.3 (4)]{AHP} and \cite[Lemma 2.5.1]{LW19}. The motive $\chi_{\psf,\rel}(\overline{\FN}(\widetilde{\wP}^\square_{G,b}))$ is $|\phi_{{\widetilde{\FP}^{\square}}_{G,b}}|$ times the motive of the neutral component, hence the motives $\chi_{\psf,\rel}(\overline{\FN}(\widetilde{\wP}^\square_{G,b}))$ and $\chi_{\psf,\rel}(\overline{\FN}(\widetilde{\wP}^\square_{\widehat G,b}))$  are equal, which concludes the proof of Theorem \ref{thm:mainintfibers}.

\section{A motivic Fundamental lemma}\label{sec:fl}
We show in this section how to deduce from the Geometric Stabilization theorem~\ref{th:geo:stab} a motivic version of the Fundamental Lemma. This part is geometric in the sense that it does not rely on any cohomological methods, and follows closely Ngô's original argument.

\subsection{Statement}\label{statement}
We briefly review \cite{chl}, in which it is shown that the orbital integrals of the Fundamental Lemma can be encoded by motivic integrals. This is similar to the encoding of the Hitchin fibration of Section \ref{sec:def:geo:stab}. 

Let $L$ be a discretely valued field with pseudo-finite residue field $K$. We encode unramified extensions of $L$ of fixed degree $r$ as definable sets in the Denef-Pas language using parameters similarly to what we have done in Section \ref{sec:galois:coh}. We introduce a parameter $b = (b_0, \dots, b_{r-1})$ in $L^r$.
The conditions that the $b_i$ belong to the valuation ring  of $L$ and that the reduction of the polynomial 
$m_b=x^r+ b_{r- 1} x^{r- 1}+ \dots b_0$ modulo the maximal ideal is an irreducible polynomial over the residue field are definable in the Denef-Pas language.
The unramified extension $L_r = L [x] /m_b$ of $L$ defined by this polynomial is then viewed as a definable set with parameter $b$ by identifying it with $L^r$. As in Section \ref{sec:galois:coh}, we also use a parameter $\tau$ for a generator of the Galois group $\Gal(L_r/L)$.

Recall that an unramified reductive group $G$ over a $L$ is a quasi-split reductive group over $L$ that splits over an unramified extension of $L$. 

We fix a split reductive group $\BG$ by fixing its root datum, as well as a faithful representation of $\BG$. We choose a quasi-split form of $\BG$ by fixing a element of $\rmH^1(\Gal(L_r/L),\mathrm{Out}(\BG))$, where $r$ is some fixed integer. The choice of $\tau$ identifies this torsor with an endomorphism $\theta$ of $\BG(L_r)$. Let $\wG$ be the definable set with parameters such that for every discretely valued field $L$ with pseudo-finite residue field, $\wG(L)$ is the set of fixed points of $\BG(L_r)$ under the endomorphism $\theta\circ \tau$. 

We have similarly definable sets $\wT$ for the maximal torus of $\wG$ determined by the root datum, $\wg$ for the  Lie algebra of $\wG$.

We also have the definable set of $\wG$-conjugacy classes in $\wg$, the Chevalley base $\wc$, together with a definable map $\chi\colon \wg \to \wc$. 

Fix some $\wG$-invariant and $\wT$-invariant definable volume forms $\abs{\omega_G}$ and $\abs{\omega_T}$.  We normalize integrals by choosing $\abs{\omega_G}$ such that $\int_{\wG(\mathcal{O})}\abs{\omega_G}=1$.

For $\gamma \in \wg$ regular semi-simple, define the motivic orbital integral as
\[
O_\gamma(\mathbf{1}_{\wg(\mathcal{O})})=\int_{\wT \backslash\wG}^\mot \mathbf{1}_{\wg(\mathcal{O})}(g^{-1}\gamma g)\frac{\abs{\omega_G}}{ \abs{\omega_T}}.
\]
We note that the integral depends on a choice of normalization for $\abs{\omega_T}$, but we do not include it explicitly in the notation. 

Let $I_\gamma$ the centralizer of $\gamma$. Let $\gamma'\in \wg$ which is stably conjugate to $\gamma$, that is, conjugate over $\wG(L_r)$. Using the choice of a generator $\tau$, we get a finite definable set $\mathcal{D}$ of representatives of stable conjugacy classes of $\gamma$. This $\gamma'$ determines an invariant $\mathrm{inv}(\gamma,\gamma')\in \rmH^1(L,I_\gamma)$, and this assignment is definable when we consider $\rmH^1(L,I_\gamma)$ as definable using cocycles as in \ref{sec:galois:coh}. In particular we may identify $\mathcal{D}$ with a subgroup of $\rmH^1(L,I_\gamma)$.

Let $\kappa$ be a character of $\rmH^1(L,I_\gamma)$.  The Tate-Nakayama pairing $\langle\mathrm{inv}(\gamma,\gamma'),\kappa \rangle$ is a definable map by \cite{chl}, in the sense that its image lies in $\ZZ/n\ZZ$ for some $n$, and the level sets (for fixed $\kappa$ and varying $\gamma, \gamma'$) are definable sets. Below we consider its motivic version, as in Section \ref{sec:motchar}.

Define the $\kappa$-orbital integral as
\[
O_\gamma^\kappa(\mathbf{1}_{\wg(\mathcal{O})})=\sum_{\gamma'\in \mathcal{D}}\langle\mathrm{inv}(\gamma,\gamma'),\kappa \rangle O_{\gamma'}(\mathbf{1}_{\wg(\mathcal{O})})
.\]
When $\kappa$ is the trivial character, we write $O_\gamma^{\mathrm{stab}}(\mathbf{1}_{\wg(\mathcal{O})})$ instead of $O_\gamma^\kappa(\mathbf{1}_{\wg(\mathcal{O})})$.

The Kostant section $\varepsilon\colon  \wc \to  \wg$ is definable as well, and we use it to define for $a\in \wc$ regular semi-simple
\[O_a^\kappa(\mathbf{1}_{\wg(\mathcal{O})})=O_{\varepsilon(a)}^\kappa(\mathbf{1}_{\wg(\mathcal{O})}).\]

Let $H$ be an endoscopic group of $G$, $\wH$ the associated definable group and $\wh$ the Lie algebra. We add subscripts $H$ to the various objects associated to $H$, such as $\wc_H$. We have an inclusion $\wc_H \to \wc$. Let $a_H\in \wc_H(L)$ be regular semi-simple, and $a$  its image in $\wc(L)$.  The group $H$ determines a character $\kappa$ of $\rmH^1(L,J_a)$, where $J_a \cong I_{\varepsilon(a)}$ denotes the regular centralizer of $a$, see \cite[Section 1.4]{ngo:fl}. Notice that $J_a(L) = T(L)$ by definition. 

We can now state the motivic version of the Fundamental Lemma. 

\begin{theorem}
\label{th:fl}
Let $G,H,a_H, a,\kappa$ be as above. Then there exists some $A_a$, motive of a commutative algebraic group over $\QQ(a)$ and the parameters, such that
\[
A_a \, O_a^\kappa(\mathbf{1}_{\wg(\mathcal{O})})=A_a\, \eL^{r^G_H(a_H)} \, O_{a_H}^{\mathrm{stab}}(\mathbf{1}_{\wh(\mathcal{O})} )
\]
in $\K{\DAC(\QQ(a),\rel,\Lambda)}\otimes \BQ$.
\end{theorem}

\begin{rmk} Even though Theorem \ref{th:fl} is a local statement, the proof uses the global geometry of the Hitchin system. This will imply that $A_a$ will in general contain a factor of an abelian variety coming from the factor $\chi_\psf((\BP'_a)^0)$ in the notation of Section \ref{lsec} below. Furthermore since the dimensions of the Hitchin system depends on $a$, we do not expect $A_a$ to behave as a motivic constructible function for varying $a$ in $\wc_H^{\mathrm{rss}}$. 

If one specializes to local fields of large residue characteristic, as in Corollary \ref{corfl} below, one can cancel the factor $A_a$ on both sides. In particular $A_a$ is not a zero-divisor of $O_a^\kappa(\mathbf{1}_{\wg(\mathcal{O})})- \eL^{r^G_H(a_H)} \, O_{a_H}^{\mathrm{stab}}(\mathbf{1}_{\wh(\mathcal{O})} )$.
\end{rmk}

We recover a variant of the result of Cluckers-Hales-Loeser \cite[Section 9.2]{chl}.
\begin{proposition}
\label{prop:spec-fl}
Let $\BG$ be a split reductive group over a normal domain $R$ of finite type over $\BZ$ and $r$ an integer. There is a non-empty open subscheme of $\Spec(R)$ such that for every closed point $x$ of $U$, and local field $L$ with a map $R\llb t\rrb\to L$ sending $t$ to a uniformizer of $L$ and residue field $\BF_x$, the following holds. 

Denote by $\FO$ the valuation ring of $L$. Let $G$ be a quasi-split reductive group over $L$, with split form $\BG$, Lie algebra $\g$ and Chevalley base $\mfc$. Assume that $G$ splits over an unramified extension of degree $r$. Let $H$ be an endoscopic group of $G$ and $\wH$ a definable endoscopic group of $\wG$ that specializes to $H$. For $a\in \mfc_H^{\mathrm{rss}}$ regular semi-simple, let $\kappa$ the character of $\rmH^1(L,J_a)$ associated to it. Then up to some choice of parameters, the motivic constructible function $O_a^\kappa(\mathbf{1}_{\wg(\mathcal{O})})$ considered in $\FC_\mot(\wc_H^{\mathrm{rss}})$, when specialized via $\Tr\Frob_x$ to a function on $\mfc_H^{\mathrm{rss}}$, is the function
\[
a\in \mfc_H^{\mathrm{rss}}\longmapsto O_a^\kappa(\mathbf{1}_{\g(\mathcal{O})}).
\]
\end{proposition}
\begin{proof}
By construction of the motivic orbital integral, the locus of integration  and the integrand specialize to their counterpart on $L$. Hence the result follows from Proposition \ref{prop:spec-int}.
\end{proof}

Recall from Ngô \cite[Th\'eor\`eme 1]{ngo:fl} the statement of the Fundamental Lemma of Langlands-Shelstad for Lie algebras.
\begin{corollary}\label{corfl}
The Fundamental Lemma \ref{FUL} holds for quasi-split reductive groups of large enough residue characteristic.
\end{corollary}
\begin{proof}
We just have to combine Theorem \ref{th:fl} with Proposition \ref{prop:spec-fl}, noting that every $a\in \mfc_H^{\mathrm{rss}}$ is the specialization of some $b\in \wc_H^{\mathrm{rss}}$. Dividing both sides of the specialization to a local field of the equality in Theorem \ref{th:fl} by $\Tr\Frob_x (A_b)$, which is non-zero as the number of points of an algebraic group, yields the equality of the Fundamental Lemma.
\end{proof}

\subsection{From orbital integrals to affine Springer fibers}

The relation between orbital integrals and affine Springer fibers first appeared in \cite{GKM04}. For any discrete and cocompact subgroup $\wLam \subseteq \wT$ defined by a finite set of generators, we deduce by Fubini
that
\begin{multline*}
\vol(\wLam \backslash \wT,\omega_T) \, O_a(\mathbf{1}_{\wg(\mathcal{O})}) =\vol(\wLam \backslash \wT,\omega_T) \int_{\wT \backslash \wG}^\mot \mathbf{1}_{\wg(\mathcal{O})}(g^{-1}\varepsilon(a)g)\frac{\abs{\omega_G}}{ \abs{\omega_T}} \\
=  \int_{\wLam \backslash\wG}^\mot \mathbf{1}_{\wg(\mathcal{O})}(g^{-1}\varepsilon(a)g)\abs{\omega_G}.
\end{multline*}
Furthermore, there is a $\wG(\Oc)$-fibration
\[ \{g \in \wLam \backslash \wG \ |\ g^{-1} \varepsilon(a) g \in  \wg(\mathcal{O})    \}  \longrightarrow  \{g \in \wLam \backslash \wG / \wG(\Oc) \ |\ g^{-1} \varepsilon(a) g \in  \wg(\mathcal{O})    \}.   \]
Here the right hand side can be identified with the $K(a)$-points of $\wLam \backslash \FN_a$, where $\FN_a$ denotes the affine Springer fiber associated with $a$. Since it is not of finite type, $\FN_a$ is not definable. However, by \cite{KL88}, or more precisely \cite[Proposition 3.4.1]{ngo:fl}, the quotient $\wLam \backslash \FN_a$ is a projective finite type scheme whose points can be encoded by a definable set. We thus get 
\begin{equation}\label{untint}\vol(\wLam \backslash \wT,\omega_T) \, O_a(\mathbf{1}_{\wg(\mathcal{O})}) = \chi_\psf(\wLam \backslash \FN_a). \end{equation}

Now let $\gamma'\in  \wg$ be stably conjugate to $\varepsilon(a)$ and let $\delta = \inv(\varepsilon(a),\gamma') \in \rmH^1(L,J_a)$ denote the corresponding class. By \cite[Lemme 8.2.4]{ngo:fl} we have an isomorphism
\[ \rmH^1(L,J_a) \cong \rmH^1(K,\BP(J'_a)), \]
where $J'_a \to J_a$  is a group scheme with connected fibers generically isomorphic to $J_a$ and $\BP(J'_a)$ denotes the Picard-stack of $J'_a$-torsors with a generic trivialization. Since $J'_a$ is connected, the $K$-points of  $\BP(J'_a)$ can be identified with $J_a(L)/J'_a(\FO_L)$, see \cite[Section 3.3]{ngo:fl}. The lattice $\wLam$ thus naturally embeds into $\BP(J'_a)$ and the quotient $\wLam \backslash \BP(J'_a)$ is an affine finite type group scheme \cite[Lemme 3.8.1]{ngo:fl}. We write $\wLam \backslash \BP(\wJ'_a)$ for the associated definable set and $\pi_0(a)$ for its group of connected components. Since $\wLam$ is discrete and torsion free we have by Lang's theorem an inclusion 
\[ \rmH^1(K,\BP(J'_a)) \subseteq \rmH^1(K,\wLam \backslash \BP(J'_a)) = \rmH^1(K,\pi_0(a)).\]

 Similarly to the case of Higgs bundles $\BP(J_a)$ and thus also $\BP(J'_a)$  act naturally on $\FN_a$ and it follows from the discussion in \cite[Section 15.5]{GKM04} that we have an isomorphism
\[ \wLam \backslash \FN_{\gamma'} \cong (\wLam \backslash \FN_a)^\delta,   \]
where $(\wLam \backslash \FN_a)^\delta$ denotes the twisted variety as in Section \ref{sec:twisted}. Since \eqref{untint} also holds with $\varepsilon(a)$ replaced by $\gamma'$ we finally get, for every $\kappa$,
\begin{multline}\label{kapspi} \vol(\wLam \backslash \wT,\omega_T) \, O^\kappa_a(\mathbf{1}_{\wg(\mathcal{O})}) = \sum_{\gamma' \in \mathcal{D}} \langle \mathrm{inv}(\varepsilon(a),\gamma'),\kappa \rangle \chi_\psf((\wLam \backslash \FN_a)^\delta)  \\  = 
\sum_{\delta \in \rmH^1(K,\BP(J^0_a))} \langle\delta,\kappa \rangle \chi_\psf((\wLam \backslash \FN_a)^\delta) 
 = |\rmH^1(K,\pi_0(a))| \  \chi_\psf(\wLam \backslash \FN_a,\kappa),  \end{multline}
where the $\kappa$-isotypical component is defined as in Section \ref{sec:isocomp} and $\mathcal{D}$ is identified with a subgroup of $ \rmH^1(L,J_a) \cong \rmH^1(K,\BP(J^0_a))$ through $\inv(\varepsilon(a),\cdot)$ as in Section \ref{statement}. Here we also used \cite[Lemme 8.2.6]{ngo:fl}, which implies that $(\wLam \backslash \FN_a)^\delta$ has no $K$-rational point for $\delta \notin \mathcal{D}$ and thus  $\chi_\psf((\wLam \backslash \FN_a)^\delta)= 0$ in this case.

We can further simplify \eqref{kapspi} as follows. First, again by Fubini, we have 
\begin{multline*} \vol(\wLam \backslash \wT,\omega_T) = \vol(\wJ'_a(\FO),\omega_T) \chi_\psf( \wLam \backslash \BP(\wJ'_a)) \\ =\vol(\wJ'_a(\FO),\omega_T) |\pi_0(a)|  \chi_\psf(\BP(\wJ'_a)^0).    
\end{multline*}

Secondly since $\Gal(K) = \widehat{\BZ}$, say with topological generator $\sigma$, we have an exact sequence
\[  0 \longrightarrow \pi_0(a)(K) \longrightarrow \pi_0(a)(K^\mathrm{alg}) \xrightarrow{\: 1-\sigma\:} \pi_0(a)(K^\mathrm{alg}) \longrightarrow  \rmH^1(K,\pi_0(a)) \longrightarrow 0,  \]
and thus $|\pi_0(a)(K)| = |\rmH^1(K,\pi_0(a))|$. Plugging this into \eqref{kapspi} and dividing by $|\pi_0(a)(K)| = |\rmH^1(K,\pi_0(a))|$ we get
\begin{equation}\label{sprorb} \vol(\wJ'_a(\FO),\omega_T) \chi_\psf\left( \BP(\wJ'_a)^0\right) O^\kappa_a(\mathbf{1}_{\wg(\mathcal{O})}) =  \chi_\psf\left(\wLam \backslash \FN_a ,\kappa\right), \end{equation}
which is the motivic analogue of \cite[Proposition 8.2.5]{ngo:fl}. Notice the equality $\vol(\wJ'_a(\FO),\omega_T) = \eL^\alpha \chi_\psf((\wJ'_a)_K)$, with $\alpha$ depending on the choice of $\omega_T$, in particular $\vol(\wJ'_a(\FO),\omega_T) $ is the motive of a commutative algebraic group.

\subsection{From affine Springer to Hitchin fibers}\label{lsec}
The connection between Sprin\-ger and Hitchin fibers is given by the product formula \cite[Proposition 4.15.1]{ngo:fl}. Consider the geometric setup of $G$-Higgs bundles on a curve $X$ in the generality of Section \ref{higgsb}. Let $a \in \A^{\ani}$ be a point with pseudo-finite residue field. By definition we get a morphism $a: X \to \mfc_D$ and we denote by $U \subseteq X$ the open dense preimage of $\mfc_D^{rs}$ under $a$. For any closed point $\nu \in X \setminus U$ the completion of the local ring at $\nu$ is a discretely valued field $L_\nu$ with pseudo-finite residue field $K_\nu$. We write a subscript $\nu$ for the restriction of $G$, $T$, $a$ etc. to this formal neighbourhood after having fixed a trivialization of $D$. In particular $a_\nu \in \mfc_\nu^{rs}(L_\nu)$ gives rise to an affine Springer fiber $\FN_{a_\nu}$.

Using the Kostant section one can define a morphism
\[ \prod_{\nu \in X \setminus U} \FN_{a_\nu} \longrightarrow \BM_a,  \]
and similarly for the symmetry groups $\prod_{\nu \in X \setminus U} \BP(J_{a_\nu}) \rightarrow \BP_{a}$. The product formula \cite[Proposition 4.15.1]{ngo:fl}, see also \cite[Proposition 8.4.2]{ngo:fl}, states that over $\overline{K}$ we have a homeomorphism of proper DM-stacks
\begin{equation}\label{cstqu}    \prod_{\nu \in X \setminus U} \FN_{a_\nu}  \times^{\prod_{\nu \in X \setminus U} \BP(J_{a_\nu})}  \BP_{a} \longrightarrow \BM_a, \end{equation}
compatible with the actions of $\Gal(\overline{K}/K)$ and  $\BP_{a}$. Furthermore let $\BP'_a \to  \BP(J_a)$ be as in \cite[Section 8.4]{ngo:fl}, in particular $\BP'_a$ is an algebraic group locally of finite type. It follows essentially by definition that \eqref{cstqu} continues to hold with  $ \BP(J_{a_\nu})$ and $\BP_a$  replaced by $\BP(J'_{a_\nu})$ and $\BP'_a$, or more conveniently for us, we have a fibration
\[\prod_{\nu \in X \setminus U} \FN_{a_\nu}  \times  \BP'_{a} \longrightarrow \BM_a, \]
whose fibers are $\prod_{\nu \in X \setminus U} \BP(J'_{a_\nu})$-torsors.

By an argument similar to the one in the previous section we obtain for each Galois-invariant character $\kappa$ of $\pi_0(\BP'_a)$ the relation
\begin{equation}\label{hispri}  \chi_\psf((\wP'_{a})^0) \prod_{\nu \in X\setminus U} \chi_\psf(\wLam_\nu \backslash \FN_{a_\nu},\kappa)=   \chi_\psf(\wM_{a},\kappa) \prod_{\nu \in X\setminus U} \chi_\psf(\BP(\wJ'_{a_\nu})^0).\end{equation}
 The formula is the motivic analogue of \cite[Corollaire 8.4.4]{ngo:fl}.  

\subsection{Back to orbital integrals}

It is explained in \cite[Section 8.6]{ngo:fl} how to construct from the initial data $G,H,a_H, a,\kappa$ of Theorem \ref{th:fl} a curve $X$ over $K$, a line bundle $D$ on $X$, a distinguished $K$-point $\nu_0$ and all the auxiliary data to construct the Hitchin fibrations for $G$ and $H$. Furthermore, by choosing the degree of $D$ sufficiently large, the affine linear sub-space $\mathrm Z \subseteq \A_H$, consisting of elements $a' \in \A_H$ such that $a'_{\nu_0} \equiv a_H \mod (T^N)$ for $N$ as in \cite[Proposition 3.5.1]{ngo:fl}, is non-empty and its pull-back $\widetilde{\mathrm Z} \subseteq \widetilde{\A}_H$ is geometrically irreducible by the same argument as in \cite[Lemma 5.3.2]{ngo:fl}. For dimension reasons also $\widetilde{\mathrm Z} \cap \widetilde{\A}^{\ani}_H$ is non-empty and geometrically irreducible. Since $K$ is pseudo-finite we deduce $\widetilde{\mathrm Z}\cap  \widetilde{\A}^{\ani}_H(K) \neq \emptyset$. Thus we may find $a'_H \in \A_H^{\ani}(K)$ which admits a lift to  $  \widetilde{\A}^{\ani}_H$ and such that 
\[ \chi_\psf(\wLam_{\nu_0} \backslash \FN_{a'_{\nu_0}},\kappa) \ \ \  \text{ and } \ \ \ \chi_\psf(\wLam_{\nu_0} \backslash \FN_{a'_{H,\nu_0}},\Stab)\]
compute the orbital integrals appearing in \ref{th:fl} via \eqref{sprorb}. Here we wrote $a'$ for the image of $a'_H$ in $\A^{\ani}$.  

Since both $a'_H$ and $a'$ admit lifts to $  \widetilde{\A}^{\ani}_H$ and $\widetilde{\A}^{\ani}$ we may identify the corresponding fibers of $\wM_H$ and $\wM$ with the ones of $\widetilde{\wM}_H$ and $\widetilde{\wM}$. The Geometric Stabilization theorem \ref{th:geo:stab:a} for $G$ and $H$ together with formula \eqref{hispri} then gives

\begin{align}\label{fff} \chi_\psf((\wP'_{a'})^0) \prod_{\nu \in X\setminus U} \chi_\psf(&\wLam_\nu \backslash \FN_{a'_\nu},\kappa)=\\  \nonumber
&\chi_\psf((\wP'_{a'})^0)\BL^{r_H^{G}(D)} \prod_{\nu \in X\setminus U} \chi_\psf(\wLam_\nu \backslash \FN_{a'_{H,\nu}},\Stab).   \end{align}

One can further adjust the choice of $a'_H$ as in \cite[Section 8.6.6]{ngo:fl} in such a way, that for all $\nu \in X \setminus U$ different from $\nu_0$ the Springer fibers are of a very simple form, which is analyzed in \cite[Lemme 8.5.7]{ngo:fl}. In the first two cases in \textit{loc. cit.} both Springer-fibers are $0$-dimensional, $r^G_H(a_{H,\nu})=0$ and
\[ \chi_\psf(\wLam_{\nu} \backslash \FN_{a'_{\nu}},\kappa) = \chi_\psf(\wLam_{\nu} \backslash \FN_{a'_{H,{\nu}}},\Stab) = 1.  \]

In the third case we still have $\chi_\psf(\wLam_{\nu} \backslash \FN_{a'_{H,{\nu}}},\Stab) = 1$ but $r^G_H(a'_{H,\nu})=1$. The geometry of $\FN_{a'_{\nu}}$ is described in \cite[Section 8.3]{ngo:fl} and reduced to the case when $\FN_{a'_{\nu},\overline{K}}$ is an infinite chain of $\BP^1$'s. The Kostant section provides a $K_{\nu}$-rational point for $\FN_{a'_{\nu}}$ and hence $\wLam_{\nu} \backslash \FN_{a'_{\nu}}$ is isomorphic to the Weil restriction along $K_\nu/K$ of either a nodal $\BP^1$ with trivial $\pi_0(a'_{\nu})$-action or two $\BP^1$'s meeting in two points and $\pi_0(a'_{\nu}) \cong \BZ/2\BZ$ exchanging components and intersection points, see \cite[Example 8.1.15]{ngo:fl}. It follows that in both cases $\chi_\psf(\wLam_{\nu} \backslash \FN_{a'_{\nu}},\kappa) = \BL^{\deg(K_{\nu}/K)}$. Since $r_H^{G}(D) = \sum_{\nu \in X \setminus U}  \deg(K_\nu/K)r^G_H(a'_{H,\nu})$, the combination of \eqref{fff} with \eqref{sprorb} finishes the proof of Theorem \ref{th:fl}.

\subsection*{Acknowledgements} 
During the preparation of this work, A.\,F. was supported by the Swiss National Science Foundation Ambizione grant PZ00P2\_193354,  F.\,L. was partially supported by the Institut Universitaire de France and D.\,W. was supported by the Swiss National Science Foundation [No. 196960].
Part of the paper was finalized during a stay of the three authors at the CIRM with the support of the Research in Residence program and we would
like to thank the CIRM for providing such an ideal working environment. 
We gratefully thank  the referees for their careful reading and for their many comments that helped us to improve  the readability of the paper.

\bibliographystyle{abbrv}
\bibliography{master}
\end{document}